\documentclass[11pt,a4paper,oneside]{amsart}
\usepackage{graphicx}
\usepackage{color}
\usepackage{epsfig}
\usepackage[english]{babel}
\usepackage[utf8]{inputenc}
\usepackage{amsmath}
\usepackage{amsthm}
\usepackage{amsfonts}
\usepackage{graphicx}
\usepackage{bbm}
\usepackage{subfigure}
\usepackage{epstopdf}

\def\R{\mathbb{R}}
\def\N{\mathbb{N}}

\newtheorem{thm}{Theorem}

\newtheorem{lem}[thm]{Lemma}
\newtheorem{rmk}[thm]{Remark}
\newtheorem{deff}[thm]{Definition}

\newtheorem{prop}[thm]{Proposition}

\numberwithin{equation}{section}
\numberwithin{thm}{section}

\DeclareMathOperator*{\esssup}{ess\,sup}

\begin{document}

\title{Derivation of a Non-autonomous  Linear Boltzmann Equation from a Heterogeneous Rayleigh Gas}

\author{Karsten Matthies}
\address{Department of Mathematical Sciences, University of Bath, Bath BA2
7AY, United Kingdom} \email{K.Matthies@bath.ac.uk}

\author{George Stone}
\address{Department of Mathematical Sciences, University of Bath, Bath BA2
7AY, United Kingdom} \email{G.Stone@bath.ac.uk}

\begin{abstract}
A linear Boltzmann equation with non-autonomous collision operator is rigorously derived in the Boltzmann-Grad limit
for the deterministic dynamics of  a Rayleigh gas where a tagged particle is undergoing hard-sphere collisions with heterogeneously distributed background particles, which do not interact among each other. The validity of the linear Boltzmann equation holds for arbitrary long times under moderate assumptions on spatial  continuity and higher moments of the initial distributions of the tagged particle and the  heterogeneous, non-equilibrium distribution of the background. The empiric particle dynamics are compared to the Boltzmann dynamics using evolution systems  for Kolmogorov equations of associated probability measures on collision histories.
\end{abstract}

\maketitle

\section{Introduction}

Particles undergoing hard sphere collisions are a classical topic in dynamical systems theory, e.g. the long term behaviour are of high interest in ergodic theory. Kinetic equations provide a different kind of limit description. The times remain finite, but  the number $N$ of discrete particles undergoing collisions tends to infinity as the diameter $\varepsilon$ of individual spheres tends to zero. We are interested in a global continuum  description of the particle gas in the Boltzmann-Grad scaling of $N$ and $\varepsilon$, where the overall volume of the $N$ particles tends to zero, but expected number of collisions per particle remains finite in a time of order one. The primary example
is the Boltzmann equation  given by,
\begin{equation*}
\begin{cases}
 \partial_t f  + v\cdot \nabla _x  f  &  = Q(f,f), \\
\hfill f_{t=0} & = f_0,
\end{cases}
\end{equation*}
where $f=f_t(x,v)$ represents the distribution of the gas at position $x$ and velocity $v$ at time $t$, the operator $Q$ represents the effect of self-interaction amongst the particles and $f_0$ is some given initial distribution.
Instead of considering a system of a large number of identical hard spheres evolving via elastic collisions one can consider a single tagged or tracer particle evolving among a system of fluid scatterers or background particles. With such a model one then considers the linear Boltzmann equation, where the operator $Q$ now encodes the effect of the tagged particle interacting with the scatterers, rather than self interactions amongst the particles.

 We will deal with a variant of the hard sphere flow and show the validity of  a linear Boltzmann equation, i.e. we show that the  distribution of a particle of interest in the many particle flow is well approximated by solutions  to the appropriate Boltzmann equation. This is what we mean by the derivation of a continuum description. The first major work for the full Boltzmann equation was by Lanford \cite{lanford75} giving convergence of the one particle distribution function of a hard-sphere particle model to solutions of the Boltzmann equation for short times by using the Bogoliubov-Born-Green-Kirkwood-Yvon (BBGKY) hierarchy for the evolution of $k$-particle distribution functions for all $k \in \N$, see e.g. \cite{born46,cercignani94,uchiyama1988}. The convergence was valid for short times, a fraction of the average free flight time. This proof was simplified in \cite{ukai2001boltzmann} by employing Cauchy-Kowalevski arguments. Global in time convergence results were proved in \cite{illner86,pulv87,illner89} with the assumption of sufficiently large mean free paths. Gallagher, Saint-Raymond and Texier \cite{saintraymond13} continued this development by giving detailed convergence results for short times.  Major work on short-range potentials includes \cite{PSS14}. The main challenge remains to provide convergence for arbitrarily long times.

 The first derivation of a linear Boltzmann equation for a Rayleigh gas was given in \cite{bei80}.
   In \cite{bod15brown} Bodineau, Gallagher and Saint-Raymond were able to utilise the tools from \cite{saintraymond13} to prove convergence from a hard-sphere particle model to the linear Boltzmann equation for arbitrary long times in the case that the initial distribution of the background is near equilibrium with an explicit rate of convergence. They  used the linear Boltzmann equation as an intermediary step to prove convergence to Brownian motion of a tagged particle. In \cite{bod15}, a variant gave the  derivation of the Stokes-Fourier equation again using linear Boltzmann equations  as an intermediary  step. The books \cite{cerci88,cercignani94,spohn91} give an introduction to the BBGKY hierarchy  and its link to  the Boltzmann equation. The derivation of the Boltzmann equation from a system of particles interacting with long range potentials, where each particle effects every other particle regardless of their distance, has proved more difficult. A first result was given in \cite{DePu99}. One recent result was proved in the linear case via the BBGKY hierarchy with strong decay assumptions on the potential and for arbitrarily long times in \cite{ayi17}.

A different variant to show the validity of Boltzmann-type equations has been developed in a series of papers \cite{matthies10,matt12,matt16}. There we employ infinite dimensional dynamical system and semigroup techniques to study the evolution of the probability to see collision trees. This relates the distribution of the history of the particles up to a certain time to the distribution of the particles at a specific time.  While terms in Duhamel formulas providing solutions to the BBGKY hierarchy  can be interpreted as pseudo histories, this approach uses other solution techniques and has a clearer connection to typical particle behaviour.
So far we have been able to prove convergence for arbitrary times for various  simplified particle models based on a many particle hard-sphere flows, e.g. kinetic annihilation and the  dynamics for a tagged particle interacting with moving  background  particles, which do not change. The aims  of  this paper are to develop the semigroup approach  and provide an example for the rigorous derivation of a non-autonomous linear Boltzmann equation.

A tagged particle is undergoing collisions with $n$ background particles.
If the background particles are fixed and of infinite mass  then this is the Lorentz gas first introduced  in \cite{lor05}. An  autonomous  linear Boltzmann equation can be derived as a scaling limit from a Lorentz gas with randomly placed scatterers, see for example \cite{bold83,gall99,spohn78} and a large number of references found in \cite[Chap. 8]{spohn91}. The linear Boltzmann equation can however fail as a valid approximation if there are (non-random) periodic scatterers, see  e.g. \cite{golse08,marklof10}. A different limiting stochastic process for the periodic Lorentz gas  was derived in  \cite{mark11} from the Boltzmann-Grad limit.

We consider the closely related Rayleigh gas, where the background particles move and are no longer of infinite mass.  In \cite{lebo82} Lebowitz and Spohn proved the convergence of the distribution of the tagged particle to the linear Boltzmann equation for background  data at equilibrium for arbitrarily long times, see also  \cite{lebo78,lebo82b,bei80}. In their case the distribution of the background particles does not change in time, the equilibrium distribution remains invariant under the pure transport.

In this paper we derive a non-autonomous linear Boltzmann equation via the Boltzmann-Grad limit of a Rayleigh gas particle model, where one tagged particle evolves amongst a large number of  background particles, which do not interact with each other. In contrast to \cite{matt16} the initial distribution of the background particles is now spatially heterogeneous and away from the equilibrium distribution. The background $g$ then satisfies a transport equation
\[\partial_t g(x,v,t)+ v \cdot \partial_x g(x,v,t)=0 \quad g(x,v,0)=g_0(x,v),\]
with explicit solution $g(x,v,t)=g_0(x-vt,v)$, this $g$ will introduce the non-autonomous background in  the linear Boltzmann equation. We assume that at a collision between the tagged particle and a background particle there is a full hard sphere collision in which both particles change direction, we will show that this change in the background is not relevant for the limit.

The main result is theorem~\ref{na-thm-main}, which states that the distribution of the tagged particle evolving among $N$ background particles converges as $N$ tends to infinity to the solution of the non-autonomous linear Boltzmann equation. The convergence holds for arbitrarily large times and with moderate moment assumptions on the initial data.

We extend the methods used in \cite{matt12,matt16}. The idealised equation on collision trees is stated and semigroup methods are used to show that there exists a solution. Then it is shown that the distribution on collision trees  induced by the many particle dynamics solves an empirical equation, which is of similar form as the equation for  the idealised distribution. This leads in section~\ref{na-sec-conv} to the convergence of solutions of the empirical and idealised equations, which then implies the main theorem.

A major technical difference to our paper~\cite{matt16}  is in section~\ref{na-sec-id} on the idealised equation. The introduction of a spatial dependence on the initial distribution of the background creates non-autonomous equations, which require us to study evolution system results. There is no specific evolution semigroup result for us to refer to for positive solutions of the non-autonomous equations, so our problem is viewed in the framework of general evolution system theory, which creates a number of more technical requirements.
As the question of honesty of the semigroup solution of the non-autonomous linear Boltzmann equation is also more difficult than in the autonomous case, we were unable to directly verify honesty from existing results. Instead honesty of the solution is proven indirectly via the connection to the idealised equation.
The change in collisions, where a collision between the tagged particle and a background particle is now a full hard sphere collision, makes only a minimal difference on our proof.

\section{Model and Main Result} \label{na-sec-model}
We now give our Rayleigh gas particle model in detail. The model differs from the model in~\cite{matt16} in two ways: i) we no longer assume that the initial distribution of the background particles is spatially homogeneous and ii) now when the tagged particle collides with a background particle the collision is treated as a full hard sphere collision and so the background particle changes velocity rather than continuing with the same pre-collision velocity.

Let $U = [0,1]^3 \subset \mathbb{R}^3$, with periodic boundary conditions. Let $N \in \mathbb{N}$. One tagged particle evolves amongst $N$ background particles. The tagged particle has random initial position and velocity given by $f_0 \in L^1(U \times \mathbb{R}^3)$ and the $N$ background particles have random and independent initial position and velocity given by $g_0 \in L^1(U \times \mathbb{R}^3)$. The tagged particle and background particles are modelled as spheres with unit mass and diameter $\varepsilon >0$ given by the Boltzmann-Grad scaling, $N\varepsilon^2 =1$.

The tagged particle travels with constant velocity while it remains at least $\varepsilon$ away from all background particles. Each background particles travels with constant velocity while it remains at least $\varepsilon$ away from the tagged particle. Background particles do not effect each other and freely pass through each other.  When the position of the tagged particle comes within $\varepsilon$ of the position of a background particle both particles instantaneously change velocity as described by Newtonian hard-sphere collisions. We describe this process explicitly.

Let the position and velocity of the tagged particle at time $t\geq 0$ be denoted $(x(t),v(t))$ and for $1\leq j \leq N$, let the position and velocity of background particle $j$ at time $t$ be given  by $(x_j(t),v_j(t))$. Then for all $t \geq 0$
\[ \frac{\mathrm{d}x(t)}{\mathrm{d}t} = v(t) \textrm{ and } \frac{\mathrm{d}x_j(t)}{\mathrm{d}t} = v_j(t).   \]
If there exists a $1 \leq j \leq N$ such that $|x(0) - x_j(0)| \leq \varepsilon$ then we assume that the two particles pass through each other unaffected (indeed this is well defined since the velocities are only equal with probability zero).  That is, any initial overlap is ignored and not treated as a collision. Now let $t > 0$. If for all $1\leq j \leq N$, $|x(t)-x_j(t)| > \varepsilon$ then
\[ \frac{\mathrm{d}v(t)}{\mathrm{d}t} = 0 \textrm{ and } \frac{\mathrm{d}v_j(t)}{\mathrm{d}t} = 0.   \]
Else there exists a $1\leq j \leq N$ such that $|x(t)-x_j(t)| = \varepsilon$ and both particles experience an instantaneous collision at time $t$. We denote by $v(t^-)$ and $v_j(t^-)$ the velocity of the tagged particle and background particle $j$ instantaneously before the collision and define $v(t)$ and $v_j(t)$ to the velocity of the tagged particle and background particle $j$ instantaneously after the collision.  Define the collision parameter $\nu \in \mathbb{S}^2$ by
\[ \nu : = \frac{x(t)-x_j(t)}{|x(t)-x_j(t)|}. \]

\begin{figure}[h!]
\centering
\includegraphics[scale=1]{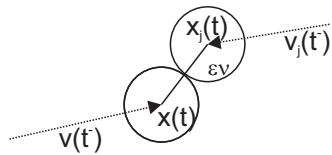}
\caption{The parameters of a collision between the tagged particle and particle  $j$.}
\label{na-fig-nu}
\end{figure}

Then $v(t)$ and $v_j(t)$ are given by
\[
v(t)  := v(t^-) - \nu \cdot (v(t^-) - v_j(t^-)) \cdot \nu;  \quad
v_j(t) := v_j(t^-) + \nu \cdot (v(t^-) - v_j(t^-)) \cdot \nu.
\]

\begin{prop}
For $N \in \mathbb{N}$ and $T>0$ fixed these dynamics are well defined up to time $T$ for all initial configurations apart from a set of zero measure.
\end{prop}
\begin{proof}
The proof of this is unchanged from~\cite[prop.1]{matt16}, which is based upon \cite[prop. 4.1.1]{saintraymond13}.
\end{proof}
\begin{deff}
For $t \geq 0$ and $N \in \mathbb{N}$ let $\hat{f}_t^N$ denote the distribution of the tagged particle at time $t$ evolving via the Rayleigh gas dynamics described above amongst $N$ background particles.
\end{deff}
We are interested in the behaviour of $\hat{f}_t^N$ as $N$ increases to infinity, or equivalently as $\varepsilon$ converges to zero. In the main theorem~\ref{na-thm-main} we show that for any fixed $T>0$ and under some assumptions on $f_0$ and $g_0$, $\hat{f}_t^N$ converges to $f_t^0$, the solution of the non-autonomous linear Boltzmann equation, in $L^1$ as $N$ tends to infinity uniformly for any $t \in [0,T]$.
\begin{deff} \label{na-deff-admiss} Let $f_0,g_0 \in L^1(U\times \mathbb{R}^3)$ be probability densities.
Then $f_0$ is said to be tagged-admissible if
\begin{equation}
\label{na-eq-f0l2assmp}
	\int_{ U \times \mathbb{R}^3} f_0 (x,v)(1+|v|^2) \, \mathrm{d}x \, \mathrm{d}v =:M_f  < \infty.
\end{equation}
Define $\bar{g}:\mathbb{R}^3 \to \mathbb{R}$ by
\begin{equation} \label{na-eq-gbar}
\bar{g}(v):= \esssup_{x \in U} g_0(x,v).
\end{equation}
Then $g_0$ is background-admissible if all of the following hold
	\begin{eqnarray}
    \label{na-eq-gl1assmp}
	\int_{ \mathbb{R}^3} \bar{g} (v)(1+|v|^2) \, \, \mathrm{d}v =:M_g < \infty,\\
	\label{na-eq-glinf}
	\esssup_{v\in\mathbb{R}^3} \bar{g}(v)(1+|v|) =:M_\infty< \infty,
	\end{eqnarray}
for almost all $v \in \mathbb{R}^3, \, g_0(\cdot,v) \in W^{1,1}(U)$ and
\begin{equation} \label{na-eq-g0w1}
\esssup_{x\in U} \int_{ \mathbb{R}^3} |\partial_x g_0(x,v)|(1+|v|) \, \, \mathrm{d}v  =:M_1 < \infty.
\end{equation}
and there exists a $M>0$ and an $0<\alpha \leq 1$ such that for almost all $v\in\mathbb{R}^3$ and for any $x,y \in U$
\begin{equation}
\label{na-eq-ghldassmp}
	|g_0(x,v)-g_0(y,v)| < M|x-y|^\alpha.
\end{equation}
\end{deff}

We now state the relevant non-autonomous linear Boltzmann equation. Firstly for $t \geq 0 $ define the operators $Q_t^{0,+}$ and $Q_t^{0,-} : L^1(U \times \mathbb{R}^3) \to L^1(U \times \mathbb{R}^3)$ by
\begin{align} \label{na-eq-Q+deff}
Q_t^{0,+}[f](x,v): =& \int_{\mathbb{S}^2} \int_{\mathbb{R}^3} f(x,v')g_t(x,\bar{v}')[(v-\bar{v})\cdot \nu]_+ \, \mathrm{d}\bar{v} \, \mathrm{d}\nu,\\
\label{na-eq-Q-deff}
Q_t^{0,-}[f](x,v): = &f(x,v)\int_{\mathbb{S}^2} \int_{\mathbb{R}^3}  g_t(x,\bar{v})[(v-\bar{v})\cdot \nu]_+ \, \mathrm{d}\bar{v} \, \mathrm{d}\nu,
\end{align}
where we use the notation $g_t(x,v):=g_0(x-tv,v)$ and where the pre-collision velocities, $v'$ and $\bar{v}'$, are given by $v' = v+ \nu \cdot (\bar{v}-v) \nu$ and $\bar{v}'=\bar{v} - \nu \cdot (\bar{v}-v)\nu$. Further define $Q_t^{0}:=Q_t^{0,+}-Q_t^{0,-}$. The non-autonomous linear Boltzmann equation is given by
\begin{equation} \label{na-eq-linboltz}
\begin{cases}
\partial_t f_t^0 (x,v) & = -v \cdot \partial_x f_t^0(x,v) + Q_t^0[f_t^0](x,v), \\
f_{t=0}^0(x,v)  & = f_0(x,v).
\end{cases}
\end{equation}

We now state the main theorem.
\begin{thm} \label{na-thm-main}
Let $0<T<\infty$ and suppose that $f_0$ and $g_0$ are tagged and background admissible probability densities respectively. Then, uniformly for $t \in [0,T]$,  $\hat{f}_t^N$, the distribution of the tagged particle at time $t$ among $N$ background particles under the above particle dynamics, converges in $L^1(U \times \mathbb{R}^3)$ as $N$ tends to infinity to $f_t^0$, a solution of the non-autonomous linear Boltzmann equation \eqref{na-eq-linboltz}.
\end{thm}

\subsection{Remarks}
\begin{enumerate}
\item We prove the result in dimension $3$. The result should also hold in the case $d=2$ or $d\geq 4$ up to a change in moment assumptions.
\item With stronger moment assumptions on the initial distributions $f_0$ and $g_0$ it may be possible to calculate explicit convergence rates and convergence in $L^1$-spaces involving moments. In particular to show \eqref{na-eq-ptepuni} we use the dominated convergence theorem which proves convergence without any explicit rate.  With further assumptions on our initial data it may be possible to prove this with a more quantitative method. Even under our  mild assumptions we can quantify corrections terms in Theorem \ref{na-thm-emp} explicitly. We will not express the explicit dependence on $T$ in the estimates, but all error estimates will grow in $T$, e.g. linearly in Lemma \ref{na-lem-rhobound}.
\item Our main extensions to~\cite{matt16} are that  methods developed here can deal with an evolving background. Hence  the methods will be relevant for more involved particle models where the background particles evolve,  such as the addition of an external force acting on the particles. In such a situation the relevant linear Boltzmann equation would include the additional force term. Also the distribution of the background particles at time $t$ would include the effects of this force. This would add additional complications to the various bounds computed throughout.
\item We hope that evolving backgrounds can be a route to approximate the behaviour of a full many particle flow, where the background is regularised by introducing appropriate counters. A collision happens between background particle $i$ and $j$ if both particles have experienced less than $k$ collisions. It remains open if the methods will be stable under letting $k$ tend to infinity and if eventually this leads to improvements of the time interval of validity of the nonlinear Boltzmann equation compared to \cite{lanford75}.
\item The conditions on $f_0$ and $g_0$ in \eqref{na-eq-f0l2assmp} and \eqref{na-eq-glinf} considerably relax the moment conditions compared to the exponential moments in the function spaces in \cite{bod15brown}, while the transport effects due the spatial heterogeneity can be well controlled using \eqref{na-eq-gl1assmp}. This allows us to use e.g. stretched exponentials or other non-Maxwellian  distributions in the background, which should be helpful, when attempting extensions as in the previous remark.
    \end{enumerate}

\subsection{Method of Proof and Propagation of Chaos}

 We consider two Kolmogorov equations on the set of all possible collision histories or `trees'. Section \ref{na-sec-id} is mostly devoted to proving theorem~\ref{na-thm-id}, where we prove that there exists a solution to the idealised equation by an iterative construction process and then prove that a number of properties hold, including the connection to the solution of the linear Boltzmann equation. In this section we introduce an $\varepsilon$ dependence in both the idealised equation and the linear Boltzmann equation to enable convergence proofs that follow later.

In section~\ref{na-sec-emp} we prove that the distribution of all possible collision histories from our particle dynamics solves the empirical equation, at least for well controlled situations, which resembles the idealised equation. We do this by explicitly calculating the rate of change of the distribution on all possible collision histories.

Finally in section~\ref{na-sec-conv} we prove the main  theorem~\ref{na-thm-main}, by proving the convergence between the solutions of the idealised and empirical equations.

A detailed comparison with the classical BBGKY approach of \cite{lanford75} adapted by \cite{saintraymond13} has been given in \cite{matt16}. A major challenge in that approach is proving the propagation of chaos. The tree history approach allows us to avoid the issue of proving the propagation of
chaos explicitly. This approach was developed in [35] to circumvent the issues around
the propagation of chaos by focusing on good histories or trees.
The idealised distribution $P^\varepsilon_t$ considers that the particles are chaotic so the probability
of seeing a background particle at $(x, v)$ at time $t$ is given exactly by $g_0(x-tv, v)$.
On the other hand for the empirical distribution no assumption of chaos is made and
the particles evolve as described by the particle dynamics. Therefore the probability of
seeing a background particle at $(x, v)$ at time $t$ is more involved than just $g_0(x-tv, v)$
since we need to consider the effect of a background particle colliding, changing velocity
and then arriving at $(x, v)$ at time $t$.
This issue is resolved by introducing  good collision trees. Good trees,
defined precisely in definition \ref{na-deff-goodtrees} below, require, among other properties, that each background
particle that the tagged particle collides with will not re-collide with the tagged
particle up to time $T$. This means that if we restrict our attention to good trees
then we know that there cannot be any re-collisions and so the distribution of the background
particles is much clearer. For this reason we only investigate the properties of
the empirical distribution $\hat P^\varepsilon_t$ on this set of good histories.
It is then shown in proposition \ref{na-prop-goodfull} below that good histories have full measure, in the
sense that the contribution of histories that are not good is vanishing as $\varepsilon$ tends to
zero.

Therefore to prove convergence between the idealised distribution and the empirical
distribution, which is the key step to proving the main theorem, we only need to
compare the idealised and empirical distributions on good histories and remark that
the effect of histories that are not good is vanishing in the limit.
Hence the propagation of chaos is proved implicitly with this collision history
method. The idealised distribution assumes chaos whereas the empirical distribution
does not. By proving the convergence from the empirical distribution to the idealised
distribution we prove the propagation of chaos implicitly.
We emphasise that good histories, due to their lack of re-collisions, mean that the
propagation of chaos holds for the particles relevant for the tagged particle.

\subsection{Tree Set Up}

Trees are defined in the same way as in~\cite{matt12,matt16}, where more detailed explanations can be found. We consider non-cyclic rooted trees of height at most one.
A tree represents a specific history of collisions. The nodes of the tree represent particles and are marked with information about that particle, while the edges of the trees represent collisions. The root of the tree represents the tagged particle and is marked with the tagged particle's initial position $(x_0,v_0) \in U \times \mathbb{R}^3$. The child nodes of the tree represent background particles that the tagged particle collides with and are marked with $(t_j,\nu_j,v_j) \in [0,T] \times \mathbb{S}^2 \times \mathbb{R}^3$ the time of collision, the collision parameter and the incoming velocity of the background particle before the collision respectively. The graph structure of the tree is of little significance in the current paper.
\begin{deff}\label{def:tree}

\begin{figure}[ht]
\includegraphics[scale=1]{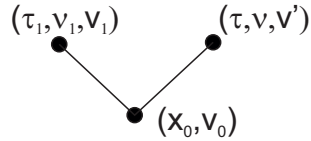}
\caption{A tree with two collisions, the time of the final collision is $\tau$.}
\end{figure}

The set of collision trees is defined by,
\[ \mathcal{MT} : = \{ (x_0,v_0),(t_1,\nu_1,v_1), \dots, (t_n,\nu_n,v_n) : n \in \mathbb{N}\cup \{0\}, t_1 < \dots < t_n \}. \]
For a tree $\Phi \in \mathcal{MT}$, $n$ denotes the number of collisions. The final collision plays an important role in this theory. We define $\tau = \tau(\Phi)$
\begin{equation}
\tau := \begin{cases} 0 & \textrm{ if } n=0, \\
t_n & \textrm{ if } n \geq 1,
\end{cases}
\end{equation}
and for $n \geq 1$ we use the notation $(\tau, \nu, v')= (t_n,\nu_n,v_n)$.
Finally, for $n \geq 1$, we define $\bar{\Phi}$ as the pruned tree of $\Phi$ with the final collision removed. For example if $\Phi=((x_0,v_0),(\tau,\nu,v'))$ then $\bar{\Phi} = ((x_0,v_0))$.

We define a metric, $d$, on $\mathcal{MT}$ as follows. For any  $\Phi,\Psi \in \mathcal{MT}$,
\begin{equation*}
d(\Phi,\Psi): = \begin{cases} 1, \qquad \qquad  \textrm{ if } n(\Phi) \neq n(\Psi)  \\
 \min \bigg\{1 , \max_{1 \leq j \leq n} |(t_j(\Phi), \nu_j(\Phi),v_j(\Phi)) - (t_j(\Psi), \nu_j(\Psi),v_j(\Psi)) | \bigg\}  \textrm{ else}.
\end{cases}
\end{equation*}
For $\Phi \in \mathcal{MT}$ and $h>0$ we define
\[ B_h(\Phi):= \left\{ \Psi \in \mathcal{MT} : d(\Phi,\Psi) < \frac{h}{2} \right\}.  \]
We obtain the Borel $\sigma$-algebra from the metric. All measures of interest will be absolutely continuous in each component of $\mathcal{MT}$ with respect the corresponding Lebesgue
measure for $U\times \R^3 \times ([0,T]\times \mathbb{S}^2 \times \R^3)^n$.
\end{deff}
We note that for a given $\varepsilon \geq 0$, the realisation of $\Phi$ at a time $t\in [0,T]$ uniquely determines $(x(t),v(t))$, the position and velocity of the tagged particle, and $(x_j(t),v_j(t))$, the position and velocity of the $j$ background particles involved in the tree. We note that $(x(t),v(t))$ is independent of $\varepsilon$ (since regardless of $\varepsilon$ the tagged particle has given velocities and collision times), but each $(x_j(t),v_j(t))$ is $\varepsilon$ dependent (since the relevant background particle must be $\varepsilon$ from the tagged particle at the collision).

Background particles might collide several times  with the tagged particle, in such case less than $n$ different particles are involved in a tree with $n(\Phi)$ collisions. Our parametrisation of the trees does not immediately identify such trees, we will later show that the resulting trees are rare.
Furthermore the realisation of $\Phi$ gives information on the remaining (at least) $N-n$ background particles, since we know that they have not interfered with the tree.

\section{The Idealised Distribution} \label{na-sec-id}
The idealised equation is the first of two Kolmogorov equations in this paper. In this section we show that there exists a solution to the idealised equation and relate it to the solution of the linear Boltzmann equation. We construct a solution by first considering the probability of finding the tagged particle at a certain position and velocity such that it has not yet had any collisions. From this we iteratively define a function and check that it solves the idealised equation and that the required connection to the linear Boltzmann equation holds.

A significant problem in this section is showing that we have the required evolution system to solve the non-autonomous equation that describes the probability of finding the tagged particle such that it has not yet experience any collisions. In the autonomous case  we were able to quote specific semigroup results for the Boltzmann equation from \cite{arlotti06}. However in this non-autonomous case we have to resort to more general evolution system theory. This leads to a number of technical results to check the various assumptions of the general theory.

In order to compare the solution of the idealised equation with the solution of the empirical equation, which is the main step in proving theorem~\ref{na-thm-main}, we consider an intermediate step by introducing a dependence on $\varepsilon$ in the idealised equation. In order to be able to connect this $\varepsilon$ dependent solution of the idealised equation to the linear Boltzmann equation we introduce an $\varepsilon$ dependent linear Boltzmann equation. Similarly to \eqref{na-eq-Q+deff} and \eqref{na-eq-Q-deff}, for $\varepsilon \geq0$, $t \geq 0 $ define $Q_t^{\varepsilon,+}$ and $Q_t^{\varepsilon,-} : L^1(U \times \mathbb{R}^3) \to L^1(U \times \mathbb{R}^3)$ by
\begin{align*}
Q_t^{\varepsilon,+}[f](x,v)&: = \int_{\mathbb{S}^2} \int_{\mathbb{R}^3} f(x,v')g_t(x+\varepsilon\nu,\bar{v}')[(v-\bar{v})\cdot \nu]_+ \, \mathrm{d}\bar{v} \, \mathrm{d}\nu,
\\
Q_t^{\varepsilon,-}[f](x,v)&: = f(x,v)\int_{\mathbb{S}^2} \int_{\mathbb{R}^3}  g_t(x+\varepsilon\nu,\bar{v})[(v-\bar{v})\cdot \nu]_+ \, \mathrm{d}\bar{v} \, \mathrm{d}\nu.
\end{align*}
Define $Q_t^{\varepsilon}:=Q_t^{\varepsilon,+}-Q_t^{\varepsilon,-}$. Then the $\varepsilon$ dependent non-autonomous linear Boltzmann equation is given by
\begin{equation} \label{na-eq-linboltzep}
\begin{cases}
\partial_t f_t^\varepsilon (x,v) & = -v \cdot \partial_x f_t^\varepsilon(x,v) + Q_t^\varepsilon[f_t^\varepsilon](x,v), \\
f_{t=0}^\varepsilon(x,v)  & = f_0(x,v).
\end{cases}
\end{equation}

We can now state the idealised equation. For $\varepsilon\geq 0$ consider
\begin{equation} \label{na-eq-id}
\begin{cases}
\partial _t P_t (\Phi) & = \mathcal{Q}^\varepsilon_t[P_t](\Phi) = \mathcal{Q}_t^{\varepsilon,+} [P_t](\Phi) - \mathcal{Q}_t^{\varepsilon,-}[P_t](\Phi), \\ P_0(\Phi) & = f_0 (x_0,v_0)\mathbbm{1}_{n(\Phi)=0},
\end{cases}
\end{equation}
where the gain term depends on the pruned tree $\bar{\Phi}$ as given in definition \ref{def:tree}
\begin{align}
\mathcal{Q}_t^{\varepsilon,+} [P_t](\Phi)&: =
\begin{cases}
 \delta(t-\tau)P_
\tau (\bar{\Phi})g_\tau(x(\tau)+\varepsilon \nu,v')[(v(\tau^-)-v')\cdot \nu]_+ & \textrm{ if } n\geq 1, \\
  0 &  \textrm{ if } n =0,
\end{cases}\\
\mathcal{Q}_t^{\varepsilon,-}[P_t](\Phi)&: = P_t(\Phi) \int_{\mathbb{S}^{2}} \int_{\mathbb{R}^3} g_t(x(t)+\varepsilon\nu,\bar{v})[(v(t)-\bar{v})\cdot \nu ]_+ \, \mathrm{d}\bar{v} \, \mathrm{d}\nu.\end{align}
For a tree $\Phi\in \mathcal{MT}$, $t\geq0$ and $\varepsilon\geq0$ we introduce the notation
\begin{equation}\label{na-eq-Ldeff}
L_t^{\varepsilon}(\Phi):= \int_{\mathbb{S}^{2}} \int_{\mathbb{R}^3} g_t(x(t)+\varepsilon\nu,\bar{v})[(v(t)-\bar{v})\cdot \nu ]_+ \, \mathrm{d}\bar{v} \, \mathrm{d}\nu
\end{equation}
and note that this implies
\[ \mathcal{Q}_t^{\varepsilon,-}[P_t](\Phi) = P_t(\Phi)L_t^\varepsilon(\Phi). \]
Moreover for any $t\in[0,T]$ and for any $\Omega \subset U \times \mathbb{R}^3$ define,
\[ S_t(\Omega):=\{ \Phi \in \mathcal{MT} : (x(t),v(t)) \in \Omega \}. \]
\begin{thm} \label{na-thm-id}
Suppose that $f_0$ and $g_0$ are tagged and background admissible respectively in the sense of definition~\ref{na-deff-admiss}. Then for all $\varepsilon\geq0$ there exists a solution $P^\varepsilon:[0,T
] \to L^1(\mathcal{MT})$ to \eqref{na-eq-id} such that for all $t\in [0,T]$, $P_t^\varepsilon$ is a probability measure on $\mathcal{MT}$. Furthermore there exists a $K>0$, independent of $\varepsilon$ such that for any $t \in [0,T]$
\begin{equation} \label{na-eq-Pt1st}
 \int_{\mathcal{MT}} P_t^\varepsilon(\Phi)(1+|v(\tau)|) \, \mathrm{d} \Phi \leq K < \infty.
\end{equation}
And for any $\varepsilon \geq 0$, $t\in[0,T]$ and any $\Omega \subset U \times \mathbb{R}^3$ measurable,
\begin{equation} \label{na-eq-ftptcon}
 \int_\Omega f_t^\varepsilon(x,v) \, \mathrm{d}x \, \mathrm{d}v = \int_{S_t(\Omega)} P_t^\varepsilon(\Phi) \, \mathrm{d}\Phi,
\end{equation}
where $f_t^\varepsilon$ is a solution to \eqref{na-eq-linboltzep}.
Finally, uniformly for $t \in [0,T]$,
\begin{equation} \label{na-eq-ptepuni}
\lim_{\varepsilon \to 0}  \int_{\mathcal{MT}}  \left| P_t^0(\Phi) - P_t^\varepsilon(\Phi) \right| \, \mathrm{d}\Phi  =0.
\end{equation}

\end{thm}

From now we assume that $f_0$ and $g_0$ are tagged and background admissible respectively.

  The rest of this section is devoted to proving theorem~\ref{na-thm-id}. We split this into a number of subsections. In the first subsection, \ref{na-sec-idevol}, we prove that there exists a solution $P_t^{\varepsilon,(0)}$ to the gainless linear Boltzmann equation and that this solution has a particular form given by an evolution system $U^\varepsilon$. This subsection takes a number of technical lemmas in order to prove various semigroup properties. Then in subsection \ref{na-sec-idlinbolt} we show that the $\varepsilon$ dependent non-autonomous linear Boltzmann equation has a solution, in the evolution system sense. Then in section~\ref{na-sec-idbuild} we construct $P_t^\varepsilon$ and show that it indeed satisfies the properties of theorem~\ref{na-thm-id}. We finish this section by using theorem~\ref{na-thm-id} to prove that the solution of the $\varepsilon$ dependent non-autonomous linear Boltzmann equation is a probability measure.

\subsection{The Evolution Semigroup}\label{na-sec-idevol}
In this subsection we prove that there exists a solution to the $\varepsilon$ dependent gainless linear Boltzmann  equation \eqref{na-eq-p0} by following standard evolution system theory as in \cite{pazy83}. This requires a number of technical results.

\begin{deff} \label{na-deff-ABops}
For any $t \in [0,T]$ and any $\varepsilon\geq 0$ define $D(A^\varepsilon(t)),D(B^\varepsilon(t)) \subset L^1(U \times \mathbb{R}^3)$ by
\begin{align*}
 D(A^\varepsilon(t))&: = \{ f \in L^1(U \times \mathbb{R}^3) : v\cdot \partial_x f(x,v) + Q_t^{\varepsilon,-}[f](x,v) \in L^1(U \times \mathbb{R}^3) \}, \\
  D(B^\varepsilon(t))&: = \{ f \in L^1(U \times \mathbb{R}^3) :  Q_t^{\varepsilon,+}[f](x,v) \in L^1(U \times \mathbb{R}^3) \}.
\end{align*}
Then define operators $A^\varepsilon(t):D(A^\varepsilon(t)) \to L^1(U \times \mathbb{R}^3)$ and $B^\varepsilon(t):D(B^\varepsilon(t)) \to L^1(U \times \mathbb{R}^3)$ by
\begin{align}
(A^\varepsilon(t)f)(x,v) &:= - v\cdot \partial_x f(x,v) - Q_t^{\varepsilon,-}[f](x,v) \label{na-eq-Adeff} \\
(B^\varepsilon(t)f)(x,v) &:= Q_t^{\varepsilon,+}[f](x,v) \label{na-eq-Bdeff}.
\end{align}
\end{deff}

\begin{prop} \label{na-prop-Pj0}
For $\varepsilon\geq 0$ there exists a solution $P^{\varepsilon,(0)}:[0,T] \to L^1(U\times \mathbb{R}^3)$ to the following equation
\begin{equation} \label{na-eq-p0}
\begin{cases}
\partial_t P_t^{\varepsilon,(0)} (x,v) & = (A^\varepsilon (t)P^{\varepsilon,(0)}_t)(x,v), \\
 P_0^{\varepsilon,(0)} (x,v) & = f_0(x,v),
\end{cases}
\end{equation}
Moreover the solution is given by $P_t^{\varepsilon,(0)}=U^\varepsilon(t,0)f_0$, where $U^\varepsilon$ is defined by,
 $U^\varepsilon:[0,T]\times[0,T] \times L^1(U\times\mathbb{R}^3)\to L^1(U\times\mathbb{R}^3)$ defined by
\begin{align} \label{na-eq-Uepdeff}
(U^\varepsilon(t,s)f)(x,v) & := \exp \left(- \int_s^t  \int_{\mathbb{S}^{2}} \int_{\mathbb{R}^3} g_\sigma(x+\varepsilon\nu - (t-\sigma) v,\bar{v})[(v-\bar{v})\cdot \nu ]_+ \, \mathrm{d}\bar{v} \, \mathrm{d}\nu \, \mathrm{d}\sigma \right) \nonumber \\
 & \quad \quad  f(x-(t-s)v,v).
\end{align}
\end{prop}
\begin{rmk}
$P_t^{\varepsilon,(0)}(x,v)$ can be thought of as the probability of finding the tagged particle at $(x,v)$ such that it has not yet experienced any collisions.
\end{rmk}
To prove this proposition we aim to apply \cite[Theorem 5.3.1]{pazy83}, which gives that there exists a evolution system defining the solution to \eqref{na-eq-p0}. First we present lemmas checking that conditions $(H1),(H2)$ and $(H3)$ hold. This tells us that there exists a unique evolution system satisfying $(E1),(E2)$ and $(E3)$. Next we show that $U^\varepsilon(t,s)$ is a strongly continuous evolution system and that it satisfies $(E1),(E2)$ and $(E3)$, so is indeed the evolution system described by \cite[Theorem 5.3.1]{pazy83}. This tells us that a solution to \eqref{na-eq-p0} is given by $U^\varepsilon(t,0)f_0$.

\begin{lem}\label{na-lem-AH1}
For $\varepsilon \geq 0$, $A^\varepsilon$, as defined in definition~\ref{na-deff-ABops}, satisfies condition $(H1)$ of \cite[Chapter 5]{pazy83}.
\end{lem}
\begin{proof}
By \cite[Theorem 10.4]{arlott07} we see that for $t\geq 0$, $A^\varepsilon(t)$ generates the $C_0$ semigroup $S_t^\varepsilon$ given by
\begin{align} \label{na-eq-Ssemigroup}
(S_t^\varepsilon(s)  P)(x,v)  =\exp \left(- \int_0^s  \int_{\mathbb{S}^{2}} \int_{\mathbb{R}^3} g_t(x+\varepsilon\nu - \sigma v,\bar{v})[(v-\bar{v})\cdot \nu ]_+ \, \mathrm{d}\bar{v} \, \mathrm{d}\nu \, \mathrm{d}\sigma \right)P(x-sv,v).
\end{align}
Since each $S_t^\varepsilon$ is a contraction semigroup this is a stable family, which proves condition $(H1)$ of  \cite[theorem 5.3.1]{pazy83}.
\end{proof}
Define
\begin{equation} \label{na-eq-Ydeff}
Y :=  \left\{ P \in L^1(U \times \mathbb{R}^3) : \textrm{for almost all } v \in \mathbb{R}^3, P(\cdot,v) \in W^{1,1}(U) \textrm{ and } \|P\|_Y < \infty \right\} ,
\end{equation}
where
\begin{equation} \nonumber
\| P \|_Y := \int_{U \times \mathbb{R}^3} (1+|v|^2)|P(x,v)| + (1+|v|)|\partial_x P(x,v)| \, \mathrm{d}x \, \mathrm{d}v.
\end{equation}
The following two lemmas, lemma~\ref{na-lem-Yinvariant} and lemma~\ref{na-lem-AH2help}, are used to help prove that condition $(H2)$ holds, which is shown in lemma~\ref{na-lem-AH2}.
\begin{lem} \label{na-lem-Yinvariant}
For $\varepsilon, t,s\geq 0$, $Y$ is invariant under the map $S_t^\varepsilon(s)$.
\end{lem}
\begin{proof}
Let $P \in Y$. It is clear that for almost all $v \in \mathbb{R}^3$, $S_t^\varepsilon(s) P(\cdot,v) \in L^1(U)$. Further, for each $i=1,2,3$, and almost all $x,v \in U \times \mathbb{R}^3$
\begin{align} \label{na-eq-partialStep}
\partial_{x_i}  \Big( (S_t^\varepsilon(s) P)(x,v) \Big)  &= -  \left( \int_0^s  \int_{\mathbb{S}^{2}} \int_{\mathbb{R}^3} \partial_{x_i} g_t(x+\varepsilon\nu - \sigma v,\bar{v})[(v-\bar{v})\cdot \nu ]_+ \, \mathrm{d}\bar{v} \, \mathrm{d}\nu \, \mathrm{d}\sigma \right) (S_t^\varepsilon(s)P)(x,v) \nonumber \\
&   - \exp \left(- \int_0^s  \int_{\mathbb{S}^{2}} \int_{\mathbb{R}^3} g_t(x+\varepsilon\nu - \sigma v,\bar{v})[(v-\bar{v})\cdot \nu ]_+ \, \mathrm{d}\bar{v} \, \mathrm{d}\nu \, \mathrm{d}\sigma \right)\partial_{x_i} P(x-sv,v).
\end{align}
Since $P \in Y$ and using \eqref{na-eq-g0w1} we can integrate each of these terms over $U$.
Hence for almost all $v \in \mathbb{R}^3$, $(S_t^\varepsilon(s)P)(\cdot,v) \in W^{1,1}(U)$. It remains to check that $\| S_t^\varepsilon(s)P \|_Y <\infty$. By bounding the exponential term in \eqref{na-eq-Ssemigroup} by $1$ we have
\begin{equation} \label{na-eq-Ynormpart1}
\int_{U \times \mathbb{R}^3} (1+|v|^2)|S_t^\varepsilon(s)P(x,v)|  \, \mathrm{d}x \, \mathrm{d}v \leq  \int_{U \times \mathbb{R}^3} (1+|v|^2)|P(x,v)|  \, \mathrm{d}x \, \mathrm{d}v \leq \| P \|_Y < \infty.
\end{equation}
Further we note that by \eqref{na-eq-g0w1} for some $C>0$,
\begin{align}\label{na-eq-intpartialg}
&\int_{U \times \mathbb{R}^3}  (1+|v|) \left( \int_0^s  \int_{\mathbb{S}^{2}} \int_{\mathbb{R}^3} \partial_{x_i} g_t(x+\varepsilon\nu - \sigma v,\bar{v})[(v-\bar{v})\cdot \nu ]_+ \, \mathrm{d}\bar{v} \, \mathrm{d}\nu \, \mathrm{d}\sigma \right)  (S_t^\varepsilon(s)P)(x,v)  \, \mathrm{d}x \, \mathrm{d}v \nonumber \\
& \leq \int_{U \times \mathbb{R}^3} (1+|v|) \left( \int_0^s  \int_{\mathbb{S}^{2}} (1+|v|)M_1 \, \mathrm{d}\nu \, \mathrm{d}\sigma \right) P(x-sv,v)  \, \mathrm{d}x \, \mathrm{d}v \nonumber \\
&\leq  C \int_{U \times \mathbb{R}^3} (1+|v|^2)  P(x-sv,v)  \, \mathrm{d}x \, \mathrm{d}v \leq C \| P \|_Y < \infty.
\end{align}
Also,
\begin{align}\label{na-eq-intexppartialp}
&\int_{U \times \mathbb{R}^3}  (1+|v|) \exp \left(- \int_0^s  \int_{\mathbb{S}^{2}} \int_{\mathbb{R}^3} g_t(x+\varepsilon\nu - \sigma v,\bar{v})[(v-\bar{v})\cdot \nu ]_+ \, \mathrm{d}\bar{v} \, \mathrm{d}\nu \, \mathrm{d}\sigma \right) \partial_{x_i} P(x-sv,v) \, \mathrm{d}x \, \mathrm{d}v \nonumber \\
\leq& \int_{U \times \mathbb{R}^3}  (1+|v|) \partial_{x_i} P(x-sv,v) \, \mathrm{d}x \, \mathrm{d}v \leq \| P \|_Y < \infty.
\end{align}
Combining \eqref{na-eq-partialStep}, \eqref{na-eq-intpartialg} and \eqref{na-eq-intexppartialp} with \eqref{na-eq-Ynormpart1} gives $\| S_t^\varepsilon(s)P \|_Y<\infty $ as required.
\end{proof}

\begin{lem} \label{na-lem-AH2help}
For $\varepsilon,t \geq 0$, $S_t^\varepsilon |_Y$, the restriction of the semigroup $S_t^\varepsilon$ to the space $Y$, is a $C_0$ semigroup on $Y$.
\end{lem}
\begin{proof}
We know that $S_t^\varepsilon$ is a semigroup in $L^1(U \times \mathbb{R}^3)$ and $Y$ is invariant under $S_t^\varepsilon$ by lemma~\ref{na-lem-Yinvariant} so the only remaining property to check is that for any $P \in Y$
\begin{equation} \label{na-eq-StC0Y}
  \lim_{s \downarrow 0} \| S_t^\varepsilon(s)P - P \|_Y = 0.
\end{equation}
Let $\eta \in C^\infty_c (U \times \mathbb{R}^3)$ be a test function and let $\delta >0$. We show for $s>0$ sufficiently small,
\begin{equation}\label{na-eq-stetay}
\| S_t^\varepsilon (s)\eta - \eta  \|_Y = \int_{U \times \mathbb{R}^3} (1+|v|^2)|S_t^\varepsilon(s)\eta - \eta | + (1+|v|)|\partial_x(S_t^\varepsilon(s)\eta - \eta)| \, \mathrm{d}x \, \mathrm{d}v < \delta.
\end{equation}
Since $\eta \in C^\infty_c (U \times \mathbb{R}^3)$ there exists an $R>0$ such that for all $|v|>R$, $\eta(\cdot,v)=0$. By \eqref{na-eq-gl1assmp} we have for any $|v|<R$,
\begin{align*}
\int_0^s \int_{\mathbb{S}^2} \int_{\mathbb{R}^3}  g_t(x+\varepsilon\nu - \sigma v, \bar{v})[(v-\bar{v})\cdot \nu]_+ \, \mathrm{d}\bar{v} \, \mathrm{d}\nu \, \mathrm{d}\sigma  &\leq \int_0^s \int_{\mathbb{S}^2} \int_{\mathbb{R}^3} \bar{g}(\bar{v})[(v-\bar{v})\cdot \nu]_+ \, \mathrm{d}\bar{v} \, \mathrm{d}\nu \, \mathrm{d}\sigma  \\
& \leq \pi \int_0^s \int_{\mathbb{R}^3} \bar{g}(\bar{v})(|v|+ |\bar{v}|) \, \mathrm{d}\bar{v} \, \mathrm{d}\sigma \leq s \pi M_g (1+R).
\end{align*}
Hence for $|v| < R$,
\begin{align*}
1 - \exp \bigg(- \int_0^s \int_{\mathbb{S}^2} \int_{\mathbb{R}^3} g_t(x+\varepsilon\nu - \sigma v, \bar{v})[(v-\bar{v})\cdot \nu]_+ \, \mathrm{d}\bar{v} \, \mathrm{d}\nu \, \mathrm{d}\sigma \bigg) \leq 1 - \exp \left(- s \pi M_g (1+R) \right),
\end{align*}
and this converges to zero as $s$ converges to zero. Therefore
\begin{align} \label{na-eq-stetacomp}
\int_{U \times \mathbb{R}^3} & (1+|v|^2)| S_t^\varepsilon(s)\eta - \eta | \, \mathrm{d}x \, \mathrm{d}v \nonumber \\
& \leq \int_{U \times B_R(0)} (1+|v|^2) \bigg( |\eta(x-sv,v)-\eta(x,v)| + \left( 1 - \exp \left(- s \pi M_g (1+R) \right) \right) | \eta(x,v) | \bigg)  \, \mathrm{d}x \, \mathrm{d}v.
\end{align}
Since $\eta$ is continuous on $U \times \mathbb{R}^3$, is it uniformly continuous on $U \times B_R(0)$ so we can make $s$ sufficiently small so that this is less than $\delta /3$.  Now by \eqref{na-eq-g0w1} we have for almost all $x$
\begin{align} \label{na-eq-partialgtbound}
\int_0^s \int_{\mathbb{S}^2}  \int_{\mathbb{R}^3} |\partial_x g_t(x+\varepsilon\nu - \sigma v,\bar{v})|[(v-\bar{v})\cdot \nu]_+ \, \mathrm{d}\bar{v} \, \mathrm{d}\nu \, \mathrm{d}\sigma
  \leq s \pi M_1(1+|v|).
\end{align}
Hence for $s$ sufficiently small, again by the uniform continuity of $\eta$ on $U \times B_R(0)$,
\begin{align} \label{na-eq-etadiff1}
\int_{U \times \mathbb{R}^3} & (1+|v|) \left| \int_0^s \int_{\mathbb{S}^2}  \int_{\mathbb{R}^3} \partial_x g_t(x+\varepsilon\nu - \sigma v,\bar{v})[(v-\bar{v})\cdot \nu]_+ \, \mathrm{d}\bar{v} \, \mathrm{d}\nu \, \mathrm{d}\sigma  \right| | S_t^\varepsilon(s)\eta(x,v)| \, \mathrm{d}x \, \mathrm{d}v \nonumber \\
& \leq s\int_{U \times B_R(0)} \pi M_1(1+|v|)^2 |\eta(x-sv,v)| \, \mathrm{d}x \, \mathrm{d}v < \delta/3.
\end{align}
Also, by a similar process to \eqref{na-eq-stetacomp} we see that for $s$ sufficiently small
\begin{align} \label{na-eq-etadiff2}
\int_{U \times \mathbb{R}^3} & (1+|v|) | S_t^\varepsilon(s)\partial_x \eta(x,v) - \partial_x\eta(x,v) | \, \mathrm{d} x \, \mathrm{d}v < \delta /3.
\end{align}
Together \eqref{na-eq-etadiff1} and \eqref{na-eq-etadiff2} give that for $s$ sufficiently small
\begin{align*}
\int_{U \times \mathbb{R}^3} & (1+|v|)|\partial_x(S_t^\varepsilon(s)\eta - \eta)| \, \mathrm{d}x \, \mathrm{d}v\\
& = \int_{U \times \mathbb{R}^3}  (1+|v|) \bigg| \left( \int_0^s \int_{\mathbb{S}^2}  \int_{\mathbb{R}^3} \partial_x g_t(x+\varepsilon\nu - \sigma v,\bar{v})[(v-\bar{v})\cdot \nu]_+ \, \mathrm{d}\bar{v} \, \mathrm{d}\nu \, \mathrm{d}\sigma \right)  \\
& \qquad \qquad  S_t^\varepsilon(s)\eta(x,v)+ S_t^\varepsilon(s)\partial_x \eta (x,v) - \partial_x \eta (x,v) \bigg| \, \mathrm{d}x \, \mathrm{d}v < 2\delta / 3.
\end{align*}
Therefore with \eqref{na-eq-stetacomp}, we see that for $s$ sufficiently small, \eqref{na-eq-stetay} holds. Now let $P \in Y$. For $\delta >0$ there exists an $\eta \in C_c^\infty (U \times \mathbb{R}^3)$ such that $\| P-\eta \|_Y < \delta$. Using this and \eqref{na-eq-stetay} finally \eqref{na-eq-StC0Y} can be proved.

\end{proof}

\begin{lem} \label{na-lem-AH2}
For $\varepsilon \geq 0$, condition $(H2)$ of \cite[Theorem 5.3.1]{pazy83} is satisfied for $Y$ as defined above.
\end{lem}
\begin{proof}
Lemma~\ref{na-lem-Yinvariant} proves that $Y$ is invariant under $S_t^\varepsilon(s)$ and lemma~\ref{na-lem-AH2help} proves that $S_t^\varepsilon |_Y$ is a $C_0$ semigroup on $Y$. It remains to prove that $A^\varepsilon|_Y$ is a stable family in $Y$. We use \cite[Theorem 5.2.2]{pazy83}. By the calculations in the proof of lemma~\ref{na-lem-Yinvariant} we have that for any $s,t,\varepsilon \geq 0$ there exists a $C\geq 1$ such that for any $P \in Y$,
\[ \| S_t^\varepsilon(s)P \|_Y \leq C \| P \|_Y.\]
Since $A^\varepsilon(t)$ is the generator of the $C_0$ semigroup $S_t^\varepsilon$, \cite[Theorem 1.5.2]{pazy83} gives that $(0,\infty) \subset \rho(A^\varepsilon(t))$. Hence to apply  \cite[Theorem 5.2.2]{pazy83} we need to show that there exists an $M\geq 1$ and $\omega \geq 0$ such that for any $k \in \mathbb{N}$, any sequence $0 \leq t_1 \leq t_2 \leq \dots \leq t_k \leq T$, any list $0 \leq s_1,\dots,s_k$ and any $P \in Y$,
\begin{equation} \label{na-eq-piPbound}
 \left\| \prod_{j=1}^k S_{t_j}^\varepsilon (s_j)P \right\|_Y \leq M \exp \left( \omega \sum_{j=1}^k s_j \right) \| P \|_Y.
\end{equation}
To that aim fix $k \in \mathbb{N}$,  $0 \leq t_1 \leq t_2 \leq \dots \leq t_k \leq T$ and $0 \leq s_1,\dots,s_k$ and $P \in Y$. Define $\Pi P:= \prod_{j=1}^k S_{t_j}^\varepsilon (s_j)P$. By repeatedly applying \eqref{na-eq-Ssemigroup} we see that,
\begin{align*}
\Pi P(x,v) & = \exp \bigg(  - \int_0^{s_1} \int_{\mathbb{S}^2} \int_{\mathbb{R}^3}  g_{t_1}(x+\varepsilon\nu - \sigma v, \bar{v})[(v-\bar{v})\cdot \nu]_+ \, \mathrm{d}\bar{v} \, \mathrm{d}\nu \, \mathrm{d}\sigma \\
& \qquad - \int_{s_1}^{s_1+s_2} \int_{\mathbb{S}^2} \int_{\mathbb{R}^3}  g_{t_2}(x+\varepsilon\nu - \sigma v, \bar{v})[(v-\bar{v})\cdot \nu]_+ \, \mathrm{d}\bar{v} \, \mathrm{d}\nu \, \mathrm{d}\sigma  - \cdots \\
& \qquad -  \int_{s_1 + \cdots + s_{k-1}}^{s_1 + \cdots + s_k} \int_{\mathbb{S}^2} \int_{\mathbb{R}^3}  g_{t_k}(x+\varepsilon\nu - \sigma v, \bar{v})[(v-\bar{v})\cdot \nu]_+ \, \mathrm{d}\bar{v} \, \mathrm{d}\nu \, \mathrm{d}\sigma \bigg) P(x-\sum_{j=1}^k s_j v,v).
\end{align*}
Denoting the expression inside the exponential by $-W$ we have, by  the same calculation as in \eqref{na-eq-partialgtbound}, for almost all $x \in U$,
\begin{align*}
 | \partial_x W |  & \leq  \int_0^{s_1} \int_{\mathbb{S}^2} \int_{\mathbb{R}^3} |  \partial_x g_{t_1}(x+\varepsilon\nu - \sigma v, \bar{v}) | [(v-\bar{v})\cdot \nu]_+ \, \mathrm{d}\bar{v} \, \mathrm{d}\nu \, \mathrm{d}\sigma \\
& \qquad + \int_{s_1}^{s_1+s_2} \int_{\mathbb{S}^2} \int_{\mathbb{R}^3}  |\partial_x g_{t_2}(x+\varepsilon\nu - \sigma v, \bar{v})|[(v-\bar{v})\cdot \nu]_+ \, \mathrm{d}\bar{v} \, \mathrm{d}\nu \, \mathrm{d}\sigma  + \cdots \\
& \qquad +  \int_{s_1 + \cdots + s_{k-1}}^{s_1 + \cdots + s_k} \int_{\mathbb{S}^2} \int_{\mathbb{R}^3} | \partial_x g_{t_k}(x+\varepsilon\nu - \sigma v, \bar{v}) | [(v-\bar{v})\cdot \nu]_+ \, \mathrm{d}\bar{v} \, \mathrm{d}\nu \, \mathrm{d}\sigma \\
& \qquad \leq  \pi M_1 (1+ |v|) \sum_{j=1}^k s_j.
\end{align*}
Hence, by bounding $\exp(-W) \leq 1$ and using that for any $s \geq 0$, $s \leq \exp(s)$ we have
\begin{align*}
\| \Pi P \|_Y & = \int_{U \times \mathbb{R}^3} (1+ |v|^2)|\Pi P(x,v)| + (1+|v|)|\partial_x (\Pi P)(x,v) | \, \mathrm{d}x \, \mathrm{d}v \\
& \leq \| P \|_Y +  \int_{U \times \mathbb{R}^3}  (1+|v|)\big( |\partial_x W| |\Pi P (x,v)| +  |\Pi \partial_x P(x,v)|  \big)  \, \mathrm{d}x \, \mathrm{d}v \\
&  \leq  \| P \|_Y +  \int_{U \times \mathbb{R}^3}  (1+|v|)\left(  \pi M_1 (1+ |v|) \sum_{j=1}^k s_j|P (x,v)| \right)  \, \mathrm{d}x \, \mathrm{d}v + \| P \|_Y \\
& \leq 2 \| P \|_Y + 2 \pi M_1  \sum_{j=1}^k s_j \int_{U \times \mathbb{R}^3}  (1+|v|^2) |P(x,v)| \, \mathrm{d}x \, \mathrm{d}v \\
& \leq 2 \pi (M_1 +1 ) \| P \|_Y  + 2\pi(M_1+1) \exp \left ( \sum_{j=1}^k s_j \right) \| P \|_Y  \leq 4 \pi (M_1+ 1) \exp \left ( \sum_{j=1}^k s_j \right) \| P \|_Y.
\end{align*}
Hence we see that for $M=4 \pi (M_1+1)$ and $\omega =1$ \eqref{na-eq-piPbound} holds. Thus we can apply \cite[Theorem 5.2.2]{pazy83} which proves that $A^\varepsilon(t) |_Y$ is a stable family in $Y$, which completes the proof of the lemma.
\end{proof}

Having proved condition $(H2)$ in the previous lemma we now move on to proving that condition $(H3)$ holds.
\begin{lem} \label{na-lem-AH3}
For $\varepsilon\geq 0 $, condition $(H3)$ of  \cite[Theorem 5.3.1]{pazy83} is satisfied for $Y$ as defined above.
\end{lem}
\begin{proof}
Let $P \in Y$. Notice that, by \eqref{na-eq-gbar} and \eqref{na-eq-gl1assmp},
\begin{align*}
\|A^\varepsilon(t)P\| &= \int_{U \times \mathbb{R}^3} \Big| -\nu \cdot \partial_x P(x,v) - P(x,v)\int_{\mathbb{S}^{2}} \int_{\mathbb{R}^3} g_t(x+\varepsilon\nu,\bar{v})[(v-\bar{v})\cdot \nu ]_+ \, \mathrm{d}\bar{v} \, \mathrm{d}\nu \Big| \, \mathrm{d}x \, \mathrm{d}v \\
& \leq \int_{U \times \mathbb{R}^3} (1+|v|)|\partial_x P(x,v)| +  |P(x,v)|\int_{\mathbb{S}^{2}} \int_{\mathbb{R}^3} \bar{g}(\bar{v})(|v|+|\bar{v}|) \, \mathrm{d}\bar{v} \, \mathrm{d}\nu \, \mathrm{d}x \, \mathrm{d}v \\
& \leq  \int_{U \times \mathbb{R}^3} (1+|v|)|\partial_x P(x,v)| +  2\pi M_g(1+|v|)|P(x,v)|\, \mathrm{d}x \, \mathrm{d}v \leq C \| P \|_Y,
\end{align*}
for some $C>0$. Hence $Y \subset D(A^\varepsilon(t))$ and $A^\varepsilon(t)$ is bounded as a map $Y \to X$. It remains to prove that $t\mapsto A^\varepsilon(t)$ is continuous in the $B(Y,X)$ norm. Let $P \in Y$, $t\geq 0$ and $\delta>0$. We seek an $\eta>0$ such that for all $s\geq 0$ with $|t-s| < \eta$, we have
\[ \| A^\varepsilon(t)P-A^\varepsilon(s)P \|_X = \int_{U \times \mathbb{R}^3} |A^\varepsilon(t)P-A^\varepsilon(s)P| \, \mathrm{d}x \, \mathrm{d}v \leq \delta.  \]
Now by the definition of $A^\varepsilon$, in definition~\ref{na-deff-ABops},
\begin{align} \label{na-eq-Aeptcts}
\int_{U \times \mathbb{R}^3} & |A^\varepsilon(t)P-A^\varepsilon(s)P| \, \mathrm{d}x \, \mathrm{d}v \nonumber \\
& \leq \int_{U \times \mathbb{R}^3} P(x,v) \int_{\mathbb{S}^{2}} \int_{\mathbb{R}^3} |g_t(x+\varepsilon\nu,\bar{v})- g_s(x+\varepsilon\nu,\bar{v})| [(v-\bar{v})\cdot \nu ]_+ \, \mathrm{d}\bar{v} \, \mathrm{d}\nu \, \mathrm{d}x \, \mathrm{d}v.
\end{align}
Now take $R\geq 1$ sufficiently large such that $R>2C/\delta$, where $C$ is as in lemma~\ref{na-lem-g0diffbound} in section~\ref{na-sec-aux} below.   Further take $\eta >0$ sufficiently small so that, $CR^5\eta^\alpha < \delta /2$. Then lemma~\ref{na-lem-g0diffbound} gives that for any $s$ such that $|t-s|^\alpha <\eta$ and for almost all $x \in U$,
\begin{align*}
\int_{\mathbb{S}^{2}} \int_{\mathbb{R}^3} &  |g_t(x+\varepsilon\nu,\bar{v})- g_s(x+\varepsilon\nu,\bar{v})| [(v-\bar{v})\cdot \nu ]_+ \, \mathrm{d}\bar{v} \, \mathrm{d}\nu \leq  C(1+|v|) \Big(\frac{1}{R} +R^5 |t-s|^\alpha \Big) \leq (1+|v|)\delta.
\end{align*}
Hence substituting this into \eqref{na-eq-Aeptcts},
\[ \int_{U \times \mathbb{R}^3}  |A^\varepsilon(t)P-A^\varepsilon(s)P| \, \mathrm{d}x \, \mathrm{d}v \nonumber  \leq \delta \int_{U \times \mathbb{R}^3} (1+|v|)P(x,v) \, \mathrm{d}x \, \mathrm{d}v. \]

Taking the supremum over all $\| P \|_Y \leq 1$ gives that $\| A^\varepsilon(t)-A^\varepsilon(s) \|_{B(Y,X)} \leq \delta$ as required.
\end{proof}

The above lemmas have proved that conditions $(H1),(H2)$ and $(H3)$ hold. We now prove that the evolution system that results from \cite[Theorem 5.3.1]{pazy83} is indeed $U^\varepsilon$ as defined in \eqref{na-eq-Uepdeff}. We first show in the following lemma that $U^\varepsilon$ is indeed an evolution system.

\begin{lem} \label{na-lem-Uepevol}
Let $U^\varepsilon$ be as in \eqref{na-eq-Uepdeff}. $U^\varepsilon$ is an exponentially bounded evolution family on  $L^1(U\times\mathbb{R}^3)$.
\end{lem}
\begin{proof}
We use \cite[definition~3.1]{nagel02}. It is clear to see that $U^\varepsilon(s,s)$ is the identity operator. Further, for $0\leq s\leq r \leq t$ and $f \in L^1(U \times \mathbb{R}^3)$ we have by \eqref{na-eq-Uepdeff},
\begin{align*}
&U^\varepsilon(t,r)  U^\varepsilon(r,s)f(x,v)   \\
=& U^\varepsilon(t,r) \exp \left(- \int_s^r  \int_{\mathbb{S}^{2}} \int_{\mathbb{R}^3} g_\sigma(x+\varepsilon\nu - (r-\sigma) v,\bar{v})[(v-\bar{v})\cdot \nu ]_+ \, \mathrm{d}\bar{v} \, \mathrm{d}\nu \, \mathrm{d}\sigma \right) f(x-(r-s)v,v) \\
= &  \exp \left(- \int_s^t  \int_{\mathbb{S}^{2}} \int_{\mathbb{R}^3} g_\sigma(x+\varepsilon\nu - (t-\sigma) v,\bar{v})[(v-\bar{v})\cdot \nu ]_+ \, \mathrm{d}\bar{v} \, \mathrm{d}\nu \, \mathrm{d}\sigma \right)   f(x-(t-s)v,v) \\
 =& U^\varepsilon(t,s)f(x,v).
\end{align*}
Exponential boundedness  follows with $M=1, \omega=0$ by bounding the exponential term in $U^\varepsilon$ by 1.
\end{proof}
In the following proposition we now prove that $U^\varepsilon$ is indeed strongly continuous.
\begin{prop} \label{na-prop-Uepstrongcts}
The evolution family $U^\varepsilon$ is strongly continuous.
\end{prop}
To prove this proposition we use part 2 of \cite[Proposition 3.2]{nagel02}. In the following lemmas we prove that iii) holds, that is, uniformly for $0 \leq s\leq t$ in compact subsets,
\begin{enumerate}
\item[a)] $\lim_{s\uparrow t}U^\varepsilon(t,s)f= f$ for all $f \in L^1(U \times \mathbb{R}^3)$
\item[b)] for each $s, f$ the mapping $[s,\infty) \ni t \to U^\varepsilon(t,s)f$ is continuous and,
\item[c)] $ \| U^\varepsilon(t,s) \|$ is bounded.
\end{enumerate}

The proposition gives that this is equivalent to i), strong continuity. We note that c) has been proved in lemma~\ref{na-lem-Uepevol}. We prove a) and b) separately in the following two lemmas.

\begin{lem} \label{na-lem-stronga}
For all $f \in L^1(U \times \mathbb{R}^3)$
\[ \lim_{s\uparrow t}U^\varepsilon(t,s)f= f, \]
uniformly for $0\leq s\leq t  \leq T$.
\end{lem}
To simplify notation here define:
\begin{equation} \label{na-eq-Edeff}
E^\varepsilon(t,s,x,v):= \exp \left(- \int_s^t  \int_{\mathbb{S}^{2}} \int_{\mathbb{R}^3} g_\sigma(x+\varepsilon\nu - (t-\sigma) v,\bar{v})[(v-\bar{v})\cdot \nu ]_+ \, \mathrm{d}\bar{v} \, \mathrm{d}\nu \, \mathrm{d}\sigma \right).
\end{equation}

\begin{proof}

Let $\eta \in C_c^\infty(U\times \mathbb{R}^3)$, $t \geq 0$ and $\delta >0$. We show that for $0\leq s \leq t$ sufficiently close to $t$,
\begin{equation}\label{na-eq-etaintgoal}
\int_{U \times \mathbb{R}^3}| U^\varepsilon(t,s)\eta(x,v) - \eta(x,v) | \,\mathrm{d}x\, \mathrm{d}v < \delta.
\end{equation}

Since $\eta \in C_c^\infty(U\times \mathbb{R}^3)$ there exists an $R>0$ such that for all $|v|>R$, $\eta(\cdot,v) =0$. By \eqref{na-eq-gl1assmp} we have, for any $v \in \mathbb{R}$ with $|v|<R$,
\begin{align} \label{na-eq-gintts}
\int_s^t  \int_{\mathbb{S}^{2}} \int_{\mathbb{R}^3} & g_\sigma(x+\varepsilon\nu - (t-\sigma) v,\bar{v})[(v-\bar{v})\cdot \nu ]_+ \, \mathrm{d}\bar{v} \, \mathrm{d}\nu \, \mathrm{d}\sigma \nonumber \\
& \leq \int_s^t  \int_{\mathbb{S}^{2}} \int_{\mathbb{R}^3} \bar{g}(\bar{v})[(v-\bar{v})\cdot \nu ]_+ \, \mathrm{d}\bar{v} \, \mathrm{d}\nu \, \mathrm{d}\sigma  \leq (t-s)\pi M_g(1+R).
\end{align}
This implies for $t-s$ sufficiently small
\begin{align} \label{na-eq-etaint1}
\int_{U \times \mathbb{R}^3} & |\eta(x-(t-s)v,v)| \left( 1-  E^\varepsilon(t,s,x,v) \right) \, \mathrm{d}x \, \mathrm{d}v \nonumber \\
& \leq  \Big( 1-  \exp \left(-(t-s)\pi M_g(1+R) \right) \Big) \int_{U \times B_R(0)}  |\eta(x-(t-s)v,v)| \mathrm{d}x \, \mathrm{d}v  < \delta /2.
\end{align}
By the uniform continuity of $\eta $ on $U \times B_R(0)$
\begin{align}\label{na-eq-etaint2}
\int_{U \times \mathbb{R}^3} & | \eta(x-(t-s)v,v) - \eta(x,v) | \,\mathrm{d}x\, \mathrm{d}v  =\int_{U \times B_R(0)}  | \eta(x-(t-s)v,v) - \eta(x,v) | \,\mathrm{d}x\, \mathrm{d}v  < \delta /2.
\end{align}
Hence by \eqref{na-eq-etaint1} and \eqref{na-eq-etaint2} for $t-s$ sufficiently small,
\begin{align*}
&\int_{U \times \mathbb{R}^3} | U^\varepsilon(t,s)\eta(x,v) - \eta(x,v) | \,\mathrm{d}x\, \mathrm{d}v \\
 \leq & \int_{U \times \mathbb{R}^3} \big| E^\varepsilon(t,s,x,v) \eta(x-(t-s)v,v)  -  \eta(x-(t-s)v,v) \big| + \big|  \eta(x-(t-s)v,v) -  \eta(x,v) \big| \,\mathrm{d}x\, \mathrm{d}v  < \delta.
\end{align*}
This proves \eqref{na-eq-etaintgoal}. Now for a general $f \in L^1(U\times \mathbb{R}^3)$ there exists an $\eta \in C_c^\infty(U\times \mathbb{R}^3)$ such that,
\begin{equation} \label{na-eq-intfetadiff}
\int_{U\times \mathbb{R}^3} |f(x,v)-\eta(x,v)| \, \mathrm{d}x \, \mathrm{d}v < \delta.
\end{equation}
The required result follows by \eqref{na-eq-intfetadiff} and comparing $f$ and $U^\varepsilon(t,s)f$ with $\eta$ and $U^\varepsilon(t,s)\eta$ respectively.
\end{proof}

\begin{lem} \label{na-lem-strongb}
For any $0\leq s \leq t$, $f \in L^1(U\times \mathbb{R}^3)$, the mapping $[s,\infty) \ni t \to U^\varepsilon(t,s)f$ \ is continuous.
\end{lem}
\begin{proof}
Fix $f \in L^1(U\times \mathbb{R}^3)$ and $\delta>0$. Let $h>0$. By lemma~\ref{na-lem-Uepevol} $U^\varepsilon$ is an evolution family so,
\[ U^\varepsilon(t+h,s)f-U^\varepsilon(t,s)f = U^\varepsilon(t+h,t)U^\varepsilon(t,s)f - U^\varepsilon(t,s)f = U^\varepsilon(t+h,t)g-g, \]
where $g=U^\varepsilon(t,s)f$. Since $g \in L^1(U \times \mathbb{R}^3)$ we can follow the proof of lemma~\ref{na-lem-stronga} to prove that for $h$ sufficiently small,
\[ \int_{U \times \mathbb{R}^3} | U^\varepsilon(t+h,t)g(x,v)-g(x,v) | \, \mathrm{d}x \, \mathrm{d}v < \delta. \]
It remains to prove that
\[ \lim_{h \downarrow 0} U^\varepsilon(t-h,s)f = U^\varepsilon(t,s)f.  \]
Fix $\delta >0$. Let $h>0$. Then using \eqref{na-eq-Edeff},
\begin{align} \label{na-eq-Uepctsminh}
& \int_{U \times \mathbb{R}^3} | U^\varepsilon(t,s)f(x,v)-U^\varepsilon(t-h,s)f(x,v) | \, \mathrm{d}x \, \mathrm{d}v \nonumber \\
 =& \int_{U \times \mathbb{R}^3} | E^\varepsilon(t,s,x,v)f(x-(t-s)v,v)
 -E^\varepsilon(t-h,s,x,v)f(x-(t-h-s)v,v) | \, \mathrm{d}x \, \mathrm{d}v \nonumber \\
 \leq& \int_{U \times \mathbb{R}^3} E^\varepsilon(t,s,x,v) | f(x-(t-s)v,v)-f(x-(t-h-s)v,v)| \nonumber \\ & \quad+|E^\varepsilon(t-h,s,x,v)-E^\varepsilon(t,s,x,v)||f(x-(t-h-s)v,v) | \, \mathrm{d}x \, \mathrm{d}v \nonumber \\
 =& I_1+I_2.
\end{align}
Now since $E^\varepsilon(t,s,x,v) \leq 1$ we have that
\begin{align*} \label{na-eq-Uepctsmin1}
I_1 & = \int_{U \times \mathbb{R}^3}  E^\varepsilon(t,s,x,v) | f(x-(t-s)v,v)-f(x-(t-h-s)v,v)| \, \mathrm{d}x \, \mathrm{d}v \nonumber \\
& \leq \int_{U \times \mathbb{R}^3}  | f(x-(t-s)v,v)-f(x-(t-h-s)v,v)| \, \mathrm{d}x \, \mathrm{d}v.
\end{align*}
We can make this less than $\delta/2$ by approximating $f$ with a test function $\eta \in C_c^\infty(U \times \mathbb{R}^3)$ as in the above lemma. We now look to $I_2$. Firstly since $f \in L^1(U\times\mathbb{R}^3)$ there exists an $R>0$ such that,
\[ \int_{U \times \mathbb{R}^3\setminus B_R(0)} f(x,v)\, \mathrm{d}x \, \mathrm{d}v < \frac{\delta}{8}.  \]
Hence,
\begin{align} \label{na-eq-I2bound}
I_2 &  = \int_{U \times B_R(0)} |E^\varepsilon(t-h,s,x,v)-E^\varepsilon(t,s,x,v)||f(x-(t-h-s)v,v) | \, \mathrm{d}x \, \mathrm{d}v \nonumber \\
& \quad \quad +  \int_{U \times \mathbb{R}^3\setminus B_R(0)} |E^\varepsilon(t-h,s,x,v)-E^\varepsilon(t,s,x,v)||f(x-(t-h-s)v,v) | \, \mathrm{d}x \, \mathrm{d}v \nonumber \\
& < \int_{U \times B_R(0)} |E^\varepsilon(t-h,s,x,v)-E^\varepsilon(t,s,x,v)||f(x-(t-h-s)v,v) | \, \mathrm{d}x \, \mathrm{d}v  + \frac{\delta}{4}.
\end{align}
By the mean value theorem for any $\alpha,\beta \leq 0$ there exists an $\theta \in (\alpha,\beta) \cup (\beta,\alpha)$ such that
\[ \left| \frac{\exp(\alpha)-\exp(\beta)}{\alpha-\beta} \right| = \exp(\theta), \]
hence
\begin{equation} \label{na-eq-expab}
|\exp(\alpha)-\exp(\beta)| \leq |\alpha - \beta|.
\end{equation}
By lemma~\ref{na-lem-g0diffbound}, for any $R_2\geq 1$ and for almost all $x \in U$,
\begin{align*}
\bigg| & \exp  \left(- \int_s^{t-h}  \int_{\mathbb{S}^{2}} \int_{\mathbb{R}^3} g_\sigma(x+\varepsilon\nu - (t-h-\sigma) v,\bar{v})[(v-\bar{v})\cdot \nu ]_+ \, \mathrm{d}\bar{v} \, \mathrm{d}\nu \, \mathrm{d}\sigma \right)  \\
& \quad \quad - \exp \left(- \int_s^{t-h}  \int_{\mathbb{S}^{2}} \int_{\mathbb{R}^3} g_\sigma(x+\varepsilon\nu - (t-\sigma) v,\bar{v})[(v-\bar{v})\cdot \nu ]_+ \, \mathrm{d}\bar{v} \, \mathrm{d}\nu \, \mathrm{d}\sigma \right) \bigg| \\
& \leq  \int_s^{t-h}  \int_{\mathbb{S}^{2}} \int_{\mathbb{R}^3} |g_\sigma(x+\varepsilon\nu - (t-h-\sigma) v,\bar{v}) -g_\sigma(x+\varepsilon\nu - (t-\sigma) v,\bar{v})|[(v-\bar{v})\cdot \nu ]_+ \, \mathrm{d}\bar{v} \, \mathrm{d}\nu \, \mathrm{d}\sigma \\
& \leq  \int_s^{t-h}  C(1+|v|) \Big(\frac{1}{R_2} +R_2^5 h^\alpha  \Big) \, \mathrm{d}\sigma  \leq (t-h-s)  C(1+|v|) \Big(\frac{1}{R_2} +R_2^5 h^\alpha  \Big)  \\
& \leq T C(1+|v|) \Big(\frac{1}{R_2} +R_2^5 h^\alpha  \Big) .
\end{align*}
Hence, by a similar calculation to \eqref{na-eq-gintts},
\begin{align*}
&| E^\varepsilon  (t-h,s,x,v)-E^\varepsilon(t,s,x,v)| \\
 =& \bigg| \exp \left(- \int_s^{t-h}  \int_{\mathbb{S}^{2}} \int_{\mathbb{R}^3} g_\sigma(x+\varepsilon\nu - (t-h-\sigma) v,\bar{v})[(v-\bar{v})\cdot \nu ]_+ \, \mathrm{d}\bar{v} \, \mathrm{d}\nu \, \mathrm{d}\sigma \right) \\
& \quad \quad  - \exp \left(- \int_s^t  \int_{\mathbb{S}^{2}} \int_{\mathbb{R}^3} g_\sigma(x+\varepsilon\nu - (t-\sigma) v,\bar{v})[(v-\bar{v})\cdot \nu ]_+ \, \mathrm{d}\bar{v} \, \mathrm{d}\nu \, \mathrm{d}\sigma \right) \bigg|\\
\leq &\bigg| \exp \left(- \int_s^{t-h}  \int_{\mathbb{S}^{2}} \int_{\mathbb{R}^3} g_\sigma(x+\varepsilon\nu - (t-h-\sigma) v,\bar{v})[(v-\bar{v})\cdot \nu ]_+ \, \mathrm{d}\bar{v} \, \mathrm{d}\nu \, \mathrm{d}\sigma \right)  \\
& \quad \quad - \exp \left(- \int_s^{t-h}  \int_{\mathbb{S}^{2}} \int_{\mathbb{R}^3} g_\sigma(x+\varepsilon\nu - (t-\sigma) v,\bar{v})[(v-\bar{v})\cdot \nu ]_+ \, \mathrm{d}\bar{v} \, \mathrm{d}\nu \, \mathrm{d}\sigma \right) \bigg| \\
& \quad \quad + \exp \left(- \int_s^{t-h}  \int_{\mathbb{S}^{2}} \int_{\mathbb{R}^3} g_\sigma(x+\varepsilon\nu - (t-\sigma) v,\bar{v})[(v-\bar{v})\cdot \nu ]_+ \, \mathrm{d}\bar{v} \, \mathrm{d}\nu \, \mathrm{d}\sigma \right) \\
& \quad \quad \left( 1- \exp \left(- \int_t^{t-h}  \int_{\mathbb{S}^{2}} \int_{\mathbb{R}^3} g_\sigma(x+\varepsilon\nu - (t-\sigma) v,\bar{v})[(v-\bar{v})\cdot \nu ]_+ \, \mathrm{d}\bar{v} \, \mathrm{d}\nu \, \mathrm{d}\sigma \right) \right) \\
\leq &T C(1+|v|) \Big(\frac{1}{R_2} +R_2^5 h^\alpha  \Big)\\& \quad \quad + \left( 1- \exp \left(- \int_t^{t-h}  \int_{\mathbb{S}^{2}} \int_{\mathbb{R}^3} g_\sigma(x+\varepsilon\nu - (t-\sigma) v,\bar{v})[(v-\bar{v})\cdot \nu ]_+ \, \mathrm{d}\bar{v} \, \mathrm{d}\nu \, \mathrm{d}\sigma \right) \right) \\
 \leq & T C(1+|v|) \Big(\frac{1}{R_2} +R_2^5 h^\alpha  \Big) + 1 - \exp (-h\pi M_g (1+|v|)).
\end{align*}
This gives that
\begin{align*}
 &\int_{U \times B_R(0)}  |E^\varepsilon(t-h,s,x,v)-E^\varepsilon(t,s,x,v)||f(x-(t-h-s)v,v) | \, \mathrm{d}x \, \mathrm{d}v \\
\leq & \int_{U \times B_R(0)} \Big(  T C(1+|v|) \Big(\frac{1}{R_2} +R_2^5 h^\alpha  \Big) + 1 - \exp (-h\pi M_g (1+|v|)) \Big)  |f(x-(t-h-s)v,v)| \, \mathrm{d}x \, \mathrm{d}v \\
 \leq & \Big(  T C(1+R) \Big(\frac{1}{R_2} +R_2^5 h^\alpha  \Big) + 1 - \exp (-h\pi M_g (1+R)) \Big)   \int_{U \times B_R(0)} |f(x-(t-h-s)v,v)| \, \mathrm{d}x \, \mathrm{d}v \\
 \leq & \Big(  T C(1+R) \Big(\frac{1}{R_2} +R_2^5 h^\alpha  \Big) + 1 - \exp (-h\pi M_g (1+R)) \Big) \| f \|.
\end{align*}
Now take $R_2 \geq 1$ sufficiently large such that
\[ \frac{TC(1+R)\|f\|}{R_2} < \frac{\delta }{12},  \]
and $h>0$ sufficiently small so that both
\[ TC(1+R)R_2^5 \|f \| h^\alpha < \frac{\delta}{12} \textrm{ and, } 1 - \exp (-h\pi M_g (1+R))  \| f \| < \frac{\delta}{12}.    \]
Hence
\[  \int_{U \times B_R(0)}   |E^\varepsilon(t-h,s,x,v)-E^\varepsilon(t,s,x,v)||f(x-(t-h-s)v,v) | \, \mathrm{d}x \, \mathrm{d}v < \frac{\delta}{4}. \]
Substituting this into \eqref{na-eq-I2bound} gives $I_2 < \delta/2$. Returning to \eqref{na-eq-Uepctsminh} this gives for $h>0$ sufficiently small,
\[ \int_{U \times \mathbb{R}^3}  | U^\varepsilon(t,s)f(x,v)-U^\varepsilon(t-h,s)f(x,v) | \, \mathrm{d}x \, \mathrm{d}v < \delta,  \]
which completes the proof of the lemma.
\end{proof}

\begin{proof}[Proof of proposition~\ref{na-prop-Uepstrongcts}]
This proposition follows from lemma~\ref{na-lem-stronga} and lemma~\ref{na-lem-strongb}.
\end{proof}

Finally to prove proposition~\ref{na-prop-Pj0} it remains to prove that $U^\varepsilon$ satisfies the properties $(E1),(E2)$ and $(E3)$.

\begin{prop} \label{na-prop-UepE123}
The evolution system $U^\varepsilon$ satisfies the properties $(E1),(E2)$ and $(E3)$ of \cite[Theorem 5.3.1]{pazy83}.
\end{prop}
\begin{proof}
By bounding the exponential term by $1$ it is clear that $(E1)$ holds with $M=1, \omega =0$. Now let $P \in Y$  and $\Omega \subset U \times \mathbb{R}^3$ be measurable. Then
\begin{align*}
\int_{\Omega} \partial_t|_{t=s}P(x-(t-s)v,v) \, \mathrm{d}x \, \mathrm{d}v = \int_{\Omega}  -v\cdot \partial_x P(x,v) \, \mathrm{d}x \, \mathrm{d}v.
\end{align*}
And using \eqref{na-eq-Edeff} we have
\begin{align*}
\int_{\Omega} & P(x,v)  \partial_t |_{t=s} E^\varepsilon(t,s,x,v) \, \mathrm{d}x \, \mathrm{d}v \\
& = \int_{\Omega}  P(x,v) \bigg( -\int_{\mathbb{S}^2} \int_{\mathbb{R}^3} g_t(x+\varepsilon\nu,\bar{v})[(v-\bar{v})\cdot \nu]_+ \, \mathrm{d}\bar{v} \, \mathrm{d}\nu  \bigg|_{t=s} \\
& \qquad -v\cdot \int_s^t \int_{\mathbb{S}^2} \int_{\mathbb{R}^3} \partial_x g_\sigma(x+\varepsilon\nu-(t-\sigma)v,\bar{v}) [(v-\bar{v})\cdot \nu]_+ \mathrm{d}\bar{v} \, \mathrm{d}\nu \, \mathrm{d}\sigma \bigg|_{t=s} \bigg) \, \mathrm{d}x \, \mathrm{d}v \\
& = - \int_{\Omega}   P(x,v)  \int_{\mathbb{S}^2} \int_{\mathbb{R}^3} g_s(x+\varepsilon\nu,\bar{v})[(v-\bar{v})\cdot \nu]_+ \, \mathrm{d}\bar{v} \, \mathrm{d}\nu \, \mathrm{d}x \, \mathrm{d}v.
\end{align*}
Hence
\begin{align*}
\int_{\Omega} & \partial_t |_{t=s} (U^\varepsilon(t,s)P)(x,v) \, \mathrm{d}x \, \mathrm{d}v \\
& = \int_{\Omega} P(x,v) \partial_t |_{t=s} E^\varepsilon(t,s,x,v) + E^\varepsilon(s,s,x,v)  \partial_t |_{t=s} P(x-(t-s)v,v) \, \mathrm{d}x \, \mathrm{d}v \\
& = \int_{\Omega}  -v\cdot \partial_x P(x,v) - P(x,v)  \int_{\mathbb{S}^2} \int_{\mathbb{R}^3} g_s(x+\varepsilon\nu,\bar{v})[(v-\bar{v})\cdot \nu]_+ \, \mathrm{d}\bar{v} \, \mathrm{d}\nu \, \mathrm{d}x \, \mathrm{d}v \\
& = \int_{\Omega} A^\varepsilon(s) P(x,v) \, \mathrm{d}x \, \mathrm{d}v.
\end{align*}
This proves $(E2)$. Further
\begin{align*}
\int_{\Omega} & E^\varepsilon(t,s,x,v) \partial_s P(x-(t-s)v,v)  \, \mathrm{d}x \, \mathrm{d}v  = \int_{\Omega}   E^\varepsilon(t,s,x,v) v \cdot \partial_x P(x-(t-s)v,v)  \, \mathrm{d}x \, \mathrm{d}v
\end{align*}
and
\begin{align*}
\int_{\Omega} &P(x-(t-s)v,v)  \partial_s E^\varepsilon(t,s,x,v)  \, \mathrm{d}x \, \mathrm{d}v \\
& = \int_{\Omega} P(x-(t-s)v,v)  E^\varepsilon(t,s,x,v)  \int_{\mathbb{S}^2}   \int_{\mathbb{R}^3} g_s(x+\varepsilon\nu - (t-s)v,\bar{v}) [(v-\bar{v})\cdot \nu]_+ \, \mathrm{d}\bar{v} \, \mathrm{d}\nu     \, \mathrm{d}x \, \mathrm{d}v.
\end{align*}
Hence
\begin{align*}
\int_{\Omega} & \partial_s (U^\varepsilon(t,s)P)(x,v) \, \mathrm{d}x \, \mathrm{d}v \\
& = \int_{\Omega} E^\varepsilon(t,s,x,v) \partial_s P(x-(t-s)v,v) + P(x-(t-s)v,v)  \partial_s E^\varepsilon(t,s,x,v)  \, \mathrm{d}x \, \mathrm{d}v \\
& = \int_{\Omega} E^\varepsilon(t,s,x,v) v \cdot \partial_x P(x-(t-s)v,v) \\
& \qquad + E^\varepsilon(t,s,x,v) P(x-(t-s)v,v)    \int_{\mathbb{S}^2}   \int_{\mathbb{R}^3} g_s(x+\varepsilon\nu - (t-s)v,\bar{v}) [(v-\bar{v})\cdot \nu]_+ \, \mathrm{d}\bar{v} \, \mathrm{d}\nu    \, \mathrm{d}x \, \mathrm{d}v \\
&= \int_{\Omega} - U^\varepsilon(t,s)A^\varepsilon(s)P(x,v)  \, \mathrm{d}x \, \mathrm{d}v.
\end{align*}
This proves $(E3)$ which completes the proof of the lemma.
\end{proof}

We can finally now combine all the results in this subsection to prove proposition~\ref{na-prop-Pj0}.

\begin{proof}[Proof of proposition~\ref{na-prop-Pj0}]
Let $\varepsilon \geq 0$. By lemmas~\ref{na-lem-AH1}, \ref{na-lem-AH2} and \ref{na-lem-AH3} we can apply \cite[Theorem 5.3.1]{pazy83}. This gives that there exists a unique evolution system satisfying $(E1),(E2),(E3)$. By lemma~\ref{na-lem-Uepevol}, $U^\varepsilon$ is an exponentially bounded evolution family and by proposition~\ref{na-prop-Uepstrongcts} it is strongly continuous. By proposition~\ref{na-prop-UepE123}, $U^\varepsilon$ satisfies these conditions and hence the solution is given by $P_t^{\varepsilon,(0)}=U^\varepsilon(t,0)f_0$ as required.
\end{proof}

\subsection{Existence of Non-Autonomous Linear Boltzmann Solution} \label{na-sec-idlinbolt}
In this subsection we prove that there exists a solution to the $\varepsilon$ dependent non-autonomous linear Boltzmann equation \eqref{na-eq-linboltzep}. We prove the result by adapting the method of \cite{arlotti14}.

\begin{prop} \label{na-prop-linboltz}
For $\varepsilon \geq 0$ there exists a solution $f^\varepsilon:[0,T] \to L^1(U\times\mathbb{R}^3)$ to the non-autonomous linear Boltzmann equation \eqref{na-eq-linboltz}. Moreover there exists a $K>0$ such that for any $t \in [0,T]$ and any $\varepsilon \geq 0$
\begin{equation} \label{na-eq-ft1st}
\int_{U \times \mathbb{R}^3} f_t^\varepsilon(x,v)(1+|v|) \, \mathrm{d}x \, \mathrm{d}v \leq K.
\end{equation}
\end{prop}

\begin{rmk}
Later, in proposition~\ref{na-prop-ftprob}, we are able to show that for any $\varepsilon \geq 0$ and any $t \in [0,T]$, $f_t^\varepsilon$ is a probability measure on $U \times \mathbb{R}^3$ and that $f_t^\varepsilon$ converges in $L^1$ to $f_t^0$ uniformly for $t\in[0,T]$.
\end{rmk}
We first introduce some notation. For $\varepsilon, t \geq 0$, $(x,v) \in U \times \mathbb{R}^3$ and $v_* \neq v \in \mathbb{R}^3$ define
\begin{align*}
\Sigma^\varepsilon_t(x,v) & := \int_{\mathbb{S}^2} \int_{\mathbb{R}^3} g_t(x+\varepsilon\nu,\bar{v})[(v-\bar{v})\cdot \nu]_+ \, \mathrm{d}\bar{v} \, \mathrm{d}\nu, \textrm{ and } \\
k^\varepsilon_t(x,v,v_*) &= \frac{1}{| v-v_* |} \int_{E_{vv_*}} g_t(x+\varepsilon\nu,w) \, \mathrm{d}w,
\end{align*}
where $E_{vv_*} = \{ w \in \mathbb{R}^3 : w \cdot (v-v_*)  = v \cdot (v-v_*)\}$. By the use of Carleman's representation (see \cite{carleman57} and \cite[Section 3]{bisi14}) we have
\begin{align*}
(B^\varepsilon(t)f)(x,v) & =  \int_{\mathbb{S}^2} \int_{\mathbb{R}^3} f(x,v')g_t(x+\varepsilon\nu,\bar{v}')[(v-\bar{v})\cdot \nu]_+ \, \mathrm{d}\bar{v} \, \mathrm{d}\nu
 = \int_{\mathbb{R}^3} k^\varepsilon_t (x,v,v_*)f(x,v_* ) \, \mathrm{d}v_*.
\end{align*}

\begin{lem} \label{na-lem-arlotti1}
For any $f \in L^1_+ (U \times \mathbb{R}^3)$ and any $s, \varepsilon \geq 0$ we have $\Sigma_t^\varepsilon U^\varepsilon(t,s)f \in L^1(U \times \mathbb{R}^3)$ for almost all $t \geq s$.
\end{lem}
\begin{proof}
We adapt the proof of \cite[lemma 5.3]{arlotti14}. Define $Q:[s, \infty) \to [0,\infty)$ by
\begin{align*}
Q(t) &: = \int_{U \times \mathbb{R}^3} \Sigma_t^\varepsilon(x,v)U^\varepsilon(t,s)f(x,v) \, \mathrm{d}x \mathrm{d}v \\
& = \int_{U \times \mathbb{R}^3} \Sigma_t^\varepsilon(x,v) \exp \left(  - \int_s^t \Sigma_\sigma^\varepsilon(x-(t-\sigma )v,v) \, \mathrm{d}\sigma \right) f(x-(t-s)v,v) \, \mathrm{d}x \, \mathrm{d}v
\end{align*}
Then for any $r>s$ we have, using the substitution $\bar{x}=x-tv$ and Fubini's theorem,
\begin{align*}
\int_s^r Q(t) \, \mathrm{d}t & = \int_s^r \int_{U \times \mathbb{R}^3} \Sigma_t^\varepsilon(\bar{x}+tv,v) \exp \left(  - \int_s^t \Sigma_\sigma^\varepsilon(\bar{x} +  \sigma  v) \, \mathrm{d}\sigma \right) f(\bar{x}+sv,v) \, \mathrm{d}\bar{x} \, \mathrm{d}v \, \mathrm{d}t \\
& = \int_{U \times \mathbb{R}^3} f(\bar{x}+sv,v) \int_s^r  \Sigma_t^\varepsilon(\bar{x}+tv,v) \exp \left(  - \int_s^t \Sigma_\sigma^\varepsilon(\bar{x} +  \sigma  v) \, \mathrm{d}\sigma \right)   \, \mathrm{d}t \, \mathrm{d}\bar{x} \, \mathrm{d}v \\
& = \int_{U \times \mathbb{R}^3} f(\bar{x}+sv,v) \int_s^r - \partial_t   \exp \left(  - \int_s^t \Sigma_\sigma^\varepsilon(\bar{x} +  \sigma  v) \, \mathrm{d}\sigma \right)   \, \mathrm{d}t \, \mathrm{d}\bar{x} \, \mathrm{d}v \\
& = \int_{U \times \mathbb{R}^3} f(\bar{x}+sv,v) \left( 1  -    \exp \left(  - \int_s^r \Sigma_\sigma^\varepsilon(\bar{x} +  \sigma  v) \, \mathrm{d}\sigma \right)  \right)  \, \mathrm{d}\bar{x} \, \mathrm{d}v \\
& = \|f \| - \| U(r,s)f \| < \infty.
\end{align*}
Hence $Q(t)$ is finite for almost all $t\geq s$ which proves the lemma.
\end{proof}

\begin{lem} \label{na-lem-arlotti2}
For any $f\in L_+^1 (U \times \mathbb{R}^3)$ and any $s \geq 0$, $U^\varepsilon(t,s) f \in D(B^\varepsilon(t))$ for almost every $t \geq s$ and the mapping $[s,\infty) \ni t \mapsto B^\varepsilon(t)U(t,s)f $ is measurable. Moreover, for any $r \geq s$,
\begin{align*}
\int_s^r \| B^\varepsilon(t)U^\varepsilon(t,s)f \| \, \mathrm{d}t = \int_s^r \| \Sigma_t^\varepsilon U^\varepsilon(t,s)f \| \, \mathrm{d}t = \| f \| - \| U^\varepsilon (r,s)f \|.
\end{align*}
\end{lem}
\begin{proof}
By changing from pre to post collisional variables, see for example \cite[Chapter 2, section 1.4.5]{02handbook}, we have, for any $x \in U$,
\begin{align} \label{na-eq-BeqSigma}
\int_{\mathbb{R}^3 }  B^\varepsilon(t)f(x,v) \, \mathrm{d}v & = \int_{\mathbb{R}^3 }  \int_{\mathbb{S}^2} \int_{\mathbb{R}^3} f(x,v')g_t(x+\varepsilon\nu,\bar{v}')[(v-\bar{v})\cdot \nu]_+ \, \mathrm{d}\bar{v} \, \mathrm{d}\nu  \, \mathrm{d}v  \nonumber \\
& = \int_{\mathbb{R}^3 } \int_{\mathbb{S}^2} \int_{\mathbb{R}^3} f(x,v)g_t(x+\varepsilon\nu,\bar{v})[(v-\bar{v})\cdot \nu]_+ \, \mathrm{d}\bar{v} \, \mathrm{d}\nu \, \mathrm{d}v \nonumber \\
& = \int_{\mathbb{R}^3 } \Sigma_t^\varepsilon(x,v)f(x,v)  \, \mathrm{d}v .
\end{align}
The required results now follow from the statement and proof of the previous lemma.
\end{proof}

\begin{proof}[Proof of proposition~\ref{na-prop-linboltz}]
For this proof we use \cite{arlotti14}. The above two lemmas give that the modification of \cite[lemma 5.11, corollary 5.12 and  assumptions 5.1]{arlotti14} to our situation hold. Hence, as in \cite[section 5.2]{arlotti14}, we see that \cite[theorem 2.1]{arlotti14} holds, which gives that there exists an evolution family $V^\varepsilon(t,s)$. Hence $f_t^\varepsilon := V^\varepsilon(t,0)f_0$ defines a solution  to the non-autonomous linear Boltzmann equation \eqref{na-eq-linboltz}. We now prove \eqref{na-eq-ft1st}. We note that by \cite[theorem 2.1]{arlotti14} for any $f \in L_+^1(U \times \mathbb{R}^3)$,
\begin{equation} \label{na-eq-Vepnorm}
\int_{U \times \mathbb{R}^3} V^\varepsilon(t,0)f(x,v) \, \mathrm{d}x \, \mathrm{d}v \leq \int_{U \times \mathbb{R}^3} f(x,v) \, \mathrm{d}x \, \mathrm{d}v,
\end{equation}
and,
\[ f_t^\varepsilon = V^\varepsilon(t,0)f_0 = U^\varepsilon(t,0)f_0 + \int_0^t V^\varepsilon(t,r)B^\varepsilon(r)U^\varepsilon(r,0)f_0 \, \mathrm{d}r.  \]
Now by \eqref{na-eq-gl1assmp} for almost all $x \in U$,
\begin{align*}
\Sigma_t^\varepsilon(x,v) & =  \int_{\mathbb{S}^2} \int_{\mathbb{R}^3} g_t(x+\varepsilon\nu,\bar{v})[(v-\bar{v})\cdot \nu]_+ \, \mathrm{d}\bar{v} \, \mathrm{d}\nu \leq \pi \int_{\mathbb{R}^3} \bar{g}(\bar{v})(|v|+|\bar{v}|) \, \mathrm{d}\bar{v} \leq \pi M_g(1+|v|).
\end{align*}
So by \eqref{na-eq-BeqSigma}, noting that $U^\varepsilon(r,0) f_0(x,v)(1+|v|) \in L_+^1(U \times \mathbb{R}^3)$,
\begin{align*}
&\int_{U \times \mathbb{R}^3}   B^\varepsilon(r)U^\varepsilon(r,0)f_0(x,v)(1+|v|) \, \mathrm{d}x \, \mathrm{d}v
 = \int_{U \times \mathbb{R}^3} \Sigma_t^\varepsilon(x,v) U^\varepsilon(r,0)f_0(x,v)(1+|v|) \, \mathrm{d}x \, \mathrm{d}v \\
 \leq& \int_{U \times \mathbb{R}^3} \pi M_g(1+|v|)^2 U^\varepsilon(r,0)f_0(x,v) \, \mathrm{d}x \, \mathrm{d}v
  \leq 2 \pi M_g \int_{U \times \mathbb{R}^3}  f_0(x,v)(1+|v|^2) \, \mathrm{d}x \, \mathrm{d}v.
\end{align*}
Hence,
\begin{align*}
\int_{U \times \mathbb{R}^3} & f_t^\varepsilon(x,v)(1+|v|) \, \mathrm{d}x \, \mathrm{d}v \\
& = \int_{U \times \mathbb{R}^3}  U^\varepsilon(t,0)f_0(x,v)(1+|v|)  + \int_0^t V^\varepsilon(t,r)B^\varepsilon(r)U^\varepsilon(r,0)f_0 \, \mathrm{d}r (x,v)(1+|v|) \, \mathrm{d}x \, \mathrm{d}v \\
&\leq \int_{U \times \mathbb{R}^3}  f_0(x,v)(1+|v|) \, \mathrm{d}x \, \mathrm{d}v + \int_0^t \int_{U \times \mathbb{R}^3} V^\varepsilon(t,r)B^\varepsilon(r)U^\varepsilon(r,0)f_0  (x,v)(1+|v|)\, \mathrm{d}x \, \mathrm{d}v \, \mathrm{d}r \\
& \leq \int_{U \times \mathbb{R}^3}  f_0(x,v)(1+|v|) \, \mathrm{d}x \, \mathrm{d}v + \int_0^t 2 \pi M_g \int_{U \times \mathbb{R}^3}  f_0(x,v)(1+|v|^2) \, \mathrm{d}x \, \mathrm{d}v \, \mathrm{d}r \\
& \leq M_f(1+2\pi T M_g) =: K < \infty.
\end{align*}
\end{proof}
\begin{rmk}
We were unable to adapt the honesty results of \cite{arlotti14} to our situation, so we cannot yet deduce that $V^\varepsilon$ is an honest semigroup and that the solution $f_t^\varepsilon = V^\varepsilon(t,0)f_0$ conserves mass in the expect way. Honestly is proved later in proposition~\ref{na-prop-ftprob} by exploiting the connection to the idealised equation.
\end{rmk}

\subsection{Building the Solution} \label{na-sec-idbuild}
In this subsection we construct the function $P_t^\varepsilon(\Phi)$ iteratively and prove that is satisfies the properties in theorem~\ref{na-thm-id}. After defining $P_t^\varepsilon(\Phi)$, we define $P_t^{\varepsilon,(j)}$, which similarly to $P_t^{\varepsilon,(0)}$, can be thought of as the probability that the tagged particle is at a certain position and has experienced exactly $j$ collisions. Once a few properties of $P_t^{\varepsilon,(j)}$ have been checked the majority of theorem~\ref{na-thm-id} follows. Proving that $P_t^\varepsilon$ is indeed a probability measure on $\mathcal{MT}$ requires a careful analysis of the time derivative of $P_t^\varepsilon$.

This subsection differs from our previous work \cite{matt16} in two ways. Firstly the $\varepsilon$ dependence, which is important in  the later proof to deal with spatial dependence for positive $\varepsilon$, but  makes little technical difference. Secondly there are significant differences in proving that $P_t^\varepsilon$ is a probability measure. Previously it followed from the honesty of the solution of the autonomous linear Boltzmann equation that the idealised distribution is a probability measure. However in this case we do not have the equivalent honesty result for the non-autonomous linear Boltzmann. Therefore we prove that $P_t^\varepsilon$ is a probability measure by explicitly showing that the measure of the whole space has zero derivative with respect to time. This requires a significant number of calculations.
\begin{deff} \label{na-deff-Ptphi}
For $j \in \mathbb{N} \cup \{ 0\}$ define,
\begin{equation} \label{na-eq-tjdeff}
\mathcal{T}_j := \{ \Phi \in \mathcal{MT} : n(\Phi) = j \}.
\end{equation}
That is, $\mathcal{T}_j$ contains all trees with exactly $j$ collisions.
Let $\varepsilon \geq 0$ and $t \in [0,T]$.  For $\Phi \in \mathcal{T}_0$ define
\begin{equation} \label{na-eq-ptt0}
 P_t^\varepsilon(\Phi) := P_t^{\varepsilon,(0)} (x(t),v(t)).
\end{equation}
Else define
\begin{align} \label{na-eq-ptfull}
P_t^\varepsilon(\Phi): = \mathbbm{1}_{t \geq \tau} \exp \left(- \int_\tau^t \int_{\mathbb{S}^2} \int_{\mathbb{R}^3} g_\sigma(x(\sigma)+\varepsilon\nu,\bar{v}) [(v(\tau)-\bar{v})\cdot \nu']_+ \, \mathrm{d}\bar{v} \, \mathrm{d}\nu' \, \mathrm{d}\sigma  \right) \nonumber \\ P_\tau (\bar{\Phi})g_\tau(x(\tau)+\varepsilon\nu,v')[(v(\tau^-)-v')\cdot \nu]_+.
\end{align}
The right hand side of this equation depends on $P_\tau^\varepsilon (\bar{\Phi})$ but since $\bar{\Phi}$ has degree exactly one less than $\Phi$ and we have defined $P_t^\varepsilon(\Phi)$ for trees $\Phi$ with degree $0$ the equation is well defined. Note that this definition implies that for any, $\varepsilon \geq 0$, $\Phi \in \mathcal{MT}$ and $\tau \leq s \leq t \leq T$.
\[ P_t^\varepsilon(\Phi ) = \exp\left ( - \int_s^t L_\sigma ^\varepsilon(\Phi) \, \mathrm{d}\sigma \right) P_s^\varepsilon(\Phi).  \]
\end{deff}

\begin{deff} \label{na-deff-ptj}
Let $t \in [0,T]$, $j \geq 1$, $\varepsilon \geq 0$ and $\Omega \subset U \times \mathbb{R}^3$ be measurable. Recall we define $S_t(\Omega)= \{ \Phi \in \mathcal{MT} : (x(t),v(t)) \in \Omega  \}$. Define
\begin{equation} \label{na-eq-deffstjomega}
S_t^j(\Omega):= S_t(\Omega)\cap \mathcal{T}_j.
\end{equation}
Then define
\begin{equation}
P_t^{\varepsilon,(j)}(\Omega) : = \int_{S_t^j(\Omega)} P_t^\varepsilon(\Phi) \, \mathrm{d}\Phi.
\end{equation}
\end{deff}

\begin{lem} \label{na-lem-ptjabscts}
Let $t \in [0,T]$, $j \geq 1$ and $\varepsilon \geq 0$. Then $P_t^{\varepsilon,(j)}$ is absolutely continuous with respect to the Lebesgue measure on $U \times \mathbb{R}^3$.
\end{lem}
\begin{proof}
Let $j=1$ and $\Omega \subset U \times \mathbb{R}^3$ be measurable. Then by \eqref{na-eq-ptfull} we have
\begin{align}
P_t^{\varepsilon,(1)}(\Omega) & = \int_{S_t^1(\Omega)} P_t^\varepsilon(\Phi) \, \mathrm{d}\Phi \nonumber \\
&= \int_0^t \int_{\mathbb{S}^2} \int_{\mathbb{R}^3} \int_U \int_{\mathbb{R}^3} \exp \Bigg(- \int_\tau^t \int_{\mathbb{S}^2} \int_{\mathbb{R}^3} g_\sigma(x(\sigma)+\varepsilon\nu',\bar{v})  [(v(\tau)-\bar{v})\cdot \nu']_+ \, \mathrm{d}\bar{v} \, \mathrm{d}\nu' \, \mathrm{d}\sigma  \Bigg)\nonumber \\  &  \qquad
P_\tau^{\varepsilon,(0)}(x_0+\tau v_0,v_0)g_\tau(x_0+\tau v_0 + \varepsilon\nu,v') 
[(v_0 - v')\cdot\nu]_+ \mathbbm{1}_{(x(t),v(t)) \in \Omega} \, \mathrm{d}v_0 \, \mathrm{d}x_0 \, \mathrm{d}v' \, \mathrm{d}\nu \, \mathrm{d}\tau.\label{na-eq-ptepOmega}
\end{align}
We now introduce a change of coordinates $(\nu,x_0,v_0,v') \mapsto (\nu,x,v,\bar{w})$ defined by,
\begin{align*}
v  = v_0 + \nu (v'-v_0)\cdot \nu; \quad x = x_0 + \tau v_0 + (t-\tau)v; \quad
\bar{w} = v' - \nu (v'-v_0 )\cdot \nu.
\end{align*}
Computing the Jacobian of this transformation,
\begin{equation*}
\begin{pmatrix}
\textrm{Id} & 0 & 0 & 0 \\
& \textrm{Id} & & \\
& 0 & \textrm{Id} - \nu \otimes \nu & \nu \otimes \nu  \\
& 0 & \nu \otimes \nu &  \textrm{Id} - \nu \otimes \nu
\end{pmatrix}
\end{equation*}
where the non-filled entries are not required to compute the determinant. We now see that the bottom right 2x2 matrix has determinant $-1$ and hence the absolute value of the determinant of the Jacobi matrix is 1. We note that under this transformation for $t\geq \tau$, $(x(t),v(t)) = (x,v)$ and for $\tau \leq \sigma \leq t $, $x(\sigma) = x-(t-\sigma)v$. Hence with this transformation \eqref{na-eq-ptepOmega} becomes
\begin{align}
 P_t^{\varepsilon,(1)}(\Omega) 
& = \int_\Omega \int_0^t \int_{\mathbb{R}^3} \int_{\mathbb{S}^2} \exp \left(- \int_\tau^t \int_{\mathbb{S}^2} \int_{\mathbb{R}^3} g_\sigma(x-(t-\sigma)v+\varepsilon\nu',\bar{v}) [(v-\bar{v})\cdot \nu']_+ \, \mathrm{d}\bar{v} \, \mathrm{d}\nu' \, \mathrm{d}\sigma  \right) \nonumber \\
& \qquad \qquad P_\tau^{\varepsilon,(0)}(x-(t-\tau)v,w')g_\tau(x-(t-\tau)v+\varepsilon\nu,\bar{w}')
 [(v - \bar{w})\cdot\nu]_+  \, \mathrm{d}\bar{w} \, \mathrm{d}\nu \, \mathrm{d}\tau \, \mathrm{d}x \, \mathrm{d}v,\label{na-eq-Pt1mble}
\end{align}
where $w'=v+\nu (\bar{w}-v) \cdot \nu$ and $\bar{w}'=\bar{w}-\nu(\bar{w}-v)\cdot\nu$. Hence we see that if the Lebesgue measure of $\Omega$ is zero then $P_t^{\varepsilon,(1)}(\Omega)$ equals zero also. For $j \geq 1$ we use a similar approach using the iterative formula for $P_t^\varepsilon(\Phi)$.
\end{proof}

\begin{rmk}
By the Radon-Nikodym theorem it follows that $P_t^{\varepsilon,(j)}$ has a density, which we also denote by $P_t^{\varepsilon,(j)}$. Hence for any $\Omega \subset U \times \mathbb{R}^3$ we have that
\[ \int_{\Omega} P_t^{\varepsilon,(j)}(x,v)\, \mathrm{d}x \, \mathrm{d}v = \int_{S_t^{j}(\Omega)} P_t^\varepsilon(\Phi) \, \mathrm{d}\Phi, \]
This implies that for almost all $(x,v) \in U \times \mathbb{R}^3$ we have
\[ P_t^{\varepsilon,(j)}(x,v) = \int_{S_t^{j}(x,v)} P_t^\varepsilon(\Phi) \, \mathrm{d}\Phi. \]

\end{rmk}

\begin{prop} \label{na-prop-ptjintform}
For any $\varepsilon \geq 0$, $j,t \geq 0$ for almost all $(x,v) \in U \times \mathbb{R}^3$,
\begin{equation} \label{na-eq-Ptjintform}
P_t^{\varepsilon,(j+1)}(x,v) = \int_0^t (U^\varepsilon(t,\tau) B^\varepsilon(\tau) P_\tau^{\varepsilon,(j)})(x,v) \, \mathrm{d}\tau.
\end{equation}
\end{prop}
\begin{proof}
First consider $j=0$. We prove that for any $\Omega \subset U \times \mathbb{R}^3$ measurable,
\begin{equation} \label{na-eq-Ptjintformomega}
\int_\Omega P_t^{\varepsilon,(1)}(x,v) \, \mathrm{d}x \, \mathrm{d}v = \int_\Omega \int_0^t (U^\varepsilon(t,\tau) B^\varepsilon(\tau) P_\tau^{\varepsilon,(0)})(x,v) \, \mathrm{d}\tau \, \mathrm{d}x \, \mathrm{d}v.
\end{equation}
By the definition of $B^\varepsilon$ \eqref{na-eq-Bdeff} and $U^\varepsilon$ \eqref{na-eq-Uepdeff} we have, for $w'=v+\nu(\bar{w}-\nu)\cdot \nu$ and $\bar{w}' = \bar{w} - \nu(\bar{w} - v) \cdot \nu$,
\begin{align*}
\int_\Omega \int_0^t & (U^\varepsilon(t,\tau) B^\varepsilon(\tau) P_\tau^{\varepsilon,(0)})(x,v) \, \mathrm{d}\tau \, \mathrm{d}x \mathrm{d}v \\
& = \int_\Omega \int_0^t U^\varepsilon(t,\tau) \int_{\mathbb{S}^2} \int_{\mathbb{R}^3} P_\tau^{\varepsilon,(0)}(x,w')g_\tau(x+\varepsilon \nu, \bar{w}')[(v-\bar{w}) \cdot \nu]_+ \, \mathrm{d}\bar{
w} \, \mathrm{d}\nu \, \mathrm{d}\tau \, \mathrm{d}x \, \mathrm{d}v \\
& = \int_\Omega \int_0^t \exp \left( - \int_\tau^t \int_{\mathbb{S}^2} \int_{\mathbb{R}^3} g_\sigma (x+\varepsilon\nu' -(t-\sigma)v,\bar{v})[(v-\bar{v})\cdot \nu']_+ \, \mathrm{d}\bar{v} \, \mathrm{d}\nu' \mathrm{d}\sigma   \right) \\
& \qquad
\int_{\mathbb{S}^2} \int_{\mathbb{R}^3} P_\tau^{\varepsilon,(0)}(x-(t-\tau)v,w')g_\tau(x-(t-\tau)v+\varepsilon \nu, \bar{w}') 
[(v-\bar{w}) \cdot \nu]_+ \, \mathrm{d}\bar{
w} \, \mathrm{d}\nu \, \mathrm{d}\tau \, \mathrm{d}x \, \mathrm{d}v.
\end{align*}
Hence by \eqref{na-eq-Pt1mble} we notice that this is equal to the right hand side of \eqref{na-eq-Ptjintformomega}. Hence for $j=0$ \eqref{na-eq-Ptjintform} holds for almost all $(x,v) \in U \times \mathbb{R}^3$. For $j \geq 1$ one takes a similar approach.

\end{proof}
\begin{prop} \label{na-prop-ptjsum}
For almost all $(x,v) \in U \times \mathbb{R}^3$, for any $\varepsilon,t\geq 0$,
\[ \sum_{j=0}^\infty P_t^{\varepsilon,(j)}(x,v)= f_t^\varepsilon(x,v). \]
\end{prop}
\begin{proof}
In the proof of proposition~\ref{na-prop-linboltz} we saw, by using
\cite{arlotti14}, that $f_t^\varepsilon = V^\varepsilon(t,0)f_0$. By the proof of \cite[theorem 2.1]{arlotti14}, we have,
\begin{equation} \label{na-eq-ftsum}
f_t^\varepsilon = V^\varepsilon(t,0)f_0 = \sum_{j=0}^\infty V_j^\varepsilon(t,0)f_0,
\end{equation}
where $V_0^\varepsilon = U^\varepsilon$, and for $j \geq 0$,
\[ V_{j+1}^\varepsilon(t,s) = \int_s^t V_j^\varepsilon(t,r)B^\varepsilon(r)U^\varepsilon(r,s) \, \mathrm{d}r. \]
We notice by proposition~\ref{na-prop-Pj0}, $V_0^\varepsilon(t,0)f_0  = U^\varepsilon(t,0)f_0=P_t^{\varepsilon,(0)}$. Hence by proposition~\ref{na-prop-ptjintform},
\begin{align*}
 V_1^\varepsilon(t,0)f_0 & = \int_0^t V_0^\varepsilon(t,r)B^\varepsilon(r)U^\varepsilon(r,0) f_0 \, \mathrm{d}r = \int_0^t U^\varepsilon(t,r)B^\varepsilon(r)U^\varepsilon(r,0) f_0 \, \mathrm{d}r
\\
& = \int_0^t U^\varepsilon(t,r)B^\varepsilon(r)P_r^{\varepsilon,(0)} \, \mathrm{d}r = P_t^{\varepsilon,(1)}.
\end{align*}
Further by Fubini's theorem and proposition~\ref{na-prop-ptjintform},
\begin{align*}
V_2^\varepsilon(t,0)f_0 & = \int_0^t V_1^\varepsilon(t,r)B^\varepsilon(r)U^\varepsilon(r,0) f_0 \, \mathrm{d}r\\
& =  \int_0^t \int_r^t U^\varepsilon(t,t_1)B^\varepsilon(t_1)U^\varepsilon(t_1,r) B^\varepsilon(r)U^\varepsilon(r,0) f_0 \, \mathrm{d}t_1 \, \mathrm{d}r  \\
& = \int_0^t \int_0^{t_1}  U^\varepsilon(t,t_1)B^\varepsilon(t_1)U^\varepsilon(t_1,r) B^\varepsilon(r)U^\varepsilon(r,0) f_0  \, \mathrm{d}r \, \mathrm{d}t_1 \\
& =  \int_0^t U^\varepsilon(t,t_1)B^\varepsilon(t_1)  \int_0^{t_1}  U^\varepsilon(t_1,r) B^\varepsilon(r)U^\varepsilon(r,0) f_0  \, \mathrm{d}r \, \mathrm{d}t_1 \\
& =  \int_0^t U^\varepsilon(t,t_1)B^\varepsilon(t_1)  P_{t_1}^{\varepsilon,(1)} \, \mathrm{d}t_1 = P_t^{\varepsilon,(2)}.
\end{align*}
Similarly we can see that for all $j \geq 0$, $V_j^\varepsilon(t,0)f_0 = P_t^{\varepsilon,(j)}$. Hence by \eqref{na-eq-ftsum} the required result holds.
\end{proof}

We can now prove the existence part  of theorem~\ref{na-thm-id}. The next lemmas  provide the additional properties.
\begin{prop} \label{na-prop-idpart1}
For any $\varepsilon \geq 0$ and $t \in [0,T]$, $P_t^\varepsilon \in L^1(\mathcal{MT})$ and $P^\varepsilon$ is a solution to \eqref{na-eq-id}. Furthermore  \eqref{na-eq-Pt1st} and \eqref{na-eq-ftptcon} hold.
\end{prop}
\begin{proof}
Firstly let $\Omega \subset U \times \mathbb{R}^3$ be measurable. By proposition~\ref{na-prop-ptjsum}, and since each $P_t^{\varepsilon,(j)}$ is positive, the monotone convergence theorem and definition~\ref{na-deff-ptj},
\begin{align} \label{na-eq-ftptphiprop}
 \int_\Omega f_t^\varepsilon(x,v) \, \mathrm{d}x \, \mathrm{d}v & = \int_\Omega \sum_{j=0}^\infty P_t^{\varepsilon,(j)}(x,v) \, \mathrm{d}x \, \mathrm{d}v  = \sum_{j=0}^\infty \int_\Omega  P_t^{\varepsilon,(j)}(x,v) \, \mathrm{d}x \, \mathrm{d}v \nonumber \\
 & = \sum_{j=0}^\infty \int_{S_t^j(\Omega)} P_t^\varepsilon(\Phi) \, \mathrm{d}\Phi    = \int_{S_t(\Omega)} P_t^\varepsilon(\Phi) \, \mathrm{d}\Phi.
\end{align}
Hence for $\Omega = U \times \mathbb{R}^3$ we have
\begin{equation} \label{na-eq-ptepftepint}
\int_{\mathcal{MT}} P_t^\varepsilon(\Phi) \, \mathrm{d}\Phi = \int_{U \times \mathbb{R}^3} f_t^\varepsilon(x,v) \, \mathrm{d}x \, \mathrm{d}v < \infty.
\end{equation}
Thus $P_t^\varepsilon \in L^1(\mathcal{MT})$. Now we check that $P_t^\varepsilon(\Phi)$ indeed solves \eqref{na-eq-id}. For $\Phi \in \mathcal{T}_0$ noting that $x(t)=x_0+tv_0$ and $v(t)=v_0$, we have for $t\geq 0$
\begin{align} \label{na-eq-Ptepphiform}
P_t^\varepsilon(\Phi) & = P_t^{\varepsilon,(0)}(x(t),v(t)) = U^\varepsilon(t,0)f_0(x(t),v(t)) \nonumber \\
& =  \exp \left( -\int_0^t  \int_{\mathbb{S}^{2}} \int_{\mathbb{R}^3} g_\sigma(x(\sigma)+\varepsilon\nu,\bar{v})[(v_0-\bar{v})\cdot \nu ]_+ \, \mathrm{d}\bar{v} \, \mathrm{d}\nu \, \mathrm{d}\sigma \right)f_0(x_0,v_0).
\end{align}
We see that this gives the required initial value at $t=0$, that it is differentiable with respect to $t$ and differentiates to give the required term. Now consider $\Phi \in \mathcal{T}_j$ for $j \geq 1$. By definition \eqref{na-eq-ptfull} we see that for $t<\tau$, $P_t^\varepsilon(\Phi)=0$ and that $P_\tau^\varepsilon(\Phi)$ has the required form. We also see that for $t> \tau$ we have
\begin{align*}
\partial_t P_t^\varepsilon(\Phi) & = \partial_t \Bigg(  \exp \left(- \int_\tau^t \int_{\mathbb{S}^2} \int_{\mathbb{R}^3} g_\sigma(x(\sigma)+\varepsilon\nu,\bar{v}) [(v(\tau)-\bar{v})\cdot \nu']_+ \, \mathrm{d}\bar{v} \, \mathrm{d}\nu' \, \mathrm{d}\sigma  \right) \nonumber \\ & \quad \quad  P_\tau (\bar{\Phi})g_\tau(x(\tau)+\varepsilon\nu,v')[(v(\tau^-)-v')\cdot \nu]_+ \Bigg) \\
& = -\int_{\mathbb{S}^2} \int_{\mathbb{R}^3} g_t(x(t)+\varepsilon\nu,\bar{v}) [(v(\tau)-\bar{v})\cdot \nu']_+ \, \mathrm{d}\bar{v} \, \mathrm{d}\nu' \\
& \quad \quad \exp \left(- \int_\tau^t \int_{\mathbb{S}^2} \int_{\mathbb{R}^3} g_\sigma(x(\sigma)+\varepsilon\nu,\bar{v}) [(v(\tau)-\bar{v})\cdot \nu']_+ \, \mathrm{d}\bar{v} \, \mathrm{d}\nu' \, \mathrm{d}\sigma  \right) \nonumber \\
& \quad \quad  P_\tau (\bar{\Phi})g_\tau(x(\tau)+\varepsilon\nu,v')[(v(\tau^-)-v')\cdot \nu]_+ \\
& =- L_t^\varepsilon(\Phi)P_t^\varepsilon(\Phi).
\end{align*}
We now prove \eqref{na-eq-Pt1st}. Let $K>0$ be as in proposition~\ref{na-prop-linboltz}. By a similar argument to \eqref{na-eq-ftptphiprop}, by using the same method as the proof of lemma~\ref{na-lem-ptjabscts} and by proposition~\ref{na-prop-linboltz} we have
\begin{align*}
&\int_{\mathcal{MT}}  P_t^\varepsilon(\Phi)(1+|v(\tau)|) \, \mathrm{d} \Phi  = \sum_{j=0}^\infty \int_{\mathcal{T}_j} P_t^\varepsilon(\Phi)(1+|v(\tau)|) \, \mathrm{d}\Phi  = \sum_{j=0}^\infty \int_{U \times \mathbb{R}^3}  P_t^{\varepsilon,(j)}(x,v)(1+|v|) \, \mathrm{d}x \, \mathrm{d}v \\
  =&  \int_{U \times \mathbb{R}^3} \sum_{j=0}^\infty P_t^{\varepsilon,(j)}(x,v)(1+|v|) \, \mathrm{d}x \, \mathrm{d}v  = \int_{U \times \mathbb{R}^3} f_t^\varepsilon(x,v)(1+|v|) \, \mathrm{d}x \, \mathrm{d}v  \leq K.
\end{align*}
We see that \eqref{na-eq-ftptcon} has been proved in \eqref{na-eq-ftptphiprop}.

\end{proof}
The only remaining parts of theorem~\ref{na-thm-id} are that $P_t^\varepsilon$ is a probability measure on $\mathcal{MT}$ and that \eqref{na-eq-ptepuni} holds. We remark here that for the corresponding result for $P_t$ in the autonomous case being a probability measure resulted from the fact that we were able to prove that the semigroup defining the solution of the autonomous linear Boltzmann equation was honest and hence conserved mass. However in this non-autonomous case we have not been able to find equivalent honesty results. Therefore we prove that $P_t^\varepsilon$ is a probability measure explicitly by showing that $\int_{\mathcal{MT}} P_t^\varepsilon(\Phi) \, \mathrm{d}\Phi$ is differentiable with respect to $t$ and has derivative zero.

To that aim, the following lemmas calculate various limits that are required to show that $\int_{\mathcal{MT}} P_t^\varepsilon(\Phi) \, \mathrm{d}\Phi$ is differentiable, which is finally proved in lemma~\ref{na-lem-prob+} and \ref{na-lem-prob-}.

\begin{lem}
Let $\varepsilon \geq 0$ and $\Psi \in \mathcal{MT}$. Then for $t \geq \tau$,
\begin{align} \label{na-eq-probhelp11}
\frac{1}{h} & \int_t^{t+h} \int_{\mathbb{S}^2} \int_{\mathbb{R}^3} | P_s^\varepsilon(\Psi) g_s(x(s)+\varepsilon\bar{\nu},\bar{v}) -P_t^\varepsilon(\Psi) g_t(x(t)+\varepsilon\bar{\nu},\bar{v})  | [(v(\tau)-\bar{v})\cdot \bar{\nu}]_+ \, \mathrm{d}\bar{v} \, \mathrm{d}\bar{\nu} \, \mathrm{d}s \nonumber \\
& \qquad \qquad  \to 0 \textrm{ as } h \downarrow 0.
\end{align}
And for $t>\tau$
\begin{align} \label{na-eq-probhelp12}
\frac{1}{h} & \int_{t-h}^{t} \int_{\mathbb{S}^2} \int_{\mathbb{R}^3} | P_s^\varepsilon(\Psi) g_s(x(s)+\varepsilon\bar{\nu},\bar{v}) -P_t^\varepsilon(\Psi) g_t(x(t)+\varepsilon\bar{\nu},\bar{v})  | [(v(\tau)-\bar{v})\cdot \bar{\nu}]_+ \, \mathrm{d}\bar{v} \, \mathrm{d}\bar{\nu} \, \mathrm{d}s \nonumber \\
& \qquad \qquad  \to 0 \textrm{ as } h \downarrow 0.
\end{align}
\end{lem}
\begin{proof}
We begin with \eqref{na-eq-probhelp11}. Let $t \geq \tau$. Firstly,
\begin{align} \label{na-eq-psgsptgt}
| P_s^\varepsilon(\Psi) &  g_s(x(s)+\varepsilon\bar{\nu},\bar{v}) -P_t^\varepsilon(\Psi) g_t(x(t)+\varepsilon\bar{\nu},\bar{v})  | \nonumber \\
& \leq P_t^\varepsilon(\Psi) | g_s(x(s)+\varepsilon\bar{\nu},\bar{v}) - g_t(x(t)+\varepsilon\bar{\nu},\bar{v})  |  + g_s(x(s)+\varepsilon\bar{\nu},\bar{v}) |P_t^\varepsilon(\Psi) - P_s^\varepsilon(\Psi)|.
\end{align}
Noting that for $s,t\geq\tau$, $|x(s)-x(t)| = |t-s||v(\tau)|$ it follows by lemma~\ref{na-lem-g0diffbound}, with $R=h^{-\alpha/6}$ that
\begin{align} \label{na-eq-ptgtgs}
\frac{1}{h} & \int_{t}^{t+h} \int_{\mathbb{S}^2} \int_{\mathbb{R}^3}  P_t^\varepsilon(\Psi) | g_s(x(s)+\varepsilon\bar{\nu},\bar{v}) - g_t(x(t)+\varepsilon\bar{\nu},\bar{v})  |  [(v(\tau)-\bar{v})\cdot \bar{\nu}]_+ \, \mathrm{d}\bar{v} \, \mathrm{d}\bar{\nu} \, \mathrm{d}s \nonumber \\
& \leq \frac{1}{h}  \int_{t}^{t+h} C P_t^\varepsilon(\Psi) (1+|v(\tau)|) \left( h^{\alpha/6} + h^{-5\alpha/6}|t-s|^\alpha (1+|v(\tau)|^\alpha)  \right) \, \mathrm{d}s \nonumber \\
& = C P_t^\varepsilon(\Psi) (1+|v(\tau)|) \left( h^{\alpha/6}  + h^{-5\alpha/6} (1+|v(\tau)|^\alpha) \frac{1}{h}  \int_{t-h}^{t}  |t-s|^\alpha \, \mathrm{d}s  \right) \nonumber \\
& =  C P_t^\varepsilon(\Psi)(1+|v(\tau)|) \left( h^{\alpha/6}  + h^{-5\alpha/6} (1+|v(\tau)|^\alpha)  \frac{h^\alpha}{\alpha + 1}  \right) \nonumber \\
& = C P_t^\varepsilon(\Psi) (1+|v(\tau)|) \left( 1  + \frac{1+|v(\tau)|^\alpha}{\alpha + 1}    \right) h^{\alpha/6} 
\to 0 \textrm{ as } h \to 0.
\end{align}
Now let $\delta > 0$. By the proof of proposition~\ref{na-prop-idpart1} we know that $P_s^\varepsilon(\Phi)$ is differentiable with respect to $s$ and hence continuous for $s \geq \tau$. So for $h$ sufficiently small and any $s \in [t,t+h]$
\[ |P_t^\varepsilon(\Phi) - P_s^\varepsilon(\Phi)| < \frac{\delta}{\pi M_g (1+|v(\tau)|)}. \]
Thus
\begin{align*}
\frac{1}{h} & \int_{t}^{t+h} \int_{\mathbb{S}^2} \int_{\mathbb{R}^3}  g_s(x(s)+\varepsilon\bar{\nu},\bar{v}) |P_t^\varepsilon(\Psi) - P_s^\varepsilon(\Psi)|  [(v(\tau)-\bar{v})\cdot \bar{\nu}]_+ \, \mathrm{d}\bar{v} \, \mathrm{d}\bar{\nu} \, \mathrm{d}s  \\
& < \frac{\delta}{\pi M_g (1+|v(\tau)|)} \frac{1}{h} \int_{t}^{t+h} \int_{\mathbb{S}^2} \int_{\mathbb{R}^3}  g_s(x(s)+\varepsilon\bar{\nu},\bar{v})  [(v(\tau)-\bar{v})\cdot \bar{\nu}]_+ \, \mathrm{d}\bar{v} \, \mathrm{d}\bar{\nu} \, \mathrm{d}s \\
& < \frac{\delta}{\pi M_g (1+|v(\tau)|)} \frac{1}{h}  \int_{t}^{t+h} \pi M_g(1+|v(\tau)|) \, \mathrm{d}s = \delta.
\end{align*}
This, together with \eqref{na-eq-psgsptgt} and \eqref{na-eq-ptgtgs} proves \eqref{na-eq-probhelp11}.

The proof of \eqref{na-eq-probhelp12} is similar but we must exclude $t=\tau$ because in that case $P_s^\varepsilon(\Phi)$ in the integrand is always 0. For $t > \tau$ we take $h$ sufficiently small so that $t-h > \tau$ and hence $P_s^\varepsilon(\Phi)$ is continuous with respect to $s$ for $s \in [t-h,t]$. The result now follows by the same method as \eqref{na-eq-probhelp11}.
\end{proof}
\begin{deff} \label{na-deff-MTS}
For any $S \subset [0,T]$ define $\mathcal{MT}_S$ by
\begin{equation} \label{na-eq-MTS}
 \mathcal{MT}_S := \{ \Phi \in \mathcal{MT} : \tau \in S \}.
\end{equation}
And for $t \in [0,T]$ define
\begin{equation} \label{na-eq-MTt}
 \mathcal{MT}_t := \mathcal{MT}_{\{t\}} = \{ \Phi \in \mathcal{MT} : \tau =t \}.
\end{equation}
\end{deff}
\begin{lem} \label{na-lem-MTtzero}
For any $t \in (0,T]$, $\mathcal{MT}_t$ is a set of zero measure with respect to the Lebesgue measure on $\mathcal{MT}$.
\end{lem}
\begin{proof}
Let $t \in (0,T]$. Then $\mathcal{MT}_t \cap \mathcal{T}_0 = \emptyset $ since for any $\Phi \in \mathcal{T}_0$, $\tau =0$. Now for any $j \geq 1$, $\mathcal{MT}_t \cap \mathcal{T}_j$ is a set of co-dimension 1 in $\mathcal{T}_j$ (since one component, the final collision time, is fixed) and hence has zero measure. Since,
\[ \mathcal{MT}_t = \cup_{j \geq 1} \mathcal{MT}_t \cap \mathcal{T}_j \]
it follows that $\mathcal{MT}_t$ is \ set of zero measure.
\end{proof}

\begin{deff}
Let $\Psi \in \mathcal{MT}$. For $s \in (\tau,T] $, $\bar{\nu} \in \mathbb{S}^2$ and $\bar{v} \in \mathbb{R}^3$, when the context is clear let $\Psi' := \Psi \cup (s,\bar{\nu},\bar{v})$ denote the new tree formed by adding the collision $(s,\bar{\nu},\bar{v})$ to $\Psi$.
\end{deff}

\begin{lem} \label{na-lem-probhelp2}
Let $t \in (0,T]$ and $\varepsilon \geq 0$. Then for any $\Psi \in \mathcal{MT}$,
\begin{align} \label{na-eq-probhelp21}
L_t^\varepsilon(\Psi)& P_t^\varepsilon(\Psi)  \nonumber \\
& = \lim_{h \downarrow 0 } \frac{1}{h} \int_t^{t+h} \int_{\mathbb{S}^2} \int_{\mathbb{R}^3} \exp\left( - \int_{s}^{t+h} L_\sigma^\varepsilon(\Psi') \, \mathrm{d}\sigma \right) P_s^\varepsilon(\Psi) g_s(x(s)+\varepsilon\bar{\nu},\bar{v})  [(v(\tau)-\bar{v})\cdot \bar{\nu}]_+ \, \mathrm{d}\bar{v} \, \mathrm{d}\bar{\nu} \, \mathrm{d}s.
\end{align}
and for almost all $\Psi \in \mathcal{MT}$
\begin{align} \label{na-eq-probhelp22}
L_t^\varepsilon(\Psi)& P_t^\varepsilon(\Psi)  \nonumber \\
& = \lim_{h \downarrow 0 } \frac{1}{h} \int_{t-h}^{t} \int_{\mathbb{S}^2} \int_{\mathbb{R}^3} \exp\left( - \int_{s}^{t} L_\sigma^\varepsilon(\Psi') \, \mathrm{d}\sigma \right) P_s^\varepsilon(\Psi) g_s(x(s)+\varepsilon\bar{\nu},\bar{v})  [(v(\tau)-\bar{v})\cdot \bar{\nu}]_+ \, \mathrm{d}\bar{v} \, \mathrm{d}\bar{\nu} \, \mathrm{d}s.
\end{align}
\end{lem}
\begin{proof}
Let $t \in (0,T]$. We first prove \eqref{na-eq-probhelp21}. Let $\Psi \in \mathcal{MT}$. If $t<\tau$ then the left hand side is zero and the right side is zero also, since for $h$ sufficiently small $P_s^\varepsilon(\Psi) =0$ for all $s \in [t,t+h]$. Suppose $t \geq \tau$. Note that,
\begin{align*}
L_t^\varepsilon(\Psi)  P_t^\varepsilon(\Psi) & = \frac{1}{h} \int_{t}^{t+h} L_t^\varepsilon(\Psi)P_t^\varepsilon(\Psi) \,  \mathrm{d}s \\
& = \frac{1}{h} \int_{t}^{t+h} \int_{\mathbb{S}^2} \int_{\mathbb{R}^3} P_t^\varepsilon(\Psi) g_t(x(t)+\varepsilon\bar{\nu},\bar{v}) [(v(\tau)-\bar{v})\cdot \bar{\nu}]_+ \, \mathrm{d}\bar{v} \, \mathrm{d}\bar{\nu} \, \mathrm{d}s.
\end{align*}
Hence to prove \eqref{na-eq-probhelp21} we show that
\begin{align*}
 \frac{1}{h} & \int_t^{t+h} \int_{\mathbb{S}^2} \int_{\mathbb{R}^3} \bigg| \exp\left( - \int_{s}^{t+h} L_\sigma^\varepsilon(\Psi') \, \mathrm{d}\sigma \right) P_s^\varepsilon(\Psi) g_s(x(s)+\varepsilon\bar{\nu},\bar{v}) - P_t^\varepsilon(\Psi) g_t(x(t)+\varepsilon\bar{\nu},\bar{v}) \bigg| \nonumber \\
& \qquad \qquad [(v(\tau)-\bar{v})\cdot \bar{\nu}]_+ \, \mathrm{d}\bar{v} \, \mathrm{d}\bar{\nu} \, \mathrm{d}s \\
& \to 0 \textrm{ as } h \downarrow 0.
\end{align*}
Now,
\begin{align*}
&\bigg| \exp  \left( - \int_{s}^{t+h} L_\sigma^\varepsilon(\Psi') \, \mathrm{d}\sigma \right) P_s^\varepsilon(\Psi) g_s(x(s)+\varepsilon\bar{\nu},\bar{v}) - P_t^\varepsilon(\Psi) g_t(x(t)+\varepsilon\bar{\nu},\bar{v}) \bigg| \\
 \leq & | P_s^\varepsilon(\Psi) g_s(x(s)+\varepsilon\bar{\nu},\bar{v}) - P_t^\varepsilon(\Psi) g_t(x(t)+\varepsilon\bar{\nu},\bar{v}) |   \\
 & \quad+ \left( 1- \exp  \left( - \int_{s}^{t+h} L_\sigma^\varepsilon(\Psi') \, \mathrm{d}\sigma \right) \right)P_t^\varepsilon(\Psi) g_t(x(t)+\varepsilon\bar{\nu},\bar{v}).
\end{align*}
So by using \eqref{na-eq-probhelp11} it remains to prove that
\begin{align}  \label{na-eq-prob221}
I(h) & := \frac{1}{h}  \int_t^{t+h} \int_{\mathbb{S}^2} \int_{\mathbb{R}^3} \left( 1- \exp  \left( - \int_{s}^{t+h} L_\sigma^\varepsilon(\Psi') \, \mathrm{d}\sigma \right) \right)P_t^\varepsilon(\Psi)
g_t(x(t)+\varepsilon\bar{\nu},\bar{v}) [(v(\tau)-\bar{v})\cdot \bar{\nu}]_+ \, \mathrm{d}\bar{v} \, \mathrm{d}\bar{\nu} \, \mathrm{d}s \nonumber \\
& \to 0 \textrm{ as } h \downarrow 0.
\end{align}
Recall $\Psi' = \Psi \cup (s, \bar{\nu},\bar{v})$. Denote by $w$ the velocity of the root particle of $\Psi'$ after its final collision at $s$. Then,
\[ w = v(\tau) + \bar{\nu} (v(\tau) - \bar{v}) \cdot \bar{\nu} \]
Hence,
\[ |w| \leq |v(\tau)| + |v(\tau )-\bar{v}| \leq 2|v(\tau)| + |\bar{v}|.  \]
Thus,
\begin{align*}
\int_{s}^{t+h} L_\sigma^\varepsilon(\Psi') \, \mathrm{d}\sigma & = \int_{s}^{t+h} \int_{\mathbb{S}^2} \int_{\mathbb{R}^3}  g_\sigma (x(\Psi')(\sigma) + \varepsilon \nu_1,v_1)[(w-v_1)\cdot \nu_1]_+ \, \mathrm{d}\sigma \\
& \leq  \pi \int_{s}^{t+h}  \int_{\mathbb{R}^3}  \bar{g} (v_1)(|w|+|v_1|) \, \mathrm{d}v_1 \, \mathrm{d}\sigma
\leq \pi M_g(1 + |w|) (t+h-s) \\
& \leq h \pi M_g(1 + 2|v(\tau)| + |\bar{v}|).
\end{align*}
It follows that
\[ 1- \exp  \left( - \int_{s}^{t+h} L_\sigma^\varepsilon(\Psi') \, \mathrm{d}\sigma \right) \leq 1- \exp  \big( -  h \pi M_g(1 + 2|v(\tau)| + |\bar{v}|) \big).  \]
Hence
\begin{align} \label{na-eq-Ihprob}
I(h) & \leq \frac{1}{h} \int_t^{t+h} \int_{\mathbb{S}^2} \int_{\mathbb{R}^3} \left( 1- \exp  \big( -  h \pi M_g(1 + 2|v(\tau)| + |\bar{v}|) \big) \right) 
P_t^\varepsilon(\Psi) \bar{g}(\bar{v}) (|(v(\tau)| + |\bar{v}|) \, \mathrm{d}\bar{v} \, \mathrm{d}\bar{\nu} \, \mathrm{d}s \nonumber \\
& \leq  \pi P_t^\varepsilon(\Psi) \int_{\mathbb{R}^3} \left( 1- \exp  \big( -  h \pi M_g(1 + 2|v(\tau)| + |\bar{v}|) \big) \right) \bar{g}(\bar{v}) (|(v(\tau)| + |\bar{v}|) \, \mathrm{d}\bar{v}.
\end{align}
Let $\delta > 0$. By \eqref{na-eq-gl1assmp} there exists an $R>0$ such that,
\[ \int_{\mathbb{R}^3 \setminus B_R(0)} \bar{g}(\bar{v}) (1 + |\bar{v}|) \, \mathrm{d}\bar{v} < \frac{\delta}{\pi (1+P_t^\varepsilon(\Psi))(1+|v(\tau)|)}.  \]
Hence,
\begin{align}
\pi P_t^\varepsilon(\Psi) &  \int_{\mathbb{R}^3 \setminus B_R(0)} \left( 1- \exp  \big( -  h \pi M_g(1 + 2|v(\tau)| + |\bar{v}|) \big) \right) \bar{g}(\bar{v}) (|(v(\tau)| + |\bar{v}|) \, \mathrm{d}\bar{v} \nonumber \\
&  \leq \pi P_t^\varepsilon(\Psi)   \int_{\mathbb{R}^3 \setminus B_R(0)}  \bar{g}(\bar{v}) (|(v(\tau)| + |\bar{v}|) \, \mathrm{d}\bar{v} < \delta.\label{na-eq-prob222}
\end{align}
Further for $h$ sufficiently small,
\begin{align*}
\pi P_t^\varepsilon(\Psi)  & \int_{B_R(0)} \left( 1- \exp  \big( -  h \pi M_g(1 + 2|v(\tau)| + |\bar{v}|) \big) \right) \bar{g}(\bar{v}) (|(v(\tau)| + |\bar{v}|) \, \mathrm{d}\bar{v} \\
& \leq \pi P_t^\varepsilon(\Psi)  \left( 1- \exp  \big( -  h \pi M_g(1 + 2|v(\tau)| + R) \big) \right) \int_{B_R(0)} \bar{g}(\bar{v}) (|(v(\tau)| + R) \, \mathrm{d}\bar{v} \\
& \leq \pi M_g P_t^\varepsilon(\Psi) (|v(\tau)| + R) \left( 1- \exp  \big( -  h \pi M_g(1 + 2|v(\tau)| + R) \big) \right) < \delta.
\end{align*}
By substituting this and \eqref{na-eq-prob222} into \eqref{na-eq-Ihprob} we see that \eqref{na-eq-prob221} holds, which concludes the proof of \eqref{na-eq-probhelp21}.

We now prove \eqref{na-eq-probhelp22}, which we  prove holds for all $\Psi \in \mathcal{MT} \setminus \mathcal{MT}_t $. Indeed $\mathcal{MT}_t$ is a set of zero measure by lemma~\ref{na-lem-MTtzero}. Let $\Psi \in \mathcal{MT} \setminus \mathcal{MT}_t $. If $\tau > t $ then the left hand side of \eqref{na-eq-probhelp22} is zero and the right hand side is also zero since for any $s \in [t-h,t]$, $P_s^\varepsilon(\Phi) = 0$. If $t > \tau$ we use the same method as we used for \eqref{na-eq-probhelp21}, using \eqref{na-eq-probhelp12} instead of \eqref{na-eq-probhelp11}.
\end{proof}

\begin{lem} \label{na-lem-prob+}
Let $\varepsilon \geq 0$ and $t \in (0,T]$. Then $\partial_t^+ \int_\mathcal{MT} P_t^\varepsilon(\Phi ) \, \mathrm{d}\Phi$ exists and is equal to zero.
\end{lem}
\begin{proof}
Fix $\varepsilon \geq 0$ and $t \in (0,T]$. We want to show that
\[ \lim_{h \downarrow 0} \frac{1}{h} \int_{\mathcal{MT}} P_{t+h}^\varepsilon(\Phi) - P_t^\varepsilon(\Phi) \, \mathrm{d}\Phi =0. \]
For $\mathcal{MT}_S$ as defined in definition~\ref{na-deff-MTS} and $h>0$ we have,
\begin{align*}
&\frac{1}{h}\int_{\mathcal{MT}}   P_{t+h}^\varepsilon(\Phi)   - P_t^\varepsilon(\Phi) \, \mathrm{d}\Phi \\
 =& \frac{1}{h} \int_{\mathcal{MT}_{[0,t]}} P_{t+h}^\varepsilon(\Phi) - P_t^\varepsilon(\Phi) \, \mathrm{d}\Phi + \frac{1}{h} \int_{\mathcal{MT}_{(t,t+h]}} P_{t+h}^\varepsilon(\Phi) - P_t^\varepsilon(\Phi) \, \mathrm{d}\Phi\\&
 + \frac{1}{h} \int_{\mathcal{MT}_{(t+h,T]}} P_{t+h}^\varepsilon(\Phi) - P_t^\varepsilon(\Phi) \, \mathrm{d}\Phi .
\end{align*}
We show that each of these terms converges and that their sum is zero. Firstly,
\begin{align} \label{na-eq-+1}
\frac{1}{h} \int_{\mathcal{MT}_{(t+h,T]}} P_{t+h}^\varepsilon(\Phi) - P_t^\varepsilon(\Phi) \, \mathrm{d}\Phi = \frac{1}{h} \int_{\mathcal{MT}_{(t+h,T]}} 0 \, \mathrm{d}\Phi  = 0.
\end{align}
Now note that for any $h>0$, $\int_t^{t+h} L_s^\varepsilon(\Phi) \, \mathrm{d}s \leq h \pi M_g (1+|v(\tau)|)$. Hence for any $h>0$,
\begin{align*}
\frac{1}{h} &   \left| \exp \left( - \int_t^{t+h} L_s^\varepsilon(\Phi) \, \mathrm{d}s \right) - 1 \right| \leq  \frac{1}{h}  \big( 1-  \exp \left( -  h \pi M_g (1+|v(\tau)|) \right)  \big)  \\
& \leq \frac{1}{h}  \big( 1- 1  -  h \pi M_g (1+|v(\tau)|)   \big) = \pi M_g (1+|v(\tau)|).
\end{align*}
By \eqref{na-eq-Pt1st} it follows that
\begin{align*}
 \int_{\mathcal{MT}_{[0,t]}} &  \frac{1}{h}  \left| \exp \left( - \int_t^{t+h} L_s^\varepsilon(\Phi) \, \mathrm{d}s \right) - 1 \right| P_t^\varepsilon(\Phi) \, \mathrm{d}\Phi
  \leq \int_{\mathcal{MT}} \pi M_g(1+|v(\tau)|) P_t^\varepsilon(\Phi) \, \mathrm{d}\Phi \leq \pi M_g K < \infty.
\end{align*}
Hence by the dominated convergence theorem and the fact that for any $\Phi$ with $\tau > t$, $P_t^\varepsilon(\Phi) = 0$,
\begin{align}
&\frac{1}{h}  \int_{\mathcal{MT}_{[0,t]}} P_{t+h}^\varepsilon(\Phi)  - P_t^\varepsilon(\Phi) \, \mathrm{d}\Phi \nonumber \\
 = &\frac{1}{h}  \int_{\mathcal{MT}_{[0,t]}} \exp \left( - \int_t^{t+h} L_s^\varepsilon(\Phi) \, \mathrm{d}s \right)P_{t}^\varepsilon(\Phi) - P_t^\varepsilon(\Phi) \, \mathrm{d}\Phi \nonumber \\
 = & \int_{\mathcal{MT}_{[0,t]}} \frac{1}{h}  \left( \exp \left( - \int_t^{t+h} L_s^\varepsilon(\Phi) \, \mathrm{d}s \right) - 1 \right) P_t^\varepsilon(\Phi) \, \mathrm{d}\Phi \nonumber \\
 \xrightarrow{h \downarrow 0} &\int_{\mathcal{MT}_{[0,t]}} \partial_\sigma|_{\sigma=t} \exp \left( - \int_t^{\sigma} L_s^\varepsilon(\Phi) \, \mathrm{d}s \right) P_t^\varepsilon(\Phi) \, \mathrm{d}\Phi = - \int_{\mathcal{MT}}  L_t^\varepsilon(\Phi)  P_t^\varepsilon(\Phi) \, \mathrm{d}\Phi.\label{na-eq-+2}
\end{align}
Now
\begin{align*}
&\frac{1}{h}  \int_{\mathcal{MT}_{(t,t+h]}} P_{t+h}^\varepsilon(\Phi) - P_t^\varepsilon(\Phi) \, \mathrm{d}\Phi  \\
 =& \frac{1}{h} \int_{\mathcal{MT}_{(t,t+h]}} P_{t+h}^\varepsilon(\Phi)  \, \mathrm{d}\Phi = \frac{1}{h} \int_{\mathcal{MT}_{(t,t+h]}} \exp\left( - \int_\tau^{t+h} L_\sigma^\varepsilon(\Phi) \, \mathrm{d}\sigma \right) P_{\tau}^\varepsilon(\Phi)  \, \mathrm{d}\Phi \\
 =&\frac{1}{h} \int_{\mathcal{MT}_{(t,t+h]}} \exp\left( - \int_\tau^{t+h} L_\sigma^\varepsilon(\Phi) \, \mathrm{d}\sigma \right)P_\tau^\varepsilon(\bar{\Phi})g_\tau(x(\tau)+\varepsilon\nu,v')[(v(\tau^-) - v')\cdot \nu]_+   \, \mathrm{d}\Phi \\
 =& \int_{\mathcal{MT}} \frac{1}{h} \int_{t}^{t+h} \int_{\mathbb{S}^2} \int_{\mathbb{R}^3}  \exp\left( - \int_s^{t+h} L_\sigma^\varepsilon(\Psi') \, \mathrm{d}\sigma \right)P_s^\varepsilon(\Psi)
 g_s(x(s)+\varepsilon\nu,\bar{v})[(v(\tau) - \bar{v})\cdot \bar{\nu}]_+ \, \mathrm{d}\bar{v} \, \mathrm{d}\bar{\nu} \, \mathrm{d}s  \, \mathrm{d}\Psi.
\end{align*}
Hence by the dominated convergence theorem and \eqref{na-eq-probhelp21},
\begin{equation} \label{na-eq-+3}
\lim_{h \downarrow 0} \frac{1}{h}  \int_{\mathcal{MT}_{(t,t+h]}} P_{t+h}^\varepsilon(\Phi) - P_t^\varepsilon(\Phi) \, \mathrm{d}\Phi = \int_{\mathcal{MT}} L_t^\varepsilon(\Phi) P_t^\varepsilon(\Phi).
\end{equation}
Combining \eqref{na-eq-+1},\eqref{na-eq-+2} and \eqref{na-eq-+3} we see that the limit indeed exists and is equal to zero, proving the lemma.
\end{proof}

\begin{lem}\label{na-lem-prob-}
Let $\varepsilon\geq 0$ and $t \in (0,T]$. Then $\partial_t^-\int_\mathcal{MT} P_t^\varepsilon(\Phi)$ exists and is equal to zero.
\end{lem}
\begin{proof}
Fix $\varepsilon \geq 0$ and $t \in (0,T]$. We show that
\[ \lim_{h \downarrow 0 } \frac{1}{h}\int_{\mathcal{MT}} P_t^\varepsilon(\Phi) - P_{t-h}^\varepsilon(\Phi) \, \mathrm{d}\Phi  =0.  \]
As in  lemma~\ref{na-lem-prob+} note that,
\begin{align*}
&\frac{1}{h} \int_{\mathcal{MT}} P_t^\varepsilon(\Phi) - P_{t-h}^\varepsilon(\Phi) \, \mathrm{d}\Phi \\
  =&\frac{1}{h} \int_{\mathcal{MT}_{[0,t-h]}} P_t^\varepsilon(\Phi) - P_{t-h}^\varepsilon(\Phi) \, \mathrm{d}\Phi + \frac{1}{h} \int_{\mathcal{MT}_{(t-h,t]}} P_t^\varepsilon(\Phi) - P_{t-h}^\varepsilon(\Phi) \, \mathrm{d}\Phi  \\&+ \frac{1}{h} \int_{\mathcal{MT}_{(t,T]}} P_t^\varepsilon(\Phi) - P_{t-h}^\varepsilon(\Phi) \, \mathrm{d}\Phi.
\end{align*}
We again show each limit exists and the sum is zero. Firstly
\begin{equation} \label{na-eq--1}
\frac{1}{h} \int_{\mathcal{MT}_{(t,T]}} P_t^\varepsilon(\Phi) - P_{t-h}^\varepsilon(\Phi) \, \mathrm{d}\Phi = 0.
\end{equation}
By lemma~\ref{na-lem-MTtzero}, \eqref{na-eq-probhelp22} and the dominated convergence theorem we have,
\begin{align} \label{na-eq--2}
\frac{1}{h} &  \int_{\mathcal{MT}_{(t-h,t]}} P_t^\varepsilon(\Phi) - P_{t-h}^\varepsilon(\Phi) \, \mathrm{d}\Phi = \frac{1}{h} \int_{\mathcal{MT}_{(t-h,t)}} P_t^\varepsilon(\Phi)  \, \mathrm{d}\Phi \nonumber \\
& = \frac{1}{h} \int_{\mathcal{MT}_{(t-h,t)}} \exp\left(-\int_\tau^t L_\sigma^\varepsilon(\Phi) \, \mathrm{d}\sigma \right)  P_\tau^\varepsilon(\bar{\Phi})g_\tau(x(\tau)+\varepsilon\nu,v')[(v(\tau^-)-v')\cdot \nu]_+  \, \mathrm{d}\Phi \nonumber  \\
& =  \int_{\mathcal{MT}} \frac{1}{h} \int_{t-h}^t \int_{\mathbb{S}^2} \int_{\mathbb{R}^3} \exp\left(-\int_s^t L_\sigma^\varepsilon(\Psi') \, \mathrm{d}\sigma \right)  P_s^\varepsilon(\Psi)  g_s(x(s)+\varepsilon\bar{\nu},\bar{v})[(v(\tau)-\bar{v})\cdot \bar{\nu}]_+ \, \mathrm{d}\bar{v} \, \mathrm{d}\bar{\nu} \, \mathrm{d}s \, \mathrm{d}\Psi \nonumber \\
& = \int_{\mathcal{MT}} L_t^\varepsilon(\Psi) P_t^\varepsilon(\Psi) \, \mathrm{d}\Psi.
\end{align}
Now for the final term we first prove that for any $\Phi \in \mathcal{MT}_{[0,t)}$
\begin{equation} \label{na-eq-prob3}
\lim_{h \downarrow 0} \frac{1}{h} P_{t-h}^\varepsilon(\Phi) \left( \exp\left(-\int_{t-h}^t L_\sigma^\varepsilon(\Phi) \, \mathrm{d}\sigma \right)   -1 \right)  = -L_t^\varepsilon(\Phi)P _t^\varepsilon(\Phi).
\end{equation}
To this aim fix $\Phi \in \mathcal{MT}_{[0,t)}$. Then $\tau < t$. Let $h$ sufficiently small so that $t-h>\tau$. Then since $P_s^\varepsilon(\Phi)$ is continuous for $s \in [\tau,T]$ we have that $P_{t-h}^\varepsilon(\Phi)$ converges to $P_{t}^\varepsilon(\Phi) $ as $h$ tends to zero. Further,
\begin{align*}
\lim_{h \downarrow 0} \frac{1}{h}  \left( \exp\left(-\int_{t-h}^t L_\sigma^\varepsilon(\Phi) \, \mathrm{d}\sigma \right)   -1 \right) = - \partial_s |_{s=t}  \exp\left(-\int_{s}^t L_\sigma^\varepsilon(\Phi) \, \mathrm{d}\sigma \right) = -L_t^\varepsilon(\Phi).
\end{align*}
This proves \eqref{na-eq-prob3}. Hence by the dominated convergence theorem and lemma~\ref{na-lem-MTtzero}
\begin{align} \label{na-eq--3}
\frac{1}{h}  \int_{\mathcal{MT}_{[0,t-h]}} P_t^\varepsilon(\Phi) - P_{t-h}^\varepsilon(\Phi) \, \mathrm{d}\Phi
& = \frac{1}{h}  \int_{\mathcal{MT}_{[0,t-h]}} \exp\left(-\int_{t-h}^t L_\sigma^\varepsilon(\Phi) \, \mathrm{d}\sigma \right) P_{t-h}^\varepsilon(\Phi) - P_{t-h}^\varepsilon(\Phi) \, \mathrm{d}\Phi \nonumber \\
& =   \int_{\mathcal{MT}_{[0,t)}} \frac{1}{h} P_{t-h}^\varepsilon(\Phi) \left( \exp\left(-\int_{t-h}^t L_\sigma^\varepsilon(\Phi) \, \mathrm{d}\sigma \right)   -1 \right) \, \mathrm{d}\Phi \nonumber \\
& \xrightarrow{h \downarrow 0} \int_{\mathcal{MT}_{[0,t)}}  -L_t^\varepsilon(\Phi)P _t^\varepsilon(\Phi) \, \mathrm{d}\Phi \nonumber \\
& = \int_{\mathcal{MT}_{[0,t]}}  -L_t^\varepsilon(\Phi)P _t^\varepsilon(\Phi) \, \mathrm{d}\Phi = \int_{\mathcal{MT}}  -L_t^\varepsilon(\Phi)P _t^\varepsilon(\Phi) \, \mathrm{d}\Phi.
\end{align}
Combining \eqref{na-eq--1}, \eqref{na-eq--2} and \eqref{na-eq--3} proves the lemma.
\end{proof}
The following lemma is used to prove \eqref{na-eq-ptepuni}.
\begin{lem} \label{na-lem-ptptw}
For almost all $\Phi \in \mathcal{MT}$, uniformly for $t \in [0,T]$,
\[ \lim_{\varepsilon \to 0} \left| P_t^0(\Phi) - P_t^\varepsilon(\Phi) \right| =0. \]
\end{lem}
\begin{proof}
Let $\varepsilon$ be sufficiently small so that lemma~\ref{na-lem-LLepcomp} holds. We prove by induction on $n$, the number of collisions in $\Phi$. Suppose $n =0$. Then by definition~\ref{na-deff-Ptphi}, \eqref{na-eq-expab} and lemma~\ref{na-lem-LLepcomp},
\begin{align*}
 \left| P_t^0(\Phi) - P_t^\varepsilon(\Phi) \right| &  = \left| \exp\left( -\int_0^t L_s^0(\Phi) \, \mathrm{d}s \right) - \exp\left( -\int_0^t L_s^\varepsilon(\Phi) \, \mathrm{d}s \right) \right| f_0(x_0,v_0) \\
 & \leq  \int_0^t | L_s^0(\Phi)- L_s^\varepsilon(\Phi) | \, \mathrm{d}s  f_0(x_0,v_0)  \leq  2CT(1+|v(\tau)|)\varepsilon^{\alpha/6} f_0(x_0,v_0),
\end{align*}
as required. Now suppose the result holds true for almost all $\Phi \in \mathcal{MT}$ with $n = j$ for some $j \geq 0$ and let $\Psi \in \mathcal{MT}$ with $n=j+1$ be such that the result holds for $\bar{\Psi}$ and,
\begin{align*}
\bar{g}(v')(1+|v'|)  & \leq M_\infty \textrm{ and } \\
| g_\tau(x(\tau),v')   - g_\tau(x(\tau)+\varepsilon\nu,v') | & \leq M\varepsilon^{\alpha}.
\end{align*}
Indeed by \eqref{na-eq-glinf} and \eqref{na-eq-ghldassmp} this only excludes a set  of zero measure. Let $\delta >0$. Then using \eqref{na-eq-ghldassmp} take $\varepsilon $ sufficiently small so that,
\begin{equation} \label{na-eq-ptepunihelp1}
|g_\tau (x(\tau),v') - g_\tau (x(\tau)+\varepsilon\nu,v')| \leq M\varepsilon ^{\alpha} < \frac{\delta}{3 (1+P_\tau^0(\bar{\Psi})) (1+ |v(\tau^-)| + |v'|)}.
\end{equation}
And using the inductive assumption take $\varepsilon$ sufficiently small so that,
\begin{align}\label{na-eq-ptepunihelp2}
| P_\tau^\varepsilon(\bar{\Psi}) - P_\tau^\varepsilon(\bar{\Psi}) | <\frac{\delta}{3  M_\infty (1+|v(\tau^-)|) }.
\end{align}
Now by the inductive assumption for $\varepsilon$ sufficiently small,
\[ 0 \leq P_\tau^\varepsilon(\bar{\Psi}) \leq  | P_\tau^0(\bar{\Psi}) - P_\tau^\varepsilon(\bar{\Psi}) | + P_\tau^0(\bar{\Psi}) \leq 1+ P_\tau^0(\bar{\Psi}).  \]
So, as in the base case, take $\varepsilon$ sufficiently small so that
\begin{align}\label{na-eq-ptepunihelp3}
\left|\exp\left( -\int_\tau^t L_s^0(\Psi) \, \mathrm{d}s \right) - \exp\left( -\int_\tau^t L_s^\varepsilon(\Psi) \, \mathrm{d}s \right) \right| & \leq  \int_\tau^t | L_s^0(\Psi)- L_s^\varepsilon(\Psi) | \, \mathrm{d}s  \nonumber \\
& \leq \frac{\delta}{3 (1+P_\tau^\varepsilon(\bar{\Psi})) M_\infty (1+|v(\tau^-)|) } .
\end{align}
Hence by \eqref{na-eq-ptepunihelp1}, \eqref{na-eq-ptepunihelp2} and \eqref{na-eq-ptepunihelp3} and bounding the exponential term by 1, for $\varepsilon$ sufficiently small,
\begin{align*}
|P_t^0(\Psi ) - P_t^\varepsilon(\Psi)| & =  \left| \exp\left( -\int_\tau^t L_s^0(\Phi) \, \mathrm{d}s \right)P_\tau^0(\Psi)  - \exp\left( -\int_\tau^t L_s^\varepsilon(\Phi) \, \mathrm{d}s \right)P_\tau^\varepsilon(\Psi) \right| \\
& = \bigg| \exp\left( -\int_\tau^t L_s^0(\Phi) \, \mathrm{d}s \right)P_\tau^0(\bar{\Psi})g_\tau(x(\tau),v')[(v(\tau^-)-v')\cdot \nu]_+   \\
& \qquad - \exp\left( -\int_\tau^t L_s^\varepsilon(\Phi) \, \mathrm{d}s \right)P_\tau^\varepsilon(\bar{\Psi})g_\tau(x(\tau)+\varepsilon\nu,v')[(v(\tau^-)-v')\cdot \nu]_+ \bigg| \\
& \leq  P_\tau^0(\bar{\Psi})[(v(\tau^-)-v')\cdot \nu]_+ | g_\tau(x(\tau),v')   - g_\tau(x(\tau)+\varepsilon\nu,v') | \\
& \qquad  + g_\tau(x(\tau)+\varepsilon\nu,v')[(v(\tau^-)-v')\cdot \nu]_+ \left|  P_\tau^0(\bar{\Psi}) -  P_\tau^\varepsilon(\bar{\Psi}) \right| \\
& \qquad + P_\tau^\varepsilon(\bar{\Psi})  g_\tau(x(\tau)+\varepsilon\nu,v')[(v(\tau^-)-v')\cdot \nu]_+ \\
&  \qquad \qquad \times \left| \exp\left( -\int_\tau^t L_s^0(\Phi) \, \mathrm{d}s \right) - \exp\left( -\int_\tau^t L_s^\varepsilon(\Phi) \, \mathrm{d}s \right) \right| \\
 & < \delta.
\end{align*}
This completes the inductive step and so proves the result.

\end{proof}
We can now prove the remainder of theorem~\ref{na-thm-id}.

\begin{proof}[Proof of Theorem~\ref{na-thm-id}]
Let $\varepsilon\geq 0$. By proposition~\ref{na-prop-idpart1} it remains only to prove that $P_t^\varepsilon$ is a probability measure and that \eqref{na-eq-ptepuni} holds. Positivity follows by the definition of $P_t^\varepsilon$ in definition~\ref{na-deff-Ptphi}. By \eqref{na-eq-ftptcon},
\[ \int_{\mathcal{MT}} P_0^\varepsilon(\Phi) \, \mathrm{d}\Phi  = \int_{U \times \mathbb{R}^3} f_0(x,v) \, \mathrm{d}x \, \mathrm{d}v = 1.   \]
Now let $t>0$. By lemmas~\ref{na-lem-prob+} and \ref{na-lem-prob-},
$\partial_t \int_{\mathcal{MT}} P_t^\varepsilon(\Phi) \, \mathrm{d}\Phi $ exists and is equal to zero. Hence,
\[ \int_\mathcal{MT} P_t^\varepsilon(\Phi) \, \mathrm{d}\Phi =1. \]
It remains to prove \eqref{na-eq-ptepuni}. Since $P_t^\varepsilon$ and $P_t^0$ are probability measures on $\mathcal{MT}$ and we have proven pointwise convergence in lemma~\ref{na-lem-ptptw} we apply Scheff\'{e}'s theorem (see \cite[Theorem 16.12]{bill12}) which immediately gives the result.
  \end{proof}
We finish this section by proving that $f_t^\varepsilon$, the evolution system solution to the $\varepsilon$ dependent linear Boltzmann equation, is a probability measure and that it converges in $L^1$ to $f_t^0$ as $\varepsilon$ tends to zero.

\begin{prop} \label{na-prop-ftprob}
For any $t \in [0,T]$ and $\varepsilon \geq 0$, $f_t^\varepsilon$ is a probability measure on $U \times \mathbb{R}^3$ and the trajectory $V^\varepsilon(t,0)f_0$ is honest (see \cite[Remark 4.20]{arlotti14}). Moreover $f_t^\varepsilon$ converges to $f_t^0$ in $L^1$ as $\varepsilon$ tends to zero uniformly for $t \in [0,T]$.
\end{prop}
\begin{proof}
Let $t\in[0,T], \varepsilon\geq 0$. Since $f_0 \in L_1^+(U\times \mathbb{R}^3) $ we have by \cite[proposition 2.2]{arlotti14} for any $j \geq 0$, $V_j^\varepsilon(t,0)f_0 \in L_1^+(U\times \mathbb{R}^3)$, where $V_j^\varepsilon$ are as in the proof of proposition~\ref{na-prop-ptjsum}. Since $V^\varepsilon = \sum_{j=0}^\infty V_j^\varepsilon$ it follows that $V^\varepsilon(t,0)f_0 \in L_1^+(U\times \mathbb{R}^3) $. Now by theorem~\ref{na-thm-id} and \eqref{na-eq-ptepftepint},
\begin{align*}
\int_{U \times \mathbb{R}^3} f_t^\varepsilon(x,v) \, \mathrm{d}x \, \mathrm{d}v = \int_{\mathcal{MT}} P_t^\varepsilon(\Phi) \, \mathrm{d}\Phi  =1,
\end{align*}
so $f_t^\varepsilon$ is a probability measure. Further this implies,
\begin{align*}
\int_{U \times \mathbb{R}^3} V^\varepsilon(t,0)f_0(x,v) \, \mathrm{d}x \, \mathrm{d}v   =\int_{U \times \mathbb{R}^3} f_0(x,v) \, \mathrm{d}x \, \mathrm{d}v.
\end{align*}
Honesty of the trajectory of $V^\varepsilon(t,0)f_0$ follows from \cite[section 4.3]{arlotti14}. To prove convergence in $L^1$, it is enough to let $t \in [0,T]$ and $\Omega \subset U \times \mathbb{R}^3$ measurable. Then by theorem~\ref{na-thm-id},
\begin{align*}
\left| \int_{\Omega} f_t^0(x,v) - f_t^\varepsilon(x,v) \, \mathrm{d}x \, \mathrm{d}v  \right| & = \left| \int_{S_t(\Omega)} P_t^0(\Phi) - P_t^\varepsilon(\Phi) \, \mathrm{d}\Phi  \right|  \leq  \int_{\mathcal{MT}} \left| P_t^0(\Phi) - P_t^\varepsilon(\Phi) \right| \, \mathrm{d}\Phi   \to 0,
\end{align*}
as required.
\end{proof}

\section{The Empirical Distribution} \label{na-sec-emp}

We now describe the empirical distribution $\hat{P}_t^\varepsilon$. The main result of this section is theorem~\ref{na-thm-emp}, where we show that $\hat{P}_t^\varepsilon$ solves the empirical equation - at least for well controlled trees. The similarity of the empirical and idealised equations is then used in section~\ref{na-sec-conv} to prove the convergence between $P_t^\varepsilon$ and $\hat{P}_t^\varepsilon$, which is used to prove the required convergence of theorem~\ref{na-thm-main}.

\begin{deff}
For $t \in [0,T]$ and $\varepsilon >0$ let $\hat{P}_t^\varepsilon$ be the probability measure on $\mathcal{MT}$ obtained by observing the particle dynamics as described in section \ref{na-sec-model}. Notice that for any $\Omega \subset U \times \mathbb{R}^3$, $\varepsilon >0$ and $t \in [0,T]$
\[ \int_{\Omega} \hat{f}_t^N(x,v) \, \mathrm{d}x \, \mathrm{d}v =  \hat{P}_t^\varepsilon(S_t(\Omega)). \]
\end{deff}
\begin{lem}\label{lem-abscts}
Given the assumptions of Theorem \ref{na-thm-main},  the empirical distribution of the tagged particle $\hat{f}_t^N$ is absolutely continuous with respect to the Lebesgue measure on $U\times \mathbb{R}^3$.
\end{lem}
\begin{proof} The empirical distribution of the tagged particle $\hat{f}_t^N$ can be equivalently obtained by integrating over the background particles of the $N+1$  particle distribution. The initial distribution of the $N+1$ particles is given by
\[\mbox{Prob}_N(((x_0,v_0),(x_1,v_1),\ldots,(x_N,v_N))=((u_0,w_0),(u_1,w_1),\ldots,(u_N,w_N)))=f_0(u_0,w_0) \prod_{i=1}^Ng_0(u_i,w_i), \]
which, under our assumptions, is absolutely continuous with respect to Lebesgue measure on $(U \times \mathbb{R}^3)^{N+1}$. As the $N+1$ particle flow preserves Lebesgue measure, this implies that
the empiric $N+1$ particle distribution is absolutely  continuous with respect to Lebesgue measure on $(U \times \mathbb{R}^3)^{N+1}$. Hence its marginal $\hat{f}_t^N$ is
  absolutely continuous with respect to the Lebesgue measure on $U\times \mathbb{R}^3$.
\end{proof}
We now describe the set of `good' trees that we will work with.
\begin{deff} \label{na-deff-V}
For a tree $\Phi \in \mathcal{MT}$ and time $t \in [0,T]$ recall that we denote the position and velocity of the root by $(x(t),v(t))$ and for $j=1,\dots,n$ the position and velocity of the background particle corresponding to the $j$-th collisions by $(x_j(t),v_j(t))$. Define $\mathcal{V}(\Phi) \in [0,\infty)$ to be the maximum velocity involved in the tree,
\[ \mathcal{V}(\Phi):= \max_{t \in [0,T]} \left\{  |v(t)| ,  \max_{j=1,\dots,n(\Phi)} |v_j(t)|\right\}. \]
\end{deff}
\begin{deff} \label{na-deff-recollfree}
A tree $\Phi$ is called re-collision free at diameter $\varepsilon$ if for all $j=1,\dots,n$ and for all $t \in [0,T]\setminus\{t_j\}$ - where $t_j$ denotes the time of collision between the root and background particle $j$,
\[ |x(t)-x_j(t)|>\varepsilon. \]
That is if the root collides with a background particle at time $t_j$, it has not collided with that background particle before in the tree and up to time $T$ it does not come into contact with that particle again. Define
\[ R(\varepsilon):= \{ \Phi \in \mathcal{MT} : \Phi \textrm{ is re-collision free at diameter } \varepsilon \}.  \]
\end{deff}
\begin{deff} \label{na-deff-nongraz}
A tree $\Phi \in \mathcal{MT}$ is called non-grazing if all collisions in $\Phi$ are non-grazing, that is
\[ \min_{j=1,\dots,n(\Phi)} \nu_j \cdot (v(t_j^-) - v_j(t^-) ) > 0. \]
\end{deff}
\begin{deff} \label{na-deff-intialover}
A tree $\Phi \in \mathcal{MT}$ is called free from initial overlap at diameter $\varepsilon >0$ if initially the root is at least $\varepsilon$ away from all the background particles. That is if for $j=1,\cdots, N$
\[ |x_0 - x_j| > \varepsilon. \]
we define
\[ S(\varepsilon) := \{ \Phi \in \mathcal{MT} : \Phi \textrm{ is free from initial overlap at diameter } \varepsilon \}. \]
\end{deff}
\begin{deff} \label{na-deff-goodtrees}
For any pair of decreasing functions $V,M:(0,\infty)\to [0,\infty)$ such that $\lim_{\varepsilon \to 0} V(\varepsilon) = \infty = \lim_{\varepsilon \to 0} M(\varepsilon) $ the set of good tress of diameter $\varepsilon$ is defined by,
\begin{align*}
\mathcal{G}(\varepsilon) &: = \Big\{   \Phi \in \mathcal{MT} : n(\Phi) \leq M(\varepsilon), \, \mathcal{V}(\Phi) \leq V(\varepsilon),  \Phi \in  R(\varepsilon) \cap S(\varepsilon) \, \textrm{ and } \Phi \textrm{ is non-grazing}   \Big\}.
\end{align*}
\end{deff}
\begin{lem}
As $\varepsilon$ decreases $\mathcal{G}(\varepsilon)$ increases.
\end{lem}
\begin{proof}
The only non-trivial conditions are checking that $S(\varepsilon)$ and $R(\varepsilon)$ are increasing. To this aim suppose that $\varepsilon ' < \varepsilon$ and $\Phi \in S(\varepsilon)$. If $n=0$ then it follows from the definition that $\Phi \in S(\varepsilon')$. Else $n \geq 1$. For the background particles not involved in the tree it is clear that reducing $\varepsilon$ to $\varepsilon'$ will not cause initial overlap. For $1\leq j \leq n$ the initial position of the background particle corresponding to collision $j$ is $x(t_j)  - t_j v_j+\varepsilon\nu_j$. Since $\Phi \in S(\varepsilon)$,
\[ |x_0 - (x(t_j)  - t_j v_j+\varepsilon\nu_j) | > \varepsilon, \]
that is $x_0 - (x(t_j)  - t_j v_j) \notin B_{\varepsilon}(-\varepsilon\nu_j)$. Hence $x_0 - (x(t_j)  - t_j v_j) \notin B_{\varepsilon'}(-\varepsilon'\nu_j)$ and so $\Phi \in S(\varepsilon')$.

Now suppose that $\varepsilon ' < \varepsilon$ and $\Phi \notin R(\varepsilon')$. Then in particular $n \geq 1$ and there exists a $1 \leq j \leq n$ and $t >t_j$ such that, if we denote the velocity of the background particle $j$ after its collision at time $t_j$ by $\bar{v}$,
\[ x(t) - (x(t_j) + \varepsilon' \nu_j + (t-t_j)\bar{v}) \in \varepsilon' \mathbb{S}^2, \]
that is $x(t) - (x(t_j)+ (t-t_j)\bar{v} )\in -  \varepsilon' \nu_j + \varepsilon' \mathbb{S}^2$. Hence since the left in side is continuous with respect to $t$ it must be that there exists a $t'$ such that, $x(t') - (x(t_j)+ (t'-t_j)\bar{v}) \in -  \varepsilon \nu_j + \varepsilon \mathbb{S}^2$, i.e. $\Phi \notin R(\varepsilon)$. Hence $R(\varepsilon) \subset R(\varepsilon')$ and so $\mathcal{G}(\varepsilon) \subset \mathcal{G}(\varepsilon')$.
\end{proof}

The last lemma allows a simplified definition of $\mathcal{G}(\varepsilon)$ compared to \cite{matt16}, where the monotonicity was enforced by taking unions. We will later give restrictions on $V$ and $M$ in order to control bounds in order to prove required results.

\begin{lem} \label{na-lem-phatabscts}
Let $\varepsilon>0$ and $\Phi \in \mathcal{G}(\varepsilon)$ then $\hat{P}_t^\varepsilon$ is absolutely continuous with respect to the Lebesgue measure $\lambda$ on a neighbourhood of $\Phi$.
\end{lem}
\begin{proof}
The only difference to the proof of \cite[Lemma 4.8]{matt16} is that instead of $\int_{C_{h,j}}g_0(v)\,\mathrm{d}x\,\mathrm{d}v$, we have $\int_{C_{h,j}}g_0(x,v)\,\mathrm{d}x\,\mathrm{d}v$
due to $g_0$ depending on $x$. As $g_0\in L^1(U\times \mathbb{R}^3)$ by assumption, the argument can be concluded in the same way.

\end{proof}

From now on we let $\hat{P}_t^\varepsilon$ refer to the density of the probability measure on $\mathcal{MT}$. We now define the empirical equation, which we show $\hat{P}_t^\varepsilon$ solves. First we define the operator $\hat{\mathcal{Q}}_t^\varepsilon$, which is similar to the operator $\mathcal{Q}_t^\varepsilon$ in the idealised case, but includes the complexities of the particle evolution.

For a given tree $\Phi$, a time $t \in [0,T]$ and $\varepsilon>0$, define the function $\mathbbm{1}_t^\varepsilon[\Phi]:  U \times \mathbb{R}^3 \rightarrow \{0,1\}$ by
\begin{equation} \label{na-eq-deff1phi}
 \mathbbm{1}_t^\varepsilon[\Phi] (\bar{x},\bar{v}):= \begin{cases}
1 \textrm{ if for all } s \in (0,t), \, |x(s)-(\bar{x}+s\bar{v})| > \varepsilon, \\ 0 \textrm{ else}.
\end{cases}
\end{equation}

That is $\mathbbm{1}_t^\varepsilon[\Phi] (\bar{x},\bar{v})$ is $1$ if a background particle starting at the position $(\bar{x},\bar{v})$ avoids colliding with the root particle of the tree $\Phi$ up to the time $t$ and zero otherwise.

For $\Phi \in \mathcal{MT}, t \geq 0$ and $\varepsilon >0$, define the gain operator,
\begin{equation*}
\hat{\mathcal{Q}}_t^{\varepsilon,+}[\hat{P}_t](\Phi) :=
\begin{cases}
\delta(t-\tau)P_t(\bar{\Phi})\dfrac{g_\tau(x(\tau)+\varepsilon\nu,v')[(v(\tau^-)-v')\cdot\nu]_+}{\int_{ U \times \mathbb{R}^3}  g_0(\bar{x},\bar{v}) \mathbbm{1}_\tau^\varepsilon[\Phi] (\bar{x},\bar{v}) \, \mathrm{d}\bar{x} \, \mathrm{d}\bar{v}} & \textrm{ if } n\geq 1, \\
0 & \textrm{ if } n=0,
\end{cases}
\end{equation*}
and define the loss operator,
\[ \hat{\mathcal{Q}}_t^{\varepsilon,-} [\hat{P}_t](\Phi): =  \hat{P}_t(\Phi)\frac{\int_{\mathbb{S}^{2}}  \int_{\mathbb{R}^3}   g_t(x(t)+\varepsilon\nu,\bar{v})[(v(\tau)-\bar{v})\cdot \nu]_+ \, \mathrm{d}\bar{v} \, \mathrm{d}\nu - \hat{C}(\varepsilon)}{\int_{ U \times \mathbb{R}^3}  g_0(\bar{x},\bar{v}) \mathbbm{1}_t^\varepsilon[\Phi] (\bar{x},\bar{v}) \, \mathrm{d}\bar{x} \, \mathrm{d}\bar{v}  }. \]
For some $\hat{C}(\varepsilon)>0$ depending on $t$ and $\Phi$ of $o(1)$ as $\varepsilon$ tends to zero detailed later. Finally define the operator $\hat{\mathcal{Q}}_t^\varepsilon$ as follows,
\[ \hat{\mathcal{Q}}_t^\varepsilon = \hat{\mathcal{Q}}_t^{\varepsilon,+} - \hat{\mathcal{Q}}_t^{\varepsilon,-}  . \]

\begin{thm} \label{na-thm-emp}
For $\varepsilon$ sufficiently small and for all $\Phi \in \mathcal{G}(\varepsilon)$, $\hat{P}_t^\varepsilon$ solves the following
\begin{equation} \label{na-eq-emp}
\begin{cases}
\partial_t \hat{P}_t^\varepsilon(\Phi) & = (1-\gamma^\varepsilon(t)) \hat{\mathcal{Q}}_t^\varepsilon [\hat{P}_t^\varepsilon](\Phi) \\
\hat{P}_0^\varepsilon (\Phi)& = \zeta^\varepsilon(\Phi)  f_0 (x_0,v_0) \mathbbm{1}_{n(\Phi) = 0}.
\end{cases}
\end{equation}
The functions $\gamma^\varepsilon$ and $\zeta^\varepsilon$ are given by
\begin{equation} \label{na-eq-zetadeff}
\zeta^\varepsilon(\Phi) := \left( 1- \int_{B_\varepsilon(x_0)} \int_{\mathbb{R}^3} g_0(\bar{x},\bar{v}) \, \mathrm{d}\bar{v} \,\mathrm{d}\bar{x}\right)^N,
\end{equation}
and,
\begin{equation*}
\gamma^\varepsilon(t):=
\begin{cases}
1  & \textrm{ if } t < \tau  \\
n(\bar{\Phi})\varepsilon^2 &  \textrm{ if } t=\tau \\
n(\Phi)\varepsilon^2 & \textrm{ if } t > \tau.
\end{cases}
\end{equation*}
\end{thm}
We prove this theorem by breaking it into several lemmas proving the initial data, gain and loss term separately using the definition of $\hat{P}_t$. Firstly, the initial condition requirement for $\hat{P}_t^\varepsilon$.
\begin{deff}
Let $\omega_0  \in U \times \mathbb{R}^3$ be the random initial position and velocity of the tagged particle. For $1 \leq j \leq N$ let $\omega _ j$ be the random initial position and velocity of the $j$th background particle. By our assumptions $\omega_0$ has distribution $f_0$ and each $\omega _j$ has distribution $g_0$. Finally let $\omega = (\omega_1,\dots, \omega_N)$.
\end{deff}
\begin{lem} \label{na-lem-phatinitial}
Under the assumptions and set up of theorem~\ref{na-thm-emp} we have
\[ \hat{P}_0^\varepsilon (\Phi) = \zeta^\varepsilon(\Phi)  f_0 (x_0,v_0) \mathbbm{1}_{n(\Phi) = 0}. \]
\end{lem}
\begin{proof}
If $n(\Phi) > 0$, $\hat{P}_0(\Phi)=0$, because the tree involves collisions happening at some positive time and as such cannot have occurred at time 0.

Else $n(\Phi)=0$, so $\Phi$ contains only the root particle. The probability of finding the root at the given initial data $(x_0,v_0)$ is $f_0(x_0,v_0)$. But this must be multiplied by a factor less than one because we rule out situations that give initial overlap of the root particle with a background particle. Firstly we calculate,
 \begin{align*}
\mathbb{P}(|x_0 - x_1| > \varepsilon ) & = 1 - \mathbb{P}(|x_0 - x_1| < \varepsilon )
 = 1 - \int_{\mathbb{R}^3} \int_{|x_0 - x_1| <\varepsilon}  g_0(x_1,\bar{v}) \, \mathrm{d}{x_1} \, \mathrm{d}\bar{v}\\ & = 1- \int_{B_\varepsilon(x_0)} \int_{\mathbb{R}^3} g_0(\bar{x},\bar{v}) \, \mathrm{d}\bar{v} \,\mathrm{d}\bar{x}.
\end{align*}
Hence,
\begin{align*}
\mathbb{P}(|x_0 - x_j| > \varepsilon, \forall j = 1 ,\dots ,N ) & = \mathbb{P}(|x_0 - x_1| > \varepsilon ) ^N  =\left( 1 - \int_{B_\varepsilon(x_0)} \int_{\mathbb{R}^3} g_0(\bar{x},\bar{v}) \, \mathrm{d}\bar{v} \,\mathrm{d}\bar{x}\right) ^N = \zeta^\varepsilon(\Phi).
\end{align*}
\end{proof}

\begin{lem} \label{na-lem-empgain}
Under the set up of Theorem~\ref{na-thm-emp} for $n \geq 1$
\begin{equation*}
\hat{P}_\tau^\varepsilon(\Phi) = (1-\gamma^\varepsilon(\tau))P_\tau^\varepsilon(\bar{\Phi})\dfrac{g_\tau(x(\tau)+\varepsilon\nu,v')[(v(\tau^-)-v')\cdot\nu]_+}{\int_{ U \times \mathbb{R}^3}  g_0(\bar{x},\bar{v}) \mathbbm{1}_\tau^\varepsilon[\Phi] (\bar{x},\bar{v}) \, \mathrm{d}\bar{x} \, \mathrm{d}\bar{v}}.
\end{equation*}
\end{lem}
\begin{proof}
The proof is unchanged from the proof of \cite[lemma 4.16]{matt16}.
\end{proof}

From now on we make the following assumptions on the functions $V,M$ in the definition~\ref{na-deff-goodtrees}. Assume that for any $0<\varepsilon<1$ we have
\begin{equation} \label{na-eq-VMass}
\varepsilon V(\varepsilon)^3 \leq \frac{1}{8} \textrm{ and } M(\varepsilon) \leq \frac{1}{\sqrt[]{\varepsilon}}.
\end{equation}

Before we can prove the loss term we require a number of lemmas that are used to justify that $\hat{P}_t^\varepsilon$ is differentiable for $t>\tau$ and has the required derivative.

\begin{deff}
Let $\varepsilon>0$ and $\Phi \in \mathcal{G}(\varepsilon)$. For $h>0$ define
\begin{align*}
W_{t,h}^\varepsilon(\Phi) := &  \bigg\{ (\bar{x},\bar{v}) \in U \times \mathbb{R}^3 : \exists (t',\nu') \in (t,t+h) \times \mathbb{S}^2 \\&\textrm{ with } x(t')+\varepsilon \nu' = \bar{x} + t'\bar{v} \textrm{ and } (v(t') - \bar{v})\cdot \nu' > 0  \bigg\}.
\end{align*}
That is $W_{t,h}^\varepsilon(\Phi)$ contains all possible initial points for a background particle to start such that, if it travels with constant velocity, it will collide the tagged particle at some time in $(t,t+h)$. Further define,
\[ I_{t,h}^\varepsilon(\Phi) : = \frac{ \int_{U \times \mathbb{R}^3} g_0(\bar{x},\bar{v}) \mathbbm{1}_{W_{t,h}^\varepsilon(\Phi)}(\bar{x},\bar{v}) \mathbbm{1}_{t}^\varepsilon(\Phi)(\bar{x},\bar{v}) \, \mathrm{d}\bar{x} \, \mathrm{d}\bar{v} }{ \int_{U \times \mathbb{R}^3} g_0(\bar{x},\bar{v})  \mathbbm{1}_{t}^\varepsilon(\Phi)(\bar{x},\bar{v}) \, \mathrm{d}\bar{x} \, \mathrm{d}\bar{v}} .\]
\end{deff}

\begin{lem} \label{na-lem-2coll}
For $\varepsilon > 0$ sufficiently small, $\Phi \in \mathcal{G}(\varepsilon)$ and $t > \tau$
\begin{align} \label{na-eq-2cpos}
\lim_{h \downarrow 0}& \frac{1}{h}\hat{P}_t^\varepsilon( \# (\omega \cap W_{t,h}^\varepsilon(\Phi) ) \geq 2 \, | \, \Phi )  = 0.
\\ \label{na-eq-2cneg}
\lim_{h \downarrow 0}& \frac{1}{h}\hat{P}_{t-h}^\varepsilon( \#( \omega \cap W_{t-h,h}^\varepsilon(\Phi) )\geq 2 \, |\, \Phi )  = 0.
\end{align}
\end{lem}
\begin{proof}
We first prove \eqref{na-eq-2cpos}. Since $\Phi$ is a good tree, so in particular is re-collision free, and we are conditioning on $\Phi$ occurring at time $t$, we know that for $1\leq j \leq n$, $\omega _j \notin W_{t,h}^\varepsilon(\Phi)$ (since if this was not the case there would be a re-collision). Hence by the inclusion exclusion principle and the independence of the initial distribution of the background particles,
\begin{align} \label{na-eq-pthpos}
\hat{P}_t^\varepsilon( \# (\omega \cap W_{t,h}^\varepsilon(\Phi) ) \geq 2 \, | \, \Phi ) &  \leq \sum_{n+1 \leq i < j \leq N }  \hat{P}_t^\varepsilon( \omega_i \in W_{t,h}^\varepsilon(\Phi) \textrm{ and } \omega_j \in W_{t,h}^\varepsilon(\Phi)  \, | \, \Phi ) \nonumber \\
& \leq N^2 \hat{P}_t^\varepsilon( \omega_N \in W_{t,h}^\varepsilon(\Phi) \, |\, \Phi )^2.
\end{align}
We note that,
\[ \hat{P}_t^\varepsilon( \omega_N \in W_{t,h}^\varepsilon(\Phi) \, |\, \Phi ) = I_{t,h}^\varepsilon(\Phi). \]
We now bound the numerator and denominator of $I_{t,h}^\varepsilon(\Phi)$. Firstly, for a fixed $\bar{v}$, the set of  points $\bar{x}$ such that $(\bar{x},\bar{v}) \in W_{h,t}^\varepsilon(\Phi)$ is a cylinder of radius $\varepsilon$ and length $\int_t^{t+h} |v(s) - \bar{v}|\, \mathrm{d}s$. Hence, since $\Phi \in \mathcal{G}(\varepsilon)$,
\begin{align} \label{na-eq-ihtnum}
\int_{U \times \mathbb{R}^3} &  g_0(\bar{x},\bar{v}) \mathbbm{1}_{W_{t,h}^\varepsilon(\Phi)}(\bar{x},\bar{v}) \mathbbm{1}_{t}^\varepsilon(\Phi)(\bar{x},\bar{v}) \, \mathrm{d}\bar{x} \, \mathrm{d}\bar{v} \leq \int_{U \times \mathbb{R}^3}   \bar{g}(\bar{v}) \mathbbm{1}_{W_{t,h}^\varepsilon(\Phi)}(\bar{x},\bar{v})  \, \mathrm{d}\bar{x} \, \mathrm{d}\bar{v} \nonumber  \\
& = \pi \varepsilon^2 \int_{ \mathbb{R}^3}   \bar{g}(\bar{v}) \int_t^{t+h} |v(s) - \bar{v}|\, \mathrm{d}s  \, \mathrm{d}\bar{v} \nonumber  \leq h \pi \varepsilon^2 \int_{ \mathbb{R}^3}   \bar{g}(\bar{v}) (V(\varepsilon) + |\bar{v}|)  \, \mathrm{d}\bar{v} \nonumber \\
& \leq h \pi \varepsilon^2 M_g (V(\varepsilon) + 1) .
\end{align}
Now turning to the denominator, we note first that
\begin{align*}
\int_{U \times \mathbb{R}^3} g_0(\bar{x},\bar{v})  \mathbbm{1}_{t}^\varepsilon(\Phi)(\bar{x},\bar{v}) \, \mathrm{d}\bar{x} \, \mathrm{d}\bar{v} & = \int_{U \times \mathbb{R}^3} g_0(\bar{x},\bar{v}) (1-  \mathbbm{1}_{W_{0,t}^\varepsilon(\Phi)}(\bar{x},\bar{v}) ) \, \mathrm{d}\bar{x} \, \mathrm{d}\bar{v} \\
& = 1- \int_{U \times \mathbb{R}^3} g_0(\bar{x},\bar{v})   \mathbbm{1}_{W_{0,t}^\varepsilon(\Phi)}(\bar{x},\bar{v}) \, \mathrm{d}\bar{x} \, \mathrm{d}\bar{v}.
\end{align*}
By the same estimates as in the numerator, using $t \in [0,T]$ we have
\[ \int_{U \times \mathbb{R}^3} g_0(\bar{x},\bar{v})   \mathbbm{1}_{W_{0,t}^\varepsilon(\Phi)}(\bar{x},\bar{v}) \, \mathrm{d}\bar{x} \, \mathrm{d}\bar{v}  \leq \varepsilon^2 T M_g (V(\varepsilon) + 1).  \]
So by \eqref{na-eq-VMass} we have that for $\varepsilon$ sufficiently small this is less than $1/2$. Hence
\[ \int_{U \times \mathbb{R}^3} g_0(\bar{x},\bar{v})  \mathbbm{1}_{t}^\varepsilon(\Phi)(\bar{x},\bar{v}) \, \mathrm{d}\bar{x} \, \mathrm{d}\bar{v}  \geq 1/2. \]
Combining this and \eqref{na-eq-ihtnum} we have that for $\varepsilon$ sufficiently small,
\begin{equation}\label{na-eq-ihtbound}
 I_{t,h}^\varepsilon(\Phi) \leq 2h \pi \varepsilon^2 M_g (V(\varepsilon) + 1).
\end{equation}
Hence substituting this into \eqref{na-eq-pthpos}, and using that in the Boltzmann-Grad scaling $N\varepsilon^2 = 1$,
\begin{align*}
\hat{P}_t^\varepsilon( \# (\omega \cap W_{t,h}^\varepsilon(\Phi) ) \geq 2 \, | \, \Phi ) &  \leq N^2 \varepsilon^4 h^2 \pi ^2 M_g^2 (V(\varepsilon) + 1)^2  = h^2 \pi ^2 M_g^2 (V(\varepsilon) + 1)^2.
\end{align*}
Diving by $h$ and taking the limit $h\downarrow 0$ gives \eqref{na-eq-2cpos}. For \eqref{na-eq-2cneg} we use the same argument to see that,
\begin{align*}
\hat{P}_{t-h}^\varepsilon( \# (\omega \cap W_{t-h,h}^\varepsilon(\Phi) ) \geq 2 \, | \, \Phi ) & \leq N^2 \hat{P}_{t-h}^\varepsilon( \omega_N \in W_{t-h,h}^\varepsilon(\Phi) \, |\, \Phi )^2 = N^2 I_{t-h,h}^\varepsilon(\Phi)^2.
\end{align*}
We can now employ a similar approach to show that for $\varepsilon$ sufficiently small, after diving by $h$, this converges to zero as $h\downarrow 0$.
\end{proof}

\begin{deff}
For $\Phi \in \mathcal{MT}$, $t>\tau$, $h>0$ and $\varepsilon > 0$ define
\[ B_{t,h}^\varepsilon(\Phi) : = \{ (\bar{x},\bar{v}) \in U \times \mathbb{R}^3 : \mathbbm{1}_t^\varepsilon[\Phi] (\bar{x},\bar{v}) =0 \textrm{ and } \mathbbm{1}_{W_{t,h}^\varepsilon(\Phi)} (\bar{x},\bar{v}) =1 \} \]
That is, $B_{t,h}^a\varepsilon(\Phi)$ is the set of all initial positions such that, if a background particles starts at $(\bar{x},\bar{v})$ and travels with constant velocity  (even if it meets the tagged particle) it collides with the tagged particle once in $(0,t)$ and again in $(t,t+h)$.
\end{deff}

\begin{lem} \label{na-lem-recoll}
For $\varepsilon$ sufficiently small, $\Phi \in \mathcal{G}(\varepsilon)$ and $t>\tau$ there exists a $\hat{C}(\varepsilon) >0$ depending on $t$ and $\Phi$ with $\hat{C}(\varepsilon) =o(1)$ as $\varepsilon$ tends to zero, such that
\begin{align*}
\int_{B_{t,h}^\varepsilon(\Phi)} g_0(\bar{x},\bar{v}) \, \mathrm{d}\bar{x} \, \mathrm{d}\bar{v} = \int_{B_{t-h,h}^\varepsilon(\Phi)} g_0(\bar{x},\bar{v}) \, \mathrm{d}\bar{x} \, \mathrm{d}\bar{v} = h\varepsilon^2 \hat{C}(\varepsilon).
\end{align*}
\end{lem}
\begin{proof}
Using that for any $\Omega \subset U \times \mathbb{R}^3$ measurable
\[ \int_{\Omega} g_0(\bar{x},\bar{v})\, \mathrm{d}\bar{x} \, \mathrm{d}\bar{v} \leq \int_{\Omega} \bar{g}(\bar{v})\, \mathrm{d}\bar{x} \, \mathrm{d}\bar{v}, \]
we can repeat the proof of \cite[lemma 4.13]{matt16}.
\end{proof}
\begin{lem}
For $\varepsilon >0$ sufficiently small, $\Phi \in \mathcal{G}(\varepsilon)$ and $t>\tau$
\begin{align} \label{na-eq-0cpos}
\lim_{h \downarrow 0} \frac{1}{h}\hat{P}_t^\varepsilon  ( \# (\omega \cap W_{t,h}^\varepsilon(\Phi) ) > 0 \, | \, \Phi )
&=  (1-\gamma^\varepsilon(t)) \frac{\int_{\mathbb{S}^{2}}  \int_{\mathbb{R}^3}   g_t(x(t)+\varepsilon\nu,\bar{v})[(v(\tau)-\bar{v})\cdot \nu]_+ \, \mathrm{d}\bar{v} \, \mathrm{d}\nu - \hat{C}(\varepsilon)}{\int_{ U \times \mathbb{R}^3}  g_0(\bar{x},\bar{v}) \mathbbm{1}_t^\varepsilon[\Phi] (\bar{x},\bar{v}) \, \mathrm{d}\bar{x} \, \mathrm{d}\bar{v}  }.
\\ \label{na-eq-0cneg}
\lim_{h \downarrow 0} \frac{1}{h}\hat{P}_{t-h}^\varepsilon  ( \#( \omega \cap W_{t-h,h}^\varepsilon(\Phi) ) >0 \, |\, \Phi )
& = (1-\gamma^\varepsilon(t)) \frac{\int_{\mathbb{S}^{2}}  \int_{\mathbb{R}^3}   g_t(x(t)+\varepsilon\nu,\bar{v})[(v(\tau)-\bar{v})\cdot \nu]_+ \, \mathrm{d}\bar{v} \, \mathrm{d}\nu - \hat{C}(\varepsilon)}{\int_{ U \times \mathbb{R}^3}  g_0(\bar{x},\bar{v}) \mathbbm{1}_t^\varepsilon[\Phi] (\bar{x},\bar{v}) \, \mathrm{d}\bar{x} \, \mathrm{d}\bar{v}  }.
\end{align}
\end{lem}
\begin{proof}
We first show \eqref{na-eq-0cpos}. By \eqref{na-eq-2cpos} we only need to calculate
\[ \lim_{h \downarrow 0} \frac{1}{h}\hat{P}_t^\varepsilon  ( \# (\omega \cap W_{t,h}^\varepsilon(\Phi) ) = 1 \, | \, \Phi ). \]
Now by a similar argument to the proof of lemma~\ref{na-lem-2coll} we have
\begin{align*}
\hat{P}_t^\varepsilon  ( \# & (\omega \cap W_{t,h}^\varepsilon(\Phi) ) = 1 \, | \, \Phi )  \\
& = \sum_{i=n+1}^N \hat{P}_t^\varepsilon  ( \omega_i \in W_{t,h}^\varepsilon(\Phi) \textrm{ and for } n+1\leq j \leq N, j \neq i, \omega_j \notin W_{t,h}^\varepsilon(\Phi)  \, | \, \Phi ) \\
& = (N-n) \hat{P}_t^\varepsilon  ( \omega_N \in W_{t,h}^\varepsilon(\Phi)   \, | \, \Phi )\hat{P}_t^\varepsilon  ( \omega_{N-1} \notin W_{t,h}^\varepsilon(\Phi)   \, | \, \Phi )^{N-n-1} \\
& = (N-n)I_{t,h}^\varepsilon(\Phi) (1-I_{t,h}^\varepsilon(\Phi) )^{N-n-1}
 = (N-n) \sum_{j=0}^{N-n-1} (-1)^j\binom{N-n-1}{j}I_{t,h}^\varepsilon(\Phi)^{j+1}.
\end{align*}
By \eqref{na-eq-ihtbound} it follows that
\begin{align} \label{na-eq-pth1coll}
\lim_{h \downarrow 0} \frac{1}{h}\hat{P}_t^\varepsilon  ( \#  (\omega \cap W_{t,h}^\varepsilon(\Phi) ) = 1 \, | \, \Phi )
& = \lim_{h \downarrow 0} \frac{1}{h} (N-n) \sum_{j=0}^{N-n-1} (-1)^j\binom{N-n-1}{j}I_{t,h}^\varepsilon(\Phi)^{j+1}  \nonumber \\
& = \lim_{h \downarrow 0} \frac{1}{h} (N-n) I_{t,h}^\varepsilon(\Phi).
\end{align}
We compute this limit by noting that
\begin{align*}
I_{t,h}^\varepsilon(\Phi) & =  \frac{ \int_{U \times \mathbb{R}^3} g_0(\bar{x},\bar{v}) \mathbbm{1}_{W_{t,h}^\varepsilon(\Phi)}(\bar{x},\bar{v}) \mathbbm{1}_{t}^\varepsilon(\Phi)(\bar{x},\bar{v}) \, \mathrm{d}\bar{x} \, \mathrm{d}\bar{v} }{ \int_{U \times \mathbb{R}^3} g_0(\bar{x},\bar{v})  \mathbbm{1}_{t}^\varepsilon(\Phi)(\bar{x},\bar{v}) \, \mathrm{d}\bar{x} \, \mathrm{d}\bar{v}} \\
&  =  \frac{ \int_{U \times \mathbb{R}^3} g_0(\bar{x},\bar{v}) \mathbbm{1}_{W_{t,h}^\varepsilon(\Phi)}(\bar{x},\bar{v})  \, \mathrm{d}\bar{x} \, \mathrm{d}\bar{v}}{ \int_{U \times \mathbb{R}^3} g_0(\bar{x},\bar{v})  \mathbbm{1}_{t}^\varepsilon(\Phi)(\bar{x},\bar{v}) \, \mathrm{d}\bar{x} \, \mathrm{d}\bar{v}} - \frac{\int_{B_{t,h}^\varepsilon(\Phi)} g_0(\bar{x},\bar{v})   \, \mathrm{d}\bar{x} \, \mathrm{d}\bar{v}}{ \int_{U \times \mathbb{R}^3} g_0(\bar{x},\bar{v})  \mathbbm{1}_{t}^\varepsilon(\Phi)(\bar{x},\bar{v}) \, \mathrm{d}\bar{x} \, \mathrm{d}\bar{v}} .
\end{align*}
For the first term we note that since $t>\tau$, for any $\bar{v} \in \mathbb{R}^3$
\begin{align*}
\int_{U} \mathbbm{1}_{W_{t,h}^\varepsilon(\Phi)}(\bar{x},\bar{v})   \, \mathrm{d}\bar{x} \,   = \pi \varepsilon^2 \int_t^{t+h}|v(s) - \bar{v}| \, \mathrm{d}s  = \pi \varepsilon^2 h|v(\tau) - \bar{v}|.
\end{align*}
and for $\bar{v} \in \mathbb{R}^3$ and $t>\tau$,
\[ \{ (\bar{x},\bar{v}) \in W_{t,0}^\varepsilon(\Phi) \}= \{ (x(t)+\varepsilon\nu - t \bar{v},\bar{v}) : \nu \in \mathbb{S}^2 \textrm{ and } (v(\tau) - \bar{v})\cdot\nu> 0   \}.  \]
Hence for almost all $\Phi \in \mathcal{G}(\varepsilon)$,
\begin{align}
&\lim_{h \downarrow 0}  \frac{1}{h}(N-n)  \int_{U \times \mathbb{R}^3} g_0(\bar{x},\bar{v}) \mathbbm{1}_{W_{t,h}^\varepsilon(\Phi)}(\bar{x},\bar{v})  \, \mathrm{d}\bar{x} \, \mathrm{d}\bar{v}  = \lim_{h \downarrow 0} \frac{1}{h}(N-n)\int_{W_{t,h}^\varepsilon(\Phi)} g_0(\bar{x},\bar{v}) \, \mathrm{d}\bar{x} \, \mathrm{d}\bar{v}  \nonumber \\
 =& (N-n) \pi \varepsilon^2  \int_{\mathbb{R}^3} \int_{\mathbb{S}^2 } g_0(x(t)+\varepsilon\nu-t\bar{v},\bar{v})|v(\tau) - \bar{v}| \, \mathrm{d}\nu \, \mathrm{d}\bar{v} \nonumber \\
  =& (1-\gamma^\varepsilon(t))\int_{\mathbb{S}^2 } \int_{\mathbb{R}^3}  g_t(x(t)+\varepsilon\nu,\bar{v})[(v(\tau) - \bar{v})\cdot \nu]_+ \,  \mathrm{d}\bar{v} \, \mathrm{d}\nu.\label{na-eq-ihtlim1}
\end{align}
For the second term we have by lemma~\ref{na-lem-recoll},
\begin{align} \label{na-eq-ihtlim2}
\lim_{h \downarrow 0} \frac{1}{h} &  (N-n)   \int_{B_{t,h}^\varepsilon(\Phi)} g_0(\bar{x},\bar{v})   \, \mathrm{d}\bar{x} \, \mathrm{d}\bar{v}  = (N-n)   \varepsilon^2 \hat{C}(\varepsilon)   = (1-\gamma^\varepsilon(t))  \hat{C}(\varepsilon).
\end{align}
Subtracting  \eqref{na-eq-ihtlim2} from \eqref{na-eq-ihtlim1} and then dividing by $\int_{U \times \mathbb{R}^3} g_0(\bar{x},\bar{v})  \mathbbm{1}_{t}^\varepsilon(\Phi)(\bar{x},\bar{v}) \, \mathrm{d}\bar{x} \, \mathrm{d}\bar{v}$ and substituting into \eqref{na-eq-pth1coll} gives,
\begin{align*}
\lim_{h \downarrow 0} \frac{1}{h}\hat{P}_t^\varepsilon  ( \# & (\omega \cap W_{t,h}^\varepsilon(\Phi) )
= (1-\gamma^\varepsilon(t)) \frac{\int_{\mathbb{S}^{2}}  \int_{\mathbb{R}^3}   g_t(x(t)+\varepsilon\nu,\bar{v})[(v(\tau)-\bar{v})\cdot \nu]_+ \, \mathrm{d}\bar{v} \, \mathrm{d}\nu - \hat{C}(\varepsilon)}{\int_{ U \times \mathbb{R}^3}  g_0(\bar{x},\bar{v}) \mathbbm{1}_t^\varepsilon[\Phi] (\bar{x},\bar{v}) \, \mathrm{d}\bar{x} \, \mathrm{d}\bar{v}  },
\end{align*}
as required. For \eqref{na-eq-0cneg} we use \eqref{na-eq-2cneg} and take $h>0$ sufficiently small so that $t-h>\tau$ and repeat the same argument.
\end{proof}
\begin{lem} \label{na-lem-hatpcts}
For $\varepsilon > 0$ sufficiently small and $\Phi \in \mathcal{G}(\varepsilon)$,  $\hat{P}^\varepsilon(\Phi): (\tau,T] \to [0,\infty)$ is continuous with respect to $t$. \end{lem}
\begin{proof}
Let $t \in (\tau,T]$. Then for $h>0$,
\begin{equation} \label{na-eq-phatthdecay}
\hat{P}_{t+h}^\varepsilon(\Phi) = (1-\hat{P}_{t}^\varepsilon( \#( \omega \cap W_{t,h}^\varepsilon(\Phi) ) >0 \, |\, \Phi ) )\hat{P}_t^\varepsilon(\Phi).
\end{equation}
Hence by \eqref{na-eq-0cpos},
\begin{align*}
|\hat{P}_{t+h}^\varepsilon(\Phi) - \hat{P}_t^\varepsilon(\Phi)|  & = \hat{P}_{t}^\varepsilon( \#( \omega \cap W_{t,h}^\varepsilon(\Phi) ) >0 \, |\, \Phi ) )\hat{P}_t^\varepsilon(\Phi) \rightarrow 0 \textrm{ as } h \to 0.
\end{align*}
Now let $h >0$ be sufficiently small so that $t-h>\tau$. Then,
\begin{equation}\label{na-eq-phatt-hd}
\hat{P}_{t}^\varepsilon(\Phi) = (1-\hat{P}_{t-h}^\varepsilon( \#( \omega \cap W_{t-h,h}^\varepsilon(\Phi) ) >0 \, |\, \Phi ) )\hat{P}_{t-h}^\varepsilon(\Phi).
\end{equation}
Further by the properties of the particle dynamics we have, since there is a non-negative probability that the tagged particle experiences a collision in the time $(\tau,t-h]$
\begin{equation*}
\hat{P}_{t-h}^\varepsilon(\Phi) \leq \hat{P}_\tau^\varepsilon(\Phi).
\end{equation*}
Hence by \eqref{na-eq-0cneg}
\begin{align*}
|\hat{P}_{t}^\varepsilon(\Phi) - \hat{P}_{t-h}^\varepsilon(\Phi)|  & = \hat{P}_{t-h}^\varepsilon( \#( \omega \cap W_{t-h,h}^\varepsilon(\Phi) ) >0 \, |\, \Phi ) )\hat{P}_{t-h}^\varepsilon(\Phi) \\
& \leq \hat{P}_{t-h}^\varepsilon( \#( \omega \cap W_{t-h,h}^\varepsilon(\Phi) ) >0 \, |\, \Phi ) )\hat{P}_{\tau}^\varepsilon(\Phi) \rightarrow 0 \textrm{ as } h \to 0.
\end{align*}
\end{proof}
We can now prove that the loss term of \eqref{na-eq-emp} holds.
\begin{lem} \label{na-lem-emploss}
For $\varepsilon >0 $ sufficiently small and $\Phi \in \mathcal{G}(\varepsilon)$,  $\hat{P}^\varepsilon(\Phi) : (\tau,T] \to [0,\infty)$ is differentiable and
\[ \partial_t \hat{P}_t^\varepsilon(\Phi) = (1-\gamma^\varepsilon(t))\hat{Q}_t^\varepsilon[\hat{P}_t^\varepsilon](\Phi).  \]
\end{lem}
\begin{proof}
By \eqref{na-eq-phatthdecay} and \eqref{na-eq-0cpos}
\begin{align} \label{na-eq-hatpdiff1}
\lim_{h \downarrow 0 } & \frac{1}{h} \bigg( \hat{P}_{t+h}^\varepsilon(\Phi) - \hat{P}_t^\varepsilon(\Phi) \bigg) = \lim_{h \downarrow 0 }  \frac{1}{h}\bigg( \hat{P}_{t}^\varepsilon( \#( \omega \cap W_{t,h}^\varepsilon(\Phi) ) >0 \, |\, \Phi )  )\hat{P}_t^\varepsilon(\Phi)\bigg) \nonumber \\
& = (1-\gamma^\varepsilon(t)) \frac{\int_{\mathbb{S}^{2}}  \int_{\mathbb{R}^3}   g_t(x(t)+\varepsilon\nu,\bar{v})[(v(\tau)-\bar{v})\cdot \nu]_+ \, \mathrm{d}\bar{v} \, \mathrm{d}\nu - \hat{C}(\varepsilon)}{\int_{ U \times \mathbb{R}^3}  g_0(\bar{x},\bar{v}) \mathbbm{1}_t^\varepsilon[\Phi] (\bar{x},\bar{v}) \, \mathrm{d}\bar{x} \, \mathrm{d}\bar{v}  } \hat{P}_t^\varepsilon(\Phi) \nonumber \\
& =  (1-\gamma^\varepsilon(t))\hat{Q}_t^\varepsilon[\hat{P}_t^\varepsilon](\Phi).
\end{align}
Further by \eqref{na-eq-phatt-hd}, \eqref{na-eq-0cneg} and lemma~\ref{na-lem-hatpcts} we have,
\begin{align} \label{na-eq-hatpdiff2}
\lim_{h \downarrow 0 } & \frac{1}{h} \bigg( \hat{P}_{t}^\varepsilon(\Phi) - \hat{P}_{t-h}^\varepsilon(\Phi) \bigg) = \lim_{h \downarrow 0 }  \frac{1}{h} \bigg( \hat{P}_{t-h}^\varepsilon( \#( \omega \cap W_{t-h,h}^\varepsilon(\Phi) ) >0 \, |\, \Phi ) )\hat{P}_{t-h}^\varepsilon(\Phi) \bigg) \nonumber \\
& = (1-\gamma^\varepsilon(t)) \frac{\int_{\mathbb{S}^{2}}  \int_{\mathbb{R}^3}   g_t(x(t)+\varepsilon\nu,\bar{v})[(v(\tau)-\bar{v})\cdot \nu]_+ \, \mathrm{d}\bar{v} \, \mathrm{d}\nu - \hat{C}(\varepsilon)}{\int_{ U \times \mathbb{R}^3}  g_0(\bar{x},\bar{v}) \mathbbm{1}_t^\varepsilon[\Phi] (\bar{x},\bar{v}) \, \mathrm{d}\bar{x} \, \mathrm{d}\bar{v}  } \hat{P}_t^\varepsilon(\Phi) \nonumber \\
& =  (1-\gamma^\varepsilon(t))\hat{Q}_t^\varepsilon[\hat{P}_t^\varepsilon](\Phi).
\end{align}
Combining \eqref{na-eq-hatpdiff1} and \eqref{na-eq-hatpdiff2} proves the result.
\end{proof}

\begin{proof}[Proof of theorem~\ref{na-thm-emp}]
The result now follows by lemmas~\ref{na-lem-phatinitial}, \ref{na-lem-empgain} and \ref{na-lem-emploss}.
\end{proof}

\section{Convergence} \label{na-sec-conv}

We have proved that there exists a solution $P_t^\varepsilon$ to the idealised equation in theorem~\ref{na-thm-id} and we have shown in theorem~\ref{na-thm-emp} that the empirical distribution $\hat{P}_t^\varepsilon$ solves the empirical equation, at least for good trees. We now prove the convergence between $P_t^0$ and $\hat{P}_t^\varepsilon$, which will enable the proof of theorem~\ref{na-thm-main}. This section closely follows \cite[section 5]{matt16} the main difference being that because the initial distribution $g_0$ is now spatially inhomogeneous we take an extra step by comparing $P_t^\varepsilon$ and $\hat{P}_t^\varepsilon$. While there is clear difference in the model with background particles changing velocity when they collide with the tagged particle, this makes only a minor difference to the proof by introducing a further higher order error term.

\subsection{Comparing $P_t^\varepsilon$ and $\hat{P}_t^\varepsilon$}
We now introduce new notation. Recall \eqref{na-eq-deff1phi} and \eqref{na-eq-zetadeff}. For $\varepsilon >0$, $t \in [0,T]$ and $\Phi \in \mathcal{G}(\varepsilon)$,
\begin{align}
\eta_t^\varepsilon(\Phi) & := \int_{U \times \mathbb{R}^3} g_0(\bar{x},\bar{v})(1- \mathbbm{1}_t^\varepsilon[\Phi](\bar{x},\bar{v})) \, \mathrm{d}\bar{x} \, \mathrm{d}\bar{v}, \nonumber \\
 R_t^\varepsilon(\Phi) &: = \zeta^\varepsilon(\Phi)P_t^\varepsilon(\Phi), \nonumber \\
C^\varepsilon(\Phi)& := 2 \sup_{t \in [0,T]} \left\{ \int_{\mathbb{S}^{2}} \int_{\mathbb{R}^3} g_t(x(t)+\varepsilon\nu,\bar{v})[(v(t)-\bar{v})\cdot \nu ]_+ \, \mathrm{d}\bar{v} \, \mathrm{d}\nu \right\}, \label{na-eq-cdeff}  \\
 \rho^{\varepsilon,0}_t (\Phi) &: = \eta_t^\varepsilon(\Phi)C^\varepsilon(\Phi)t . \nonumber
\end{align}
Further for $k \geq 1$ define,
1\begin{equation}\label{na-eq-rhokdef}
\rho^{\varepsilon,k}_t (\Phi) := (1-\varepsilon)\rho^{\varepsilon,k-1}_t (\Phi) + \rho^{\varepsilon,0}_t (\Phi) + \varepsilon.
\end{equation}
Note that this implies that for $k \geq 1$,
\begin{equation} \label{na-eq-rhoalt}
\rho^{\varepsilon,k}_t(\Phi) =(1-\varepsilon)^k \rho^{\varepsilon,0}_t (\Phi) + (\rho^{\varepsilon,0}_t (\Phi) + \varepsilon) \sum_{j=1}^k(1-\varepsilon)^{k-j}.
\end{equation}
Finally define,
\[ \hat{\rho}^\varepsilon_t(\Phi) := \rho^{\varepsilon,n(\Phi)}_t(\Phi). \]

\begin{prop} \label{na-prop-phatrt}
For $\varepsilon>0$ sufficiently small, any  $t\in[0,T]$ and almost all $\Phi \in \mathcal{G}(\varepsilon)$,
\[ \hat{P}_t^\varepsilon(\Phi) - R_t^\varepsilon(\Phi) \geq - \hat{\rho}_t^\varepsilon(\Phi)R_t^\varepsilon(\Phi).  \]
\end{prop}
To prove this proposition we use the following lemmas.

\begin{lem} \label{na-lem-phatrdiff}
Let $L_t^\varepsilon(\Phi)$ be given as in \eqref{na-eq-Ldeff}. For $\Phi \in \mathcal{G}(\varepsilon)$ and $t\geq \tau$,
\begin{align*}
\hat{P}_t^\varepsilon(\Phi) - R_t^\varepsilon(\Phi) & \geq \exp \left(- \int_\tau^t (1+2\eta_s^\varepsilon(\Phi))L_s^\varepsilon(\Phi) \, \mathrm{d} s  \right) (\hat{P}_\tau^\varepsilon(\Phi) - R_\tau^\varepsilon(\Phi)) \\
& \, - 2 \eta_t^\varepsilon(\Phi) R_t^\varepsilon (\Phi) \int_\tau^t \exp \left( -  2\eta_s^\varepsilon(\Phi) \int_s^t L_\sigma^\varepsilon(\Phi) \, \mathrm{d}\sigma \right) L_s^\varepsilon(\Phi) \, \mathrm{d}s.
\end{align*}
\end{lem}

\begin{proof}
For $t=\tau$ it is clear the result holds. Let $t>\tau$. By theorem~\ref{na-thm-id} and theorem~\ref{na-thm-emp} we have that
\begin{equation}\label{na-eq-diffphatp}
\partial_t \Big( \hat{P}_t^\varepsilon(\Phi) - R_t^\varepsilon(\Phi) \Big) = -(1-\gamma^\varepsilon(t)) \hat{L}_t^\varepsilon(\Phi)\hat{P}_t^\varepsilon(\Phi) + L_t^\varepsilon(\Phi)R_t^\varepsilon(\Phi),
\end{equation}
where,
\[ \hat{L}_t^\varepsilon(\Phi): =  \frac{\int_{\mathbb{S}^{2}}  \int_{\mathbb{R}^3}   g_t(x(t)+\varepsilon\nu,\bar{v})[(v(\tau)-\bar{v})\cdot \nu]_+ \, \mathrm{d}\bar{v} \, \mathrm{d}\nu - \hat{C}(\varepsilon)}{\int_{ U \times \mathbb{R}^3}  g_0(\bar{x},\bar{v}) \mathbbm{1}_t^\varepsilon[\Phi] (\bar{x},\bar{v}) \, \mathrm{d}\bar{x} \, \mathrm{d}\bar{v}  } . \]
Further by \eqref{na-eq-gbar}, \eqref{na-eq-glinf} and \eqref{na-eq-deff1phi},
\begin{align} \label{na-eq-intgominus}
\int_{ U \times \mathbb{R}^3} &  g_0(\bar{x},\bar{v})(1- \mathbbm{1}_t^\varepsilon[\Phi] (\bar{x},\bar{v})) \, \mathrm{d}\bar{x} \, \mathrm{d}\bar{v} \leq \int_{ U \times {R}^3} \bar{g}(\bar{v})  (1- \mathbbm{1}_t^\varepsilon[\Phi] (\bar{x},\bar{v})) \, \mathrm{d}\bar{x} \, \mathrm{d}\bar{v} \nonumber \\
& \leq \int_{  \mathbb{R}^3}   \bar{g}(\bar{v}) \int_{U} (1- \mathbbm{1}_t^\varepsilon[\Phi] (\bar{x},\bar{v})) \, \mathrm{d}\bar{x} \, \mathrm{d}\bar{v} \leq \int_{  \mathbb{R}^3}   \bar{g}(\bar{v}) \left( \pi \varepsilon^2 \int_0^t |v(s)-\bar{v}| \, \mathrm{d}s \right) \, \mathrm{d}\bar{v} \nonumber \\
& \leq \pi \varepsilon^2 \int_{  \mathbb{R}^3}   \bar{g}(\bar{v}) \left( \int_0^t V(\varepsilon) + |\bar{v}| \, \mathrm{d}s \right) \, \mathrm{d}\bar{v} \leq  \pi \varepsilon^2 T \int_{  \mathbb{R}^3}   \bar{g}(\bar{v}) \left(   V(\varepsilon) + |\bar{v}| \right) \, \mathrm{d}\bar{v} \nonumber \\
& \leq  \pi \varepsilon^2 T M_g ( V(\varepsilon) + 1).
\end{align}
By \eqref{na-eq-VMass} for $\varepsilon$ sufficiently small we can make this less than $1/2$. Now using the fact that for $0\leq z \leq 1/2$ it follows $1/(1-z) \leq 1+ 2z$, we have for $\varepsilon$ sufficiently small
\begin{align*}
\frac{1}{\int_{ U \times \mathbb{R}^3}  g_0(\bar{x},\bar{v}) \mathbbm{1}_t^\varepsilon[\Phi] (\bar{x},\bar{v}) \, \mathrm{d}\bar{x} \, \mathrm{d}\bar{v}} & = \frac{1}{1- \int_{ U \times \mathbb{R}^3}  g_0(\bar{x},\bar{v}) (1-\mathbbm{1}_t^\varepsilon[\Phi] (\bar{x},\bar{v})) \, \mathrm{d}\bar{x} \, \mathrm{d}\bar{v}} \\
& \leq 1 + 2\left( \int_{ U \times \mathbb{R}^3}  g_0(\bar{x},\bar{v}) (1-\mathbbm{1}_t^\varepsilon[\Phi] (\bar{x},\bar{v})) \, \mathrm{d}\bar{x} \, \mathrm{d}\bar{v} \right)  = 1+2\eta_t^\varepsilon(\Phi).
\end{align*}
It follows that,
\begin{align*}
&(1-\gamma^\varepsilon(t))  \left( \frac{\int_{\mathbb{S}^{2}}  \int_{\mathbb{R}^3}   g_t(x(t)+\varepsilon\nu,\bar{v})[(v(\tau)-\bar{v})\cdot \nu]_+ \, \mathrm{d}\bar{v} \, \mathrm{d}\nu - \hat{C}(\varepsilon)}{\int_{ U \times \mathbb{R}^3}  g_0(\bar{x},\bar{v}) \mathbbm{1}_t^\varepsilon[\Phi] (\bar{x},\bar{v}) \, \mathrm{d}\bar{x} \, \mathrm{d}\bar{v}  } \right) \\
 \leq & \frac{\int_{\mathbb{S}^{2}}  \int_{\mathbb{R}^3}   g_t(x(t)+\varepsilon\nu,\bar{v})[(v(\tau)-\bar{v})\cdot \nu]_+ \, \mathrm{d}\bar{v} \, \mathrm{d}\nu }{\int_{ U \times \mathbb{R}^3}  g_0(\bar{x},\bar{v}) \mathbbm{1}_t^\varepsilon[\Phi] (\bar{x},\bar{v}) \, \mathrm{d}\bar{x} \, \mathrm{d}\bar{v}  }  \\\leq& (1+2\eta_t^\varepsilon(\Phi))  \int_{\mathbb{S}^{2}}  \int_{\mathbb{R}^3}   g_t(x(t)+\varepsilon\nu,\bar{v})[(v(\tau)-\bar{v})\cdot \nu]_+ \, \mathrm{d}\bar{v} \, \mathrm{d}\nu.
\end{align*}
This implies
\[ -(1-\gamma^\varepsilon(t))\hat{L}_t^\varepsilon(\Phi) \geq - (1+2\eta_t^\varepsilon(\Phi)) L_t^\varepsilon(\Phi).   \]
Returning to \eqref{na-eq-diffphatp} we now see,
\begin{align*}
\partial_t \Big( \hat{P}_t^\varepsilon(\Phi) - R_t^\varepsilon(\Phi) \Big) & \geq -(1+2\eta_t^\varepsilon(\Phi)) L_t^\varepsilon(\Phi)\hat{P}_t^\varepsilon(\Phi) + L_t^\varepsilon(\Phi)R_t^\varepsilon(\Phi) \\
& = -(1+2\eta_t^\varepsilon(\Phi)) L_t^\varepsilon(\Phi)\Big( \hat{P}_t^\varepsilon(\Phi) - R_t^\varepsilon(\Phi) \Big) - 2\eta_t^\varepsilon(\Phi) L_t^\varepsilon(\Phi)R_t^\varepsilon(\Phi) .
\end{align*}
For fixed $\Phi$ this is a $1d$ differential equation in $t$ and so by the variation of constants formula it follows that,
\begin{align*}
\hat{P}_t^\varepsilon(\Phi) - R_t^\varepsilon(\Phi)  & \geq \exp \left(- \int_\tau^t (1+2\eta_s^\varepsilon(\Phi))L_s^\varepsilon(\Phi) \, \mathrm{d} s  \right) (\hat{P}_\tau^\varepsilon(\Phi) - R_\tau^\varepsilon(\Phi)) \\
& \quad \quad - \int_\tau^t \exp \left(- \int_s^t(1+2\eta_\sigma^\varepsilon(\Phi)) L_\sigma^\varepsilon(\Phi) \,\mathrm{d}\sigma \right) 2\eta_s^\varepsilon(\Phi)L_s^\varepsilon(\Phi)R_s^\varepsilon(\Phi) \, \mathrm{d}s.
\end{align*}
Now from \eqref{na-eq-deff1phi} we see that  $\mathbbm{1}_t^\varepsilon[\Phi]$ is non-increasing in $t$ and therefore $\eta_t^\varepsilon(\Phi)$ is non-decreasing in $t$. Since $L_t^\varepsilon(\Phi)$ is non-negative it follows that
\begin{align} \label{na-eq-phatrdiff1}
\hat{P}_t^\varepsilon(\Phi) - R_t^\varepsilon(\Phi)  & \geq \exp \left( -\int_\tau^t (1+2\eta_s^\varepsilon(\Phi))L_s^\varepsilon(\Phi) \, \mathrm{d} s  \right) (\hat{P}_\tau^\varepsilon(\Phi) - R_\tau^\varepsilon(\Phi))  \nonumber  \\
& \quad \quad - 2\eta_t^\varepsilon(\Phi)\int_\tau^t \exp \left( -\int_s^t(1+2\eta_\sigma^\varepsilon(\Phi)) L_\sigma^\varepsilon(\Phi) \,\mathrm{d}\sigma \right) L_s^\varepsilon(\Phi)R_s^\varepsilon(\Phi)\, \mathrm{d}s  \nonumber \\
& \geq \exp \left( -\int_\tau^t (1+2\eta_s^\varepsilon(\Phi))L_s^\varepsilon(\Phi) \, \mathrm{d} s  \right) (\hat{P}_\tau^\varepsilon(\Phi) - R_\tau^\varepsilon(\Phi))  \nonumber \\
& \quad \quad - 2\eta_t^\varepsilon(\Phi)\int_\tau^t \exp \left(- (1+2\eta_s^\varepsilon(\Phi)) \int_s^t L_\sigma^\varepsilon(\Phi) \,\mathrm{d}\sigma \right) L_s^\varepsilon(\Phi)R_s^\varepsilon(\Phi)  \, \mathrm{d}s.
\end{align}
By definition~\ref{na-deff-Ptphi} we have for $\tau \leq s\leq  t$,
\[ R_t^\varepsilon(\Phi) = \exp \left(- \int_s^t L_\sigma^\varepsilon(\Phi) \, \mathrm{d}\sigma \right) R_s^\varepsilon(\Phi), \]
implying for  $\tau \leq s\leq  t$,
\begin{equation} \label{na-eq-rsepev}
 R_s^\varepsilon(\Phi) = \exp \left( \int_s^t L_\sigma^\varepsilon(\Phi) \, \mathrm{d}\sigma \right) R_t^\varepsilon(\Phi).
\end{equation}

Substituting this into \eqref{na-eq-phatrdiff1} we have,
\begin{align*}
&\hat{P}_t^\varepsilon(\Phi) - R_t^\varepsilon(\Phi) \\  \geq &\exp \left(- \int_\tau^t (1+2\eta_s^\varepsilon(\Phi))L_s^\varepsilon(\Phi) \, \mathrm{d} s  \right) (\hat{P}_\tau^\varepsilon(\Phi) - R_\tau^\varepsilon(\Phi))  \nonumber \\
& \quad \quad - 2\eta_t^\varepsilon(\Phi)\int_\tau^t \exp \left(- (1+2\eta_s^\varepsilon(\Phi)) \int_s^t L_\sigma^\varepsilon(\Phi) \,\mathrm{d}\sigma \right) L_s^\varepsilon(\Phi)  \exp \left(\int_s^t L_\sigma^\varepsilon(\Phi) \, \mathrm{d}\sigma \right) R_t^\varepsilon(\Phi)  \, \mathrm{d}s \\
=&  \exp \left(- \int_\tau^t (1+2\eta_s^\varepsilon(\Phi))L_s^\varepsilon(\Phi) \, \mathrm{d} s  \right) (\hat{P}_\tau^\varepsilon(\Phi) - R_\tau^\varepsilon(\Phi))  \nonumber \\
& \quad \quad - 2\eta_t^\varepsilon(\Phi)R_t^\varepsilon(\Phi) \int_\tau^t \exp \left(- 2\eta_s^\varepsilon(\Phi) \int_s^t L_\sigma^\varepsilon(\Phi) \,\mathrm{d}\sigma \right) L_s^\varepsilon(\Phi) \, \mathrm{d}s,
\end{align*}
as required.
\end{proof}
\begin{lem}\label{na-lem-convhelp}\hfill
\begin{enumerate}
\item \label{na-item-convhelp1} For $\Phi\in \mathcal{G}(\varepsilon)$ and $t\geq \tau, $
\[   2 \eta_t^\varepsilon(\Phi) \int_\tau^t \exp \left(  - (1+2\eta_s^\varepsilon(\Phi)) \int_s^t L_\sigma^\varepsilon(\Phi) \, \mathrm{d}\sigma \right) L_s^\varepsilon(\Phi) \, \mathrm{d}s \leq  \rho_t^{\varepsilon,0}(\Phi).   \]
\item \label{na-item-convhelp2} For $\varepsilon$ sufficiently small and  almost all $\Phi \in \mathcal{G}(\varepsilon)$ and any $t\in[0,T]$,
\[ 1 - \frac{1-\gamma^\varepsilon(t)}{\int_{ U \times \mathbb{R}^3}  g_0(\bar{x},\bar{v}) \mathbbm{1}_\tau^\varepsilon[\Phi] (\bar{x},\bar{v}) \, \mathrm{d}\bar{x} \, \mathrm{d}\bar{v}} \leq \varepsilon. \]
\end{enumerate}
\end{lem}
\begin{proof}
Let $\Phi \in \mathcal{G}(\varepsilon)$ and $t\geq \tau$. To prove \eqref{na-item-convhelp1} note that we need to prove,
\[ 2 \int_\tau^t \exp \left(   -(1+2\eta_s^\varepsilon(\Phi)) \int_s^t L_\sigma^\varepsilon(\Phi) \, \mathrm{d}\sigma \right) L_s^\varepsilon(\Phi) \, \mathrm{d}s \leq  C^\varepsilon(\Phi)t. \]
Firstly by definition for any $s\geq \tau$, $L_s^\varepsilon(\Phi) \leq C^\varepsilon(\Phi)/2$. Secondly since $L_\sigma^\varepsilon\geq 0$,
\[ \exp \left(-   (1+2\eta_s^\varepsilon(\Phi)) \int_s^t L_\sigma^\varepsilon(\Phi) \, \mathrm{d}\sigma \right) \leq 1. \]
Hence,
\begin{align*}
2 \int_\tau^t & \exp \left( -  (1+2\eta_s^\varepsilon(\Phi)) \int_s^t L_\sigma^\varepsilon(\Phi) \, \mathrm{d}\sigma \right) L_s^\varepsilon(\Phi) \, \mathrm{d}s  \leq  C^\varepsilon(\Phi) \int_\tau^t   \, \mathrm{d}s =  C^\varepsilon(\Phi)(t-\tau) \leq  C^\varepsilon(\Phi)t.
\end{align*}
We now prove \eqref{na-item-convhelp2}. Repeating the argument of \eqref{na-eq-intgominus} we have,
\[ \int_{ U \times \mathbb{R}^3}   g_0(\bar{x},\bar{v})(1- \mathbbm{1}_\tau^\varepsilon[\Phi] (\bar{x},\bar{v})) \, \mathrm{d}\bar{x} \, \mathrm{d}\bar{v} \leq  \pi \varepsilon^2 T M_g ( V(\varepsilon) + 1),  \]
which converges to zero as $\varepsilon$ converges to zero by \eqref{na-eq-VMass}. Hence,
\[ \int_{ U \times \mathbb{R}^3}   g_0(\bar{x},\bar{v}) \mathbbm{1}_\tau^\varepsilon[\Phi] (\bar{x},\bar{v}) \, \mathrm{d}\bar{x} \, \mathrm{d}\bar{v} = 1- \int_{ U \times \mathbb{R}^3}   g_0(\bar{x},\bar{v})(1- \mathbbm{1}_\tau^\varepsilon[\Phi] (\bar{x},\bar{v})) \, \mathrm{d}\bar{x} \, \mathrm{d}\bar{v},  \]
converges to one as $\varepsilon$ converges to zero. Now,
\begin{align*}
\frac{1}{\varepsilon} & \left( 1-\frac{1-\gamma^\varepsilon(t)}{\int_{ U \times \mathbb{R}^3}   g_0(\bar{x},\bar{v}) \mathbbm{1}_\tau^\varepsilon[\Phi] (\bar{x},\bar{v}) \, \mathrm{d}\bar{x} \, \mathrm{d}\bar{v}} \right) \\
& = \frac{1}{\varepsilon}  \left( \frac{1 - \int_{ U \times \mathbb{R}^3}   g_0(\bar{x},\bar{v}) (1- \mathbbm{1}_\tau^\varepsilon[\Phi] (\bar{x},\bar{v})) \, \mathrm{d}\bar{x} \, \mathrm{d}\bar{v}}{\int_{ U \times \mathbb{R}^3}   g_0(\bar{x},\bar{v}) \mathbbm{1}_\tau^\varepsilon[\Phi] (\bar{x},\bar{v}) \, \mathrm{d}\bar{x} \, \mathrm{d}\bar{v}}-\frac{1-\gamma^\varepsilon(t)}{\int_{ U \times \mathbb{R}^3}   g_0(\bar{x},\bar{v}) \mathbbm{1}_\tau^\varepsilon[\Phi] (\bar{x},\bar{v}) \, \mathrm{d}\bar{x} \, \mathrm{d}\bar{v}} \right)  \\
& \leq \frac{1}{\varepsilon} \left( \frac{\gamma^\varepsilon(t)}{\int_{ U \times \mathbb{R}^3}   g_0(\bar{x},\bar{v}) \mathbbm{1}_\tau^\varepsilon[\Phi] (\bar{x},\bar{v}) \, \mathrm{d}\bar{x} \, \mathrm{d}\bar{v}} \right) \\
& \leq \frac{\varepsilon n(\Phi)}{\int_{ U \times \mathbb{R}^3}   g_0(\bar{x},\bar{v}) \mathbbm{1}_\tau^\varepsilon[\Phi] (\bar{x},\bar{v}) \, \mathrm{d}\bar{x} \, \mathrm{d}\bar{v}} \leq \frac{\varepsilon M(\varepsilon)}{\int_{ U \times \mathbb{R}^3}   g_0(\bar{x},\bar{v}) \mathbbm{1}_\tau^\varepsilon[\Phi] (\bar{x},\bar{v}) \, \mathrm{d}\bar{x} \, \mathrm{d}\bar{v}}.
\end{align*}
By \eqref{na-eq-VMass} the numerator converges to zero as $\varepsilon$ converges to zero and the denominator converges to one, hence for $\varepsilon$ sufficiently small the expression is less than one, proving the required result.

\end{proof}

\begin{proof}[Proof of proposition~\ref{na-prop-phatrt}]
Let $\varepsilon$ sufficiently small and $\Phi \in \mathcal{G}(\varepsilon)$ be such that lemma~\ref{na-lem-convhelp} \eqref{na-item-convhelp2} holds, which excludes only a set of measure zero. We prove by induction on the degree of $\Phi \in \mathcal{G}(\varepsilon)$. Let $\Phi \in \mathcal{T}_0 \cap \mathcal{G}(\varepsilon)$. Then $\tau =0$ so by theorem~\ref{na-thm-id} and theorem~\ref{na-thm-emp},
\[ \hat{P}_0^\varepsilon(\Phi) = \zeta^\varepsilon(\Phi)f_0(x_0,v_0) = \zeta^\varepsilon(\Phi) P_0^\varepsilon(\Phi) = R_t^\varepsilon(\Phi). \]
Hence by lemma~\ref{na-lem-phatrdiff} and lemma~\ref{na-lem-convhelp} \eqref{na-item-convhelp1} for $t\geq 0$,
\begin{align*}
\hat{P}_t^\varepsilon(\Phi) - R_t^\varepsilon(\Phi) &  \geq - 2\eta_t^\varepsilon(\Phi)R_t^\varepsilon(\Phi) \int_\tau^t \exp \left(- 2\eta_s^\varepsilon(\Phi) \int_s^t L_\sigma^\varepsilon(\Phi) \,\mathrm{d}\sigma \right) L_s^\varepsilon(\Phi) \, \mathrm{d}s \\
& \geq -\rho_t^{\varepsilon,0}(\Phi)R_t^\varepsilon(\Phi)   = -\hat{\rho}_t^{\varepsilon}(\Phi)R_t^\varepsilon(\Phi)  .
\end{align*}
This proves the proposition in the base case. Now suppose that the proposition holds for all trees in $\mathcal{T}_{j-1} \cap \mathcal{G}(\varepsilon)$ for some $j\geq 1$ and let $\Phi \in \mathcal{T}_j \cap \mathcal{G}(\varepsilon)$. For $t < \tau$ the proposition holds trivially since the left hand side is $0$. Consider $t \geq \tau$. By theorem~\ref{na-thm-id} and theorem~\ref{na-thm-emp} we have,
\[ \hat{P}_\tau^\varepsilon(\Phi) = \left( \frac{1-\gamma^\varepsilon(t)}{\int_{ U \times \mathbb{R}^3}   g_0(\bar{x},\bar{v}) \mathbbm{1}_\tau^\varepsilon[\Phi] (\bar{x},\bar{v}) \, \mathrm{d}\bar{x} \, \mathrm{d}\bar{v}} \right)  \hat{P}_\tau^\varepsilon(\bar{\Phi})g_\tau(x(\tau)+\varepsilon\nu,v')[(v(\tau^-)-v')\cdot \nu]_+,   \]
and
\[ R_\tau^\varepsilon(\Phi) = R_\tau^\varepsilon(\bar{\Phi})g_\tau(x(\tau)+\varepsilon\nu,v')[(v(\tau^-)-v')\cdot \nu]_+. \]
Since $\bar{\Phi} \in \mathcal{T}_{j-1}\cap \mathcal{G}(\varepsilon)$ by the inductive assumption we have that, for $\tau=\tau(\Phi)$,
\[ \hat{P}_\tau^\varepsilon(\bar{\Phi}) \geq R_\tau^\varepsilon(\bar{\Phi})-\hat{\rho}_\tau^\varepsilon(\bar{\Phi})R_\tau^\varepsilon(\bar{\Phi}).  \]
Hence by lemma~\ref{na-lem-convhelp} \eqref{na-item-convhelp2} for $\varepsilon$ sufficiently small,
\begin{align} \label{na-eq-phatrdifftau}
\hat{P}_\tau^\varepsilon(\Phi) - R_\tau^\varepsilon(\Phi) & =  g_\tau(x(\tau)+\varepsilon\nu,v')[(v(\tau^-)-v')\cdot \nu]_+ \nonumber \\
& \qquad \left( \frac{1-\gamma^\varepsilon(t)}{\int_{ U \times \mathbb{R}^3}   g_0(\bar{x},\bar{v}) \mathbbm{1}_\tau^\varepsilon[\Phi] (\bar{x},\bar{v}) \, \mathrm{d}\bar{x} \, \mathrm{d}\bar{v}}   \hat{P}_\tau^\varepsilon(\bar{\Phi}) - R_\tau^\varepsilon(\bar{\Phi}) \right) \nonumber \\
& \geq g_\tau(x(\tau)+\varepsilon\nu,v')[(v(\tau^-)-v')\cdot \nu]_+ \left( (1-\varepsilon)  \hat{P}_\tau^\varepsilon(\bar{\Phi}) - R_\tau^\varepsilon(\bar{\Phi}) \right)  \nonumber  \\
& \geq g_\tau(x(\tau)+\varepsilon\nu,v')[(v(\tau^-)-v')\cdot \nu]_+  \left( (1-\varepsilon) \left(  R_\tau^\varepsilon(\bar{\Phi})-\hat{\rho}_\tau^\varepsilon(\bar{\Phi})R_\tau^\varepsilon(\bar{\Phi}) \right) - R_\tau^\varepsilon(\bar{\Phi}) \right)  \nonumber  \\
& =  g_\tau(x(\tau)+\varepsilon\nu,v')[(v(\tau^-)-v')\cdot \nu]_+ R_\tau^\varepsilon(\bar{\Phi}) \left( (1-\varepsilon) \left( 1-\hat{\rho}_\tau^\varepsilon(\bar{\Phi}) \right) - 1 \right)  \nonumber  \\
& = R_\tau^\varepsilon(\Phi) \left( -\varepsilon -(1-\varepsilon)\hat{\rho}_\tau^\varepsilon(\bar{\Phi}) \right).
\end{align}
Now the trajectory of the root particle up to time $\tau$ is identical for $\Phi$ and $\bar{\Phi}$ and recalling that $\eta_t^\varepsilon(\Phi)$ is non-decreasing with $t$ it follows,
\[ \eta_\tau^\varepsilon(\bar{\Phi}) = \eta_\tau^\varepsilon(\Phi) \leq \eta_t^\varepsilon(\Phi). \]
Further by \eqref{na-eq-cdeff} it follows that $C^\varepsilon(\bar{\Phi}) \leq C^\varepsilon(\Phi)$. These imply that,
\[ \hat{\rho}_\tau^\varepsilon(\bar{\Phi}) = \rho_\tau^{\varepsilon,j-1}(\bar{\Phi}) \leq  \rho_\tau^{\varepsilon,j-1}(\Phi).  \]
Hence \eqref{na-eq-phatrdifftau} becomes,
\begin{align*}
\hat{P}_\tau^\varepsilon(\Phi) - R_\tau^\varepsilon(\Phi) & \geq - R_\tau^\varepsilon(\Phi) \left( \varepsilon +(1-\varepsilon)\hat{\rho}_\tau^\varepsilon(\bar{\Phi}) \right)  \geq - R_\tau^\varepsilon(\Phi) \left( \varepsilon +(1-\varepsilon)\rho_\tau^{\varepsilon,j-1}(\Phi) \right).
\end{align*}
Using \eqref{na-eq-rsepev} and that $L_s^\varepsilon(\Phi)$ is non-negative, this gives that,
\begin{align*}
\exp  &  \left( \int_\tau^t  -(1+2\eta_s^\varepsilon(\Phi))L_s^\varepsilon(\Phi) \, \mathrm{d} s \right)   (\hat{P}_\tau^\varepsilon(\Phi) - R_\tau^\varepsilon(\Phi)) \\
& \geq - \exp  \left( -\int_\tau^t (1+2\eta_s^\varepsilon(\Phi))L_s^\varepsilon(\Phi) \, \mathrm{d} s \right)  R_\tau^\varepsilon(\Phi) ( \varepsilon +(1-\varepsilon)\rho_\tau^{\varepsilon,j-1}(\Phi) ) \\
& \geq  -  R_t^\varepsilon(\Phi) \exp  \left(- \int_\tau^t 2\eta_s^\varepsilon(\Phi)L_s^\varepsilon(\Phi) \, \mathrm{d} s \right)  ( \varepsilon +(1-\varepsilon)\rho_\tau^{\varepsilon,j-1}(\Phi) ) \geq -R_t^\varepsilon(\Phi)   ( \varepsilon +(1-\varepsilon)\rho_\tau^{\varepsilon,j-1}(\Phi) ).
\end{align*}
Finally we use lemma~\ref{na-lem-phatrdiff}, lemma~\ref{na-lem-convhelp} \eqref{na-item-convhelp1} and \eqref{na-eq-rhokdef} to see that,
\begin{align*}
\hat{P}_t^\varepsilon(\Phi) - R_t^\varepsilon(\Phi) & \geq  \exp \left(- \int_\tau^t (1+2\eta_s^\varepsilon(\Phi))L_s^\varepsilon(\Phi) \, \mathrm{d} s  \right) (\hat{P}_\tau^\varepsilon(\Phi) - R_\tau^\varepsilon(\Phi)) \\
& \, - 2 \eta_t^\varepsilon(\Phi) R_t^\varepsilon (\Phi) \int_\tau^t \exp \left( -  2\eta_s^\varepsilon(\Phi) \int_s^t L_\sigma^\varepsilon(\Phi) \, \mathrm{d}\sigma \right) L_s^\varepsilon(\Phi) \, \mathrm{d}s \\
& \geq - R_t^\varepsilon(\Phi)   ( \varepsilon +(1-\varepsilon)\rho_\tau^{\varepsilon,j-1}(\Phi) ) -  \rho_t^{\varepsilon,0}(\Phi)R_t^\varepsilon(\Phi) \\
& =    - R_t^\varepsilon(\Phi) \Big( \varepsilon +(1-\varepsilon)\rho_\tau^{\varepsilon,j-1}(\Phi) ) + \rho_t^{\varepsilon,0}(\Phi) \Big)
= - R_t^\varepsilon(\Phi) \rho_t^{\varepsilon,j}(\Phi) = - R_t^\varepsilon(\Phi) \hat{\rho}_t^{\varepsilon}(\Phi).
\end{align*}
This proves the inductive step and so completes the proof of the proposition.
\end{proof}
\subsection{Convergence between $P_t^0$ and $\hat{P}_t^\varepsilon$ and the proof of theorem~\ref{na-thm-main}}
\begin{lem} \label{na-lem-rhobound}
For any $\delta >0$ there exists an $\varepsilon'>0$ such that for any $0<\varepsilon<\varepsilon'$, any $t\in[0,T]$ and almost all $\Phi \in \mathcal{G}(\varepsilon)$,
\[ \hat{\rho}_t^\varepsilon(\Phi) <\delta. \]
\end{lem}
\begin{proof}
Fix $\delta >0$. By \eqref{na-eq-intgominus} we have for $\Phi \in \mathcal{G}(\varepsilon)$,
\begin{align*}
 \eta_t^\varepsilon(\Phi) & = \int_{ U \times \mathbb{R}^3}   g_0(\bar{x},\bar{v})(1- \mathbbm{1}_t^\varepsilon[\Phi] (\bar{x},\bar{v})) \, \mathrm{d}\bar{x} \, \mathrm{d}\bar{v}  \leq  \pi \varepsilon^2 T M_g ( V(\varepsilon) + 1)  =: C_1T\varepsilon^2(1+V(\varepsilon)).
\end{align*}
Secondly we note that for almost all $\Phi \in \mathcal{G}(\varepsilon)$ and any $t\in[0,T]$,
\begin{align*}
\int_{\mathbb{S}^{2}} \int_{\mathbb{R}^3} & g_t(x(t)+\varepsilon\nu,\bar{v})[(v(t)-\bar{v})\cdot \nu ]_+ \leq \int_{\mathbb{S}^{2}} \int_{\mathbb{R}^3} \bar{g}(\bar{v})(|v(t)|+|\bar{v})|) \, \mathrm{d}\bar{v} \, \mathrm{d}\nu  \\
& \leq \pi \int_{\mathbb{R}^3} \bar{g}(\bar{v})(V(\varepsilon)+|\bar{v})|) \, \mathrm{d}\bar{v} \leq \pi M_g (1+V(\varepsilon)) =:C_2(1+V(\varepsilon)).
\end{align*}
Hence,
\[ C^\varepsilon(\Phi) =  2 \sup_{t \in [0,T]} \left\{ \int_{\mathbb{S}^{2}} \int_{\mathbb{R}^3} g_t(x(t)+\varepsilon\nu,\bar{v})[(v(t)-\bar{v})\cdot \nu ]_+ \right\} \leq 2C_2 (1+V(\varepsilon)).  \]
This implies that,
\[ \rho_t^{\varepsilon,0}(\Phi)= \eta_t^\varepsilon(\Phi)C^\varepsilon(\Phi)t \leq 2C_1C_2T\varepsilon^2(1+V(\varepsilon))^2. \]
Using \eqref{na-eq-VMass} there exists an $\varepsilon_1>0$ such that for $\varepsilon < \varepsilon_1$ we have,
\[ \rho_t^{\varepsilon,0}(\Phi) < \delta/3. \]
Further by \eqref{na-eq-VMass} there exists $\varepsilon_2>0$ such that for $\varepsilon < \varepsilon_2$,
\[ \rho_t^{\varepsilon,0}(\Phi)M(\varepsilon) \leq 2C_1C_2T\varepsilon^2(1+V(\varepsilon))^2M(\varepsilon) < \delta/3.   \]
And again by \eqref{na-eq-VMass} there exists $\varepsilon_3>0$ such that for $\varepsilon < \varepsilon_3$,
\[ \varepsilon M(\varepsilon)  < \delta/3. \]
Take $\varepsilon' = \min \{ \varepsilon_1,\varepsilon_2,\varepsilon_3 ,1 \}$. Then for any $0< \varepsilon < \varepsilon'$ and for almost all $\Phi \in \mathcal{G}(\varepsilon)$ we have by \eqref{na-eq-rhoalt},
\begin{align*}
\hat{\rho}_t^\varepsilon(\Phi) & = \rho_t^{n(\Phi),\varepsilon}(\Phi) = (1-\varepsilon)^{n(\Phi)} \rho^{\varepsilon,0}_t (\Phi) + (\rho^{\varepsilon,0}_t (\Phi) + \varepsilon) \sum_{j=1}^{n(\Phi)}(1-\varepsilon)^{n(\Phi)-j} \\
& \leq \rho^{\varepsilon,0}_t (\Phi) + (\rho^{\varepsilon,0}_t (\Phi) + \varepsilon) \times n(\Phi) \leq \rho^{\varepsilon,0}_t (\Phi) + \rho^{\varepsilon,0}_t (\Phi)M(\varepsilon) + \varepsilon M(\varepsilon)  < \delta.
\end{align*}
Proving the required result.
\end{proof}

\begin{prop} \label{na-prop-goodfull}
Uniformly for $t \in [0,T]$,
\[ \lim_{\varepsilon \to 0} \int_{\mathcal{MT}\setminus \mathcal{G}(\varepsilon)} P_t^0(\Phi)  \, \mathrm{d}\Phi  =0. \]
\end{prop}
\begin{proof}
We first show that
\begin{equation} \label{na-eq-ptG0}
\int_{\mathcal{MT}\setminus \mathcal{G}(0)} P_t^0(\Phi)  \, \mathrm{d}\Phi  =0.
\end{equation}
To this aim note that $\mathcal{T}_0 \setminus R(0)$ is empty since trees with zero collisions cannot include a re-collision. Let $\Phi \in \mathcal{T}_1 \setminus R(0)$ and denote $\Phi =((x_0,v_0),(\tau,\nu,v'))$. Then the initial position and velocity of the background particle is $(x_0+\tau (v_0 - v'),v')$. Denote the velocity of the background particle after the collision by $\bar{v}$. Then $v(\tau) = v_0 - \nu (v_0 - v')\cdot \nu$ and $\bar{v}= v' + \nu (v_0 - v') \cdot \nu$. Note that this gives
\[ (\bar{v} - v(\tau))\cdot \nu = (v' - v_0 )\cdot \nu + 2 (v_0-v')\cdot \nu = (v_0-v')\cdot \nu.  \]

Since $\Phi \in \mathcal{T}_1 \setminus R(0)$ the tagged particle sees the background particle again at some time $s \in (\tau,T]$. Hence at that $s$ there exists an $m \in \mathbb{Z}^3$ such that,
\[ x(s) +m = x_0 + \tau v_0 + (s-\tau)v(\tau)  + m = x_0+\tau (v_0 - v') + \tau v' + (s-\tau)\bar{v},   \]
which gives
\[ (s-\tau)v(\tau) + m = (s-\tau)\bar{v}. \]
Hence
\[ \frac{m}{s-\tau} = \bar{v} - v(\tau). \]
This implies
\[ \frac{m \cdot \nu}{s- \tau} = (v_0-v')\cdot \nu.  \]
Hence if we consider $v_0$, $\nu$ and $v'$ fixed, then $\tau$ must be in a countable set hence $\mathcal{T}_1 \setminus R(0)$ is a set of zero measure.

Now let $j \geq 2$ and consider $\Phi \in \mathcal{T}_j \setminus R(0)$. Then either two of the collisions in $\Phi$ are with the same background particle, or the tagged particle will collide with one of the background particles again for some time $s \in (\tau,T]$. Let $\Phi = ((x_0,v_0),(t_1,\nu_1,v_1),\dots,(t_j,\nu_j,v_j))$. If we are in the first case there exists an $l \leq j$ and a $k < l$ such that the $k$th and $l$th collision are with the same background particle. Hence,
\[ v_l = v_k + \nu_l (v(t_l^-) - v_k) \cdot \nu_l, \]
Thus $v_l$ is determined by $v_k$, $\nu_l$ and $v(t_l^-)$, so $v_l$ can only be in a set of zero measure. In the second case, there exists a $ 1 \leq k \leq n$, an $s \in (\tau,T]$ and a $m \in \mathbb{Z}^3$ such that
\begin{equation} \label{na-eq-xsm}
x(s) + m = x_k(s).
\end{equation}
We prove that this implies that $t_k$ is in a set of zero measure. Note,
\[ x(s) = x_0 + t_1v_0 + (t_2 - t_1)v(t_1) + \dots + (t_j-t_{j-1})v(t_{j-1}) + (s-t_j)v(t_j). \]
And if we denote the velocity of background particle $k$ after its collision at $t_k$ as $\bar{v}$,
\begin{align*}
x_k(s) & = x_k(t_k) +(s-t_k )\bar{v}  = x(t_k)  + (s-t_k )\bar{v} \\
& = x_0 + t_1v_0 + (t_2 - t_1)v(t_1) + \dots + (t_k-t_{k-1})v(t_{k-1})  + (s-k_j )\bar{v}.
\end{align*}
Then \eqref{na-eq-xsm} gives,
\begin{align*}
 m + (t_{k+1} - t_k)v(t_k) + \dots + (t_j-t_{j-1})v(t_{j-1}) + (s-t_j)v(t_j) = (s-t_k )\bar{v}.
\end{align*}
Rearranging and taking the dot product with $\nu_k$ gives that,
\begin{align} \label{na-eq-tjvj}
t_k (v(t_k)  - \bar{v})\cdot \nu_k  = m\cdot \nu_k +\bigg( t_{k+1} v(t_k) + \dots + (t_j-t_{j-1})v(t_{j-1}) + (s-t_j)v(t_j) - s\bar{v} \bigg) \cdot \nu_k.
\end{align}
Since
\[  (v(t_k) - \bar{v})\cdot \nu_k  = (v_k - v(t_{k-1})) \cdot \nu_k \neq 0, \]
and $v(t_k)$ does not depend on $t_k$ (in the sense that $v(t_k)$ is the same for any $t_k \in (t_{k-1}, t_{k+1})$) it follows that $t_k \in (t_{k-1}, t_{k+1})$ must be in the countable set defined by \eqref{na-eq-tjvj}. Therefore $\mathcal{T}_j \setminus R(0)$ is a set of zero measure. Since $\mathcal{MT} = \cup_{j \geq 0} \mathcal{T}_j$ it follows that,
\[ \int_{\mathcal{MT}\setminus R(0)} P_t^0(\Phi)  \, \mathrm{d}\Phi  =0. \]
The other conditions on $\mathcal{G}(0)$ are clear so \eqref{na-eq-ptG0} holds. Now $\mathcal{G}(\varepsilon)$ is increasing as $\varepsilon$ decreases so for any $\Phi \in \mathcal{MT}$,
\[ \lim_{\varepsilon \to 0} \mathbbm{1}\{ \Phi \in \mathcal{G}(\varepsilon) \} = \mathbbm{1}\{ \Phi \in \mathcal{G}(0) \} \leq 1.   \]
By the dominated convergence theorem, since $P_t^0$ is a probability measure,
\begin{align*}
\lim_{\varepsilon \to 0} \int_{\mathcal{MT}\setminus \mathcal{G}(\varepsilon)} P_t^0(\Phi)  \, \mathrm{d}\Phi & = \lim_{\varepsilon \to 0} \int_{\mathcal{MT}} P_t^0(\Phi) \mathbbm{1}\{ \Phi \in \mathcal{G}(\varepsilon) \}   \, \mathrm{d}\Phi  =  \int_{\mathcal{MT}} P_t^0(\Phi) \mathbbm{1}\{ \Phi \in \mathcal{G}(0) \}   \, \mathrm{d}\Phi =0.
\end{align*}

\end{proof}
We can now prove the convergence between $P_t^0$ and $\hat{P}_t^\varepsilon$, which will then be used to prove theorem~\ref{na-thm-main}.
\begin{thm}\label{na-thm-ptphatcomp}
Uniformly for $t \in [0,T]$,
\[ \lim_{\varepsilon \to 0} \sup_{S \subset \mathcal{MT}} \left| \int_{S} P_t^0(\Phi) - \hat{P}_t^\varepsilon(\Phi) \, \mathrm{d}\Phi \right|  =0.\]
\end{thm}
\begin{proof}
Let $\delta >0$ and $S \subset \mathcal{MT}$. By proposition~\ref{na-prop-goodfull}, for $\varepsilon$ sufficiently small,
\[ \int_{S \setminus \mathcal{G}(\varepsilon) } P_t^0(\Phi)  \, \mathrm{d}\Phi  \leq \int_{\mathcal{MT} \setminus \mathcal{G}(\varepsilon) } P_t^0(\Phi)  \, \mathrm{d}\Phi < \frac{\delta}{4}.\]

By theorem~\ref{na-thm-id} for $\varepsilon$ sufficiently small,
\[ \int_{S \cap \mathcal{G}(\varepsilon)} | P_t^0(\Phi) - P_t^\varepsilon(\Phi) |  \, \mathrm{d}\Phi \leq \int_{ \mathcal{MT}} | P_t^0(\Phi) - P_t^\varepsilon(\Phi) |  \, \mathrm{d}\Phi  <  \frac{\delta}{4}. \]

Hence,
\begin{align} \label{na-eq-pt0sdiff}
\int_{S} P_t^0(\Phi) - \hat{P}_t^\varepsilon(\Phi) \, \mathrm{d}\Phi &= \int_{S \cap \mathcal{G}(\varepsilon)} P_t^0(\Phi) - \hat{P}_t^\varepsilon(\Phi) \, \mathrm{d}\Phi + \int_{S \setminus \mathcal{G}(\varepsilon)} P_t^0(\Phi) - \hat{P}_t^\varepsilon(\Phi) \, \mathrm{d}\Phi \nonumber \\
& < \int_{S  \cap \mathcal{G}(\varepsilon)} P_t^0(\Phi) - \hat{P}_t^\varepsilon(\Phi) \, \mathrm{d}\Phi +\frac{\delta}{4} \nonumber \\
& =  \int_{S \cap \mathcal{G}(\varepsilon)} P_t^0(\Phi) - P_t^\varepsilon(\Phi) \, \mathrm{d}\Phi \int_{S \cap \mathcal{G}(\varepsilon)} P_t^\varepsilon(\Phi) - \hat{P}_t^\varepsilon(\Phi) \, \mathrm{d}\Phi
+\frac{\delta}{4}\nonumber \\
& <  \frac{\delta}{2} + \int_{S \cap \mathcal{G}(\varepsilon)} P_t^\varepsilon(\Phi) - \hat{P}_t^\varepsilon(\Phi) \, \mathrm{d}\Phi.
\end{align}
Now by the definition of $\zeta^\varepsilon(\Phi)$ \eqref{na-eq-zetadeff} we see that since $g_0$ is a probability measure $\zeta^\varepsilon(\Phi) \leq 1$. Hence by lemma~\ref{na-lem-rhobound} for $\varepsilon$ sufficiently small and almost all $\Phi \in \mathcal{G}(\varepsilon)$,
\begin{equation} \label{na-eq-zetarhobound}
\zeta^\varepsilon(\Phi) \hat{\rho}_t^\varepsilon(\Phi) < \frac{\delta}{4}.
\end{equation}
Also by \eqref{na-eq-gl1assmp},
\[ \int_{B_\varepsilon(x_0)} \int_{\mathbb{R}^3} g_0(\bar{x},\bar{v}) \, \mathrm{d}\bar{v} \,\mathrm{d}\bar{x} \leq \int_{B_\varepsilon(x_0)} \int_{\mathbb{R}^3} \bar{g}(\bar{v}) \, \mathrm{d}\bar{v} \,\mathrm{d}\bar{x} \leq \frac{4}{3}M_g\pi \varepsilon^3. \]
Hence, recalling that in the Boltzmann-Grad scaling $N\varepsilon^2=1$, we have by the binomial inequality,
\begin{align*}
\zeta^\varepsilon(\Phi) &  = \left( 1- \int_{B_\varepsilon(x_0)} \int_{\mathbb{R}^3} g_0(\bar{x},\bar{v}) \, \mathrm{d}\bar{v} \,\mathrm{d}\bar{x}\right)^N
 \geq 1- N\int_{B_\varepsilon(x_0)} \int_{\mathbb{R}^3} g_0(\bar{x},\bar{v}) \, \mathrm{d}\bar{v} \,\mathrm{d}\bar{x}  \geq 1-  \frac{4}{3}M_g\pi \varepsilon.
\end{align*}
So for $\varepsilon$ sufficiently small we have
\[ 1- \zeta^\varepsilon(\Phi) <  \frac{\delta}{4}. \]
Hence by proposition~\ref{na-prop-phatrt} and \eqref{na-eq-zetarhobound} we have, for $\varepsilon$ sufficiently small and almost all $\Phi \in \mathcal{G}(\varepsilon)$
\begin{align*}
 P_t^\varepsilon(\Phi) - \hat{P}_t^\varepsilon(\Phi) & \leq  P_t^\varepsilon(\Phi) -R_t^\varepsilon(\Phi) + \hat{\rho}_t^\varepsilon(\Phi) R_t^\varepsilon(\Phi) = P_t^\varepsilon(\Phi) - \zeta^\varepsilon(\Phi)P_t^\varepsilon(\Phi)+ \zeta^\varepsilon(\Phi) \hat{\rho}_t^\varepsilon(\Phi)  P_t^\varepsilon(\Phi) \\
& = (1-\zeta^\varepsilon(\Phi)) P_t^\varepsilon(\Phi) + \zeta^\varepsilon(\Phi)\hat{\rho}_t^\varepsilon(\Phi)  P_t^\varepsilon(\Phi) < \frac{\delta}{4}  P_t^\varepsilon(\Phi)  + \frac{\delta}{4} P_t^\varepsilon(\Phi)  = \frac{\delta}{2} P_t^\varepsilon(\Phi) .
\end{align*}
Hence for $\varepsilon$ sufficiently small,
\[ \int_{S \cap \mathcal{G}(\varepsilon)} P_t^\varepsilon(\Phi) - \hat{P}_t^\varepsilon(\Phi) \, \mathrm{d}\Phi < \frac{\delta}{2}\int_{S \cap \mathcal{G}(\varepsilon)} P_t^\varepsilon(\Phi) \, \mathrm{d}\Phi \leq \frac{\delta}{2}\int_{\mathcal{MT}} P_t^\varepsilon(\Phi)  \, \mathrm{d}\Phi = \frac{\delta}{2}. \]
Substituting this into \eqref{na-eq-pt0sdiff} we see that for $\varepsilon$ sufficiently small,
\begin{equation}\label{na-eq-ptpthat1}
\int_{S} P_t^0(\Phi) - \hat{P}^\varepsilon_t(\Phi) \, \mathrm{d} \Phi < \delta.
\end{equation}
This holds for all $S \subset \mathcal{MT}$ and hence for any $S' \subset \mathcal{MT}$, since $P_t^0$ and $\hat{P}_t^\varepsilon$ are probability measures,
\begin{align*}
\int_{S'} \hat{P}_t^\varepsilon(\Phi)- P_t^0(\Phi)  \, \mathrm{d}\Phi = \int_{\mathcal{MT} \setminus S'} \hat{P}_t^\varepsilon(\Phi)- P_t^0(\Phi)  \, \mathrm{d}\Phi< \delta.
\end{align*}
Together with \eqref{na-eq-ptpthat1} this gives that  for $\varepsilon$ sufficiently small, for any $S \subset \mathcal{MT}$ we have,
\[ \left|  \int_{S} \hat{P}_t^\varepsilon(\Phi)- P_t^0(\Phi)  \, \mathrm{d}\Phi \right| < \delta, \]
which completes the proof of the theorem. \end{proof}

This now allows us to prove the main  theorem~\ref{na-thm-main}.
\begin{proof}[Proof of theorem~\ref{na-thm-main}.]
Let $t \in [0,T]$ and $\Omega \subset U \times \mathbb{R}^3$. By theorem~\ref{na-thm-id},
\[ \int_\Omega f_t^0(x,v) \, \mathrm{d}x \, \mathrm{d}v = \int_{S_t(\Omega)} P_t^0(\Phi) \, \mathrm{d}\Phi.\]
By definition $\hat{P}_t^\varepsilon$ satisfies,
\[ \int_\Omega \hat{f}^N_t(x,v) \, \mathrm{d}x \, \mathrm{d}v = \int_{S_t(\Omega)} \hat{P}_t^\varepsilon(\Phi) \, \mathrm{d}\Phi  .\]
Let $\delta > 0$. By theorem~\ref{na-thm-ptphatcomp}, for $\varepsilon$ sufficiently small (or equivalently by the Boltzmann-Grad scaling, $N\varepsilon^2 =1$, for $N$ sufficiently large) and independent of $t$,
\[\sup_{S \subset \mathcal{MT}} \left| \int_{S} P_t^0(\Phi) - \hat{P}_t^\varepsilon(\Phi) \, \mathrm{d}\Phi \right|  < \delta. \]
Hence,
\begin{align*}
 \bigg| \int_\Omega \hat{f}^N_t(x,v)  & - f_t^0(x,v) \, \mathrm{d}x \, \mathrm{d}v  \bigg|
  = \left| \int_{S_t(\Omega)} P_t^0(\Phi) - \hat{P}_t^\varepsilon(\Phi) \, \mathrm{d}\Phi \right| \leq  \sup_{S \subset \mathcal{MT}} \left| \int_{S} P_t^0(\Phi) - \hat{P}_t^\varepsilon(\Phi) \, \mathrm{d}\Phi \right| < \delta.
\end{align*}
Then using that $\hat f^N_t$ and $f^0_t$ are in $L^1(U \times \mathbb{R}^3)$ by Proposition \ref{na-prop-Pj0} and Lemma \ref{lem-abscts} and the inequality
\[ |\hat f^N_t -f^0_t|_{L^1}=2 \int_{ f^N_t \geq f^0_t} f^N_t -f^0_t \, \mathrm{d}x \, \mathrm{d}v \leq 2|\hat f^N_t -f^0_t|_{TV} \]
we obtain the result.
\end{proof}

\section{Auxiliary Results} \label{na-sec-aux}
\begin{lem} \label{na-lem-g0diffbound}
There exists a $C>0$ such that for any $\varepsilon \geq 0$, $R\geq 1$, $t,s\geq 0$, $ v \in \mathbb{R}^3$ and almost all $x,y \in U$,
\[ \int_{\mathbb{S}^2} \int_{\mathbb{R}^3} |g_t(x+\varepsilon\nu,\bar{v})-g_s(y+\varepsilon\nu,\bar{v})|[(v-\bar{v})\cdot \nu]_+ \, \mathrm{d}\bar{v} \, \mathrm{d}\nu \leq C(1+|v|) \Big(\frac{1}{R} +R^5 \Big(|x-y|^\alpha + |t-s|^\alpha \Big) \Big). \]
\end{lem}
\begin{proof}
Let $\varepsilon \geq 0$. Firstly by \eqref{na-eq-gl1assmp},
\begin{align*}
\int_{\mathbb{R}^3 \setminus B_R(0)} &  \bar{g}(\bar{v})(1+|\bar{v}|) \, \mathrm{d}\bar{v}  \leq \int_{\mathbb{R}^3 \setminus B_R(0)} \bar{g}(\bar{v})(\frac{|\bar{v}|}{R}+\frac{|\bar{v}|^2}{R}) \, \mathrm{d}\bar{v} \leq \frac{1}{R} \int_{\mathbb{R}^3} \bar{g}(\bar{v})(|\bar{v}|+ |\bar{v}|^2) \, \mathrm{d}\bar{v} \leq \frac{2M_g}{R},
\end{align*}
where $B_R(0)$ denotes the ball of radius $R$ around $0$ in $\mathbb{R}^3$. Hence
\begin{align}\label{na-eq-g0diffintbound-part1}
 \int_{\mathbb{S}^2} \int_{\mathbb{R}^3\setminus B_R(0)} & |g_t(x+\varepsilon\nu,\bar{v})-g_s(y+\varepsilon\nu,\bar{v})|[(v-\bar{v})\cdot \nu]_+ \, \mathrm{d}\bar{v} \, \mathrm{d}\nu \nonumber \\
& \leq  \int_{\mathbb{S}^2} \int_{\mathbb{R}^3\setminus B_R(0)} (g_t(x+\varepsilon\nu,\bar{v})+g_s(y+\varepsilon\nu,\bar{v}))(|v|+|\bar{v}|) \, \mathrm{d}\bar{v} \, \mathrm{d}\nu \nonumber \\
 & \leq \int_{\mathbb{S}^2} \int_{\mathbb{R}^3\setminus B_R(0)} 2\bar{g}(\bar{v})(|v|+|\bar{v}|) \, \mathrm{d}\bar{v} \, \mathrm{d}\nu \leq 4\pi \frac{M_g}{R} (1+|v|).
\end{align}
Further by \eqref{na-eq-ghldassmp}, for any $\nu \in \mathbb{S}^2$, for almost all $\bar{v} \in B_R(0)$ and almost all $x,y\in U$ we have, since $0<\alpha \leq 1$,
\begin{align*}
|g_t(x+\varepsilon\nu,\bar{v})-g_s(y+\varepsilon\nu,\bar{v})| & \leq M|x-y-(t-s)\bar{v}|^\alpha  \leq  M\Big(|x-y|^\alpha + |(t-s)\bar{v}|^\alpha \Big)  \\
& \leq M\Big(|x-y|^\alpha + R^\alpha |t-s|^\alpha \Big)  \leq MR \Big(|x-y|^\alpha + |t-s|^\alpha \Big).
\end{align*}
Hence,
\begin{align}\label{na-eq-g0diffintbound-part2}
&\int_{\mathbb{S}^2} \int_{B_R(0)}  |g_t(x+\varepsilon\nu,\bar{v})-g_s(y+\varepsilon\nu,\bar{v})|[(v-\bar{v})\cdot \nu]_+ \, \mathrm{d}\bar{v} \, \mathrm{d}\nu \nonumber \\
 \leq& \int_{\mathbb{S}^2} \int_{B_R(0)}   MR \Big(|x-y|^\alpha + |t-s|^\alpha \Big) (|v|+|\bar{v}|) \, \mathrm{d}\bar{v} \, \mathrm{d}\nu  \nonumber \\&
 \leq \int_{\mathbb{S}^2} \int_{B_R(0)}  MR \Big(|x-y|^\alpha + |t-s|^\alpha \Big) (|v|+R)\, \mathrm{d}\bar{v} \, \mathrm{d}\nu \nonumber \\
\leq&  \frac{4}{3}\pi R^3 \times 2\pi \times  MR \Big(|x-y|^\alpha + |t-s|^\alpha \Big) \times  R(1+|v|)
= \frac{8}{3}\pi^2 M R^5(1+|v|) \Big(|x-y|^\alpha + |t-s|^\alpha \Big).
\end{align}
Together \eqref{na-eq-g0diffintbound-part1} and \eqref{na-eq-g0diffintbound-part2} give,
\begin{align*}
 &\int_{\mathbb{S}^2} \int_{\mathbb{R}^3}  |g_t(x+\varepsilon\nu,\bar{v})-g_s(y+\varepsilon\nu,\bar{v})|[(v-\bar{v})\cdot \nu]_+ \, \mathrm{d}\bar{v} \, \mathrm{d}\nu \\
 =& \int_{\mathbb{S}^2} \int_{B_R(0)}   |g_t(x+\varepsilon\nu,\bar{v})-g_s(y+\varepsilon\nu,\bar{v})|[(v-\bar{v})\cdot \nu]_+ \, \mathrm{d}\bar{v} \, \mathrm{d}\nu \\
& \quad  + \int_{\mathbb{S}^2} \int_{\mathbb{R}^3\setminus B_R(0)}   |g_t(x+\varepsilon\nu,\bar{v})-g_s(y+\varepsilon\nu,\bar{v})|[(v-\bar{v})\cdot \nu]_+ \, \mathrm{d}\bar{v} \, \mathrm{d}\nu \\
 \leq&  4\pi \frac{M_g}{R}(1+|v|) + \frac{8}{3}\pi^2 R^5  M \Big(|x-y|^\alpha + |t-s|^\alpha \Big)(1+|v|)
 \leq C(1+|v|)\Big(\frac{1}{R} +R^5 \Big(|x-y|^\alpha + |t-s|^\alpha \Big) \Big),
\end{align*}
where,
\[ C:= \max \left\{ 4\pi M_g, \frac{8}{3}\pi^2 M \right\}. \]
\end{proof}
\begin{lem} \label{na-lem-LLepcomp}
For $\varepsilon > 0$ sufficiently small,  almost all $\Phi \in \mathcal{MT}$ and any $t \in [0,T]$,
\[ |L_t^0(\Phi) - L_t^\varepsilon(\Phi) | \leq 2C(1+|v(\tau)|)\varepsilon^{\alpha/6}, \]
where $C$ is as in lemma~\ref{na-lem-g0diffbound}.
\end{lem}
\begin{proof}
Let $R=\varepsilon^{-\alpha/6}$ and $\varepsilon$ sufficiently small so that $R\geq 1$. Let $\Phi \in \mathcal{MT}$ be such that for all $ t\in [0,T]$, $\nu \in \mathbb{S}^2$, and almost all $\bar{v} \in \mathbb{R}^3$,
\begin{align*}
 g_t(x(t),\bar{v})+g_t(x(t)+\varepsilon\nu,\bar{v}) & \leq 2\bar{g}(\bar{v}) \textrm{ and, } \\
|g_t(x(t),\bar{v})-g_t(x(t)+\varepsilon\nu,\bar{v})| & \leq M\varepsilon^\alpha.
\end{align*}
Indeed by \eqref{na-eq-gbar} and \eqref{na-eq-ghldassmp} this only excludes a set of zero measure. As in the previous lemma we have,
\begin{align}\label{na-eq-LL-part1}
 \int_{\mathbb{S}^2} \int_{\mathbb{R}^3\setminus B_R(0)} & |g_t(x(t),\bar{v})-g_t(x(t)+\varepsilon\nu,\bar{v})|[(v(\tau)-\bar{v})\cdot \nu]_+ \, \mathrm{d}\bar{v} \, \mathrm{d}\nu \nonumber \\
& \leq  \int_{\mathbb{S}^2} \int_{\mathbb{R}^3\setminus B_R(0)} (g_t(x(t),\bar{v})+g_t(x(t)+\varepsilon\nu,\bar{v}))(|v(\tau)|+|\bar{v}|) \, \mathrm{d}\bar{v} \, \mathrm{d}\nu \nonumber \\
 & \leq \int_{\mathbb{S}^2} \int_{\mathbb{R}^3\setminus B_R(0)} 2\bar{g}(\bar{v})(|v(\tau)|+|\bar{v}|) \, \mathrm{d}\bar{v} \, \mathrm{d}\nu \leq 4\pi \frac{M_g}{R} (1+|v(\tau)|).
\end{align}
And
\begin{align}\label{na-eq-LL-part2}
\int_{\mathbb{S}^2} \int_{B_R(0)} & |g_t(x(t),\bar{v})-g_t(x(t)+\varepsilon\nu,\bar{v})|[(v(\tau)-\bar{v})\cdot \nu]_+ \, \mathrm{d}\bar{v} \, \mathrm{d}\nu \nonumber \\
& \leq \int_{\mathbb{S}^2} \int_{B_R(0)}   M \varepsilon^\alpha (|v(\tau)|+|\bar{v}|) \, \mathrm{d}\bar{v} \, \mathrm{d}\nu
\leq \int_{\mathbb{S}^2} \int_{B_R(0)}  M\varepsilon^\alpha (|v(\tau)|+R)\, \mathrm{d}\bar{v} \, \mathrm{d}\nu \nonumber \\
& \leq \frac{4}{3}\pi R^3 \times 2\pi \times  M\varepsilon^\alpha \times  R(1+|v(\tau)|) \leq \frac{8}{3}\pi^2 M \varepsilon^\alpha R^5(1+|v(\tau)|) .
\end{align}
Combining \eqref{na-eq-LL-part1} and \eqref{na-eq-LL-part2} we have for $C$ as in the previous lemma,
\begin{align*}
|L_t^0(\Phi) - L_t^\varepsilon(\Phi) | & \leq \int_{\mathbb{S}^2} \int_{\mathbb{R}^3}  |g_t(x(t),\bar{v})-g_t(x(t)+\varepsilon\nu,\bar{v})|[(v(\tau)-\bar{v})\cdot \nu]_+ \, \mathrm{d}\bar{v} \, \mathrm{d}\nu \\
& \leq 4\pi \frac{M_g}{R} (1+|v(\tau)|) + \frac{8}{3}\pi^2 M\varepsilon^\alpha R^5(1+|v(\tau)|) \leq C (1+|v(\tau)|)\left( \frac{1}{R} + \varepsilon^\alpha R^5 \right).
\end{align*}
Substituting $R=\varepsilon^{-\alpha/6}$ gives the required result.
\end{proof}

\bibliographystyle{plain}
\bibliography{bib}


\end{document}